%% file: main.tex
\documentclass{article}
\usepackage[utf8]{inputenc}

\usepackage{amssymb} 
\usepackage{amsthm}
\usepackage{amsmath}
\usepackage{mathabx}
\usepackage{mathtools}
\usepackage{algorithm}
\usepackage{algpseudocode}
\usepackage{graphicx}
\usepackage{multirow}
\usepackage{enumerate}   
\usepackage{subcaption}
\usepackage{comment}
\usepackage[dvipsnames,table,xcdraw]{xcolor}
\usepackage{hyperref}
\usepackage[margin=3.5cm]{geometry}
\usepackage{array}
\usepackage{graphicx}
\usepackage{graphbox}
\usepackage[normalem]{ulem}

\newcolumntype{C}[1]{>{\centering\let\newline\\\arraybackslash\hspace{0pt}}m{#1}}

\title{Atomic Cluster Expansion without Self-Interaction}

\author{Cheuk Hin Ho\thanks{Department of Mathematics, University of British Columbia, Vancouver, V6T1Z2, BC, Canada},\and Timon S. Gutleb\footnotemark[1], \and Christoph Ortner\footnotemark[1]}
\date{}

\def\bX{{\bf X}}

\def\bbR{\mathbb{R}}
\def\bbC{\mathbb{C}}
\def\bbF{\mathbb{F}}
\def\bk{{\bf k}} 
\def\bK{{\bf K}} 
\def\calA{\mathcal{A}}
\def\calN{\mathcal{N}}

\def\msl{{\{\hspace{-0.75ex}\{}}
\def\msr{{\}\hspace{-0.75ex}\}}}
\def\mslb{{\big\{\hspace{-0.95ex}\big\{}}
\def\msrb{{\big\}\hspace{-0.95ex}\big\}}}

\def\bAA{{\bf A}}

\begin{document}
\input{notations}

\maketitle
\begin{abstract}
The \textit{Atomic Cluster Expansion} (ACE) (Drautz, Phys. Rev. B 99, 2019) has been widely applied in high energy physics, quantum mechanics and atomistic modeling to construct many-body interaction models respecting physical symmetries. Computational efficiency is achieved by allowing non-physical self-interaction terms in the model. 
We propose and analyze an efficient method to evaluate and parameterize an orthogonal, or, non-self-interacting cluster expansion model. We present numerical experiments demonstrating improved conditioning and more robust approximation properties than the original expansion in regression tasks both in simplified toy problems and in applications in the machine learning of interatomic potentials.
\end{abstract}

\section{Introduction}
Most physical systems of interest in applications exhibit numerous symmetries. When using machine learning to model such systems, preserving those symmetries is an important consideration for model architecture design. Permutation invariance, that is invariance of the model under permutation of input particles, is a typical property of molecular dynamics, computer vision and quantum mechanics applications \cite{thomas2018tensor, qi2017pointnet, lee2019set}. Equivariance with respect to a reductive Lie group $G$ such as the Lorentz group, orthogonal groups or the Euclidean group are also frequently encountered. One common approach in equivariant machine learning is to generate a complete basis that exhibits permutation and $G$-equivariance which spans the target function space. Concretely, we will consider the construction, efficient evaluation and numerical properties of equivariant bases of polynomials.

An important application of equivariant machine learning models is the parameterization of potential energy surfaces \cite{deringer2019machine, musil2021physics}. An early proposal for the parameterization of permutation invariant potential energy surfaces with polynomials can be found in \cite{braams2009permutationally}. Another approach called aPIPs (atomic body-ordered permutation-invariant polynomials) \cite{reg_cas} generalized this idea to much higher order body interactions. In practice, invariant theory and symbolic computation allow for the use of invariant polynomial bases of up to around body order $5$. In 2019 Drautz \cite{DrautzACE} introduced a framework called the Atomic Cluster Expansion (ACE), which falls under this family of methods but allows for the efficient parameterization of much higher body orders than aPIPs. A key computational ingredient of ACE is that the \textit{canonical} many body expansion is transformed into the \textit{self-interacting} expansion by interchanging a summation and a product operation, resulting in symmetric tensor product basis functions which are more efficient to evaluate at the cost of introducing physically undesired terms. The construction also allows explicit symmetrization with respect to the orthogonal group with a sparse linear operator.

Aside from applications in the parameterization of potential energy surfaces and activate learning for molecular dynamics \cite{vanderoord2022HAL}, ACE has also been used for the parameterization of wave functions for the Schr\"odinger equation \cite{drautzwavefunc, ACESchrodinger} and Lorentz group invariant polynomials for jet tagging \cite{Munoz_2022}. In \cite{Batatia2022mace}, a message passing neural network was introduced to attain state-of-the-art accuracy based on the systematic and theoretically sound model architecture motivated from ACE. More recently, \cite{batatia2023general} extended these ideas to a general reductive Lie group $G$. 
The success of the ACE approach motivated the discussion of its theoretical properties from a numerical analysis point of view in \cite{ACECompleteness, bachmayr2023polynomial} as well as our present investigation of the ACE basis and potential improvements to its approximation properties in this paper.

The canonical many body expansion which underpins the derivation and motivation of the self-interacting ACE expansion is generally not considered in practice since the cost of naive evaluation of the corresponding basis functions scales combinatorially with the input dimension and thus results in computationally intractable algorithms. This paper's primary focus is to show that under mild conditions one \emph{can} efficiently evaluate of the canonical cluster expansion and that this is in fact equivalent to a special regularization choice in the widely used self-interacting expansion. Since the canonical expansion has been considered computationally intractable prior to this work (with the exception of a brief comment in \cite{ACECompleteness}) despite being chemically and physically better motivated, the in-practice differences between the two expansions have remained under-explored. Thus, as far as we are aware, our current work also presents the first in-depth exploration of the canonical expansion.

After describing in detail how to efficiently evaluate the canonical expansion in Section~\ref{section:MainResults} we conduct a practical review of its spanning space compared to the self-interacting expansion, supplementing the results of previous work such as \cite{batatia2023general, ACECompleteness}. Importantly, we demonstrate that our framework respects symmetric sparsification for the Euclidean group, which is of major interest in applications. We present numerical experiments for simplified academic examples in Section~\ref{Section:NumericalExp} and applications in the machine learning of interatomic potentials in Section~\ref{section:MLIPs}, showing that the canonical cluster expansion produces models with better physical properties for applications such as molecular dynamics simulations. Additionally, we observe that the canonical expansion appears less sensitive to the regularization parameter in regression tasks which favours hyper-parameter tuning. While we demonstrate several noteworthy qualitative advantages of the canonical basis, we also observe that the two expansions do not appear to show dramatic differences in the sense of the prediction error in regression tasks as demonstrated in Section \ref{Section:NumericalExp} in a practical but relatively simple point cloud setting. We view this as additional evidence that the efficient self-interacting expansion in previous works is well suited to the problems it is used for. On the other hand, when regressing symmetric functions with fixed dimensionality, the {\it canonical basis} is clearly superior in all our tests.

\section{Theory}
\label{section:MainResults}
\subsection{Two cluster expansions of point clouds}
\label{section:preliminaries}
Let $\Omega$ be a configuration domain for particles denoted by $\bx_j$. A particle is a point in $\R^d$ decorated with additional properties such as chemical species, spin, or charge. For example, an atom in a molecular dynamics simulation could be described by position and chemical species, i.e. $\bx_j = ({\bf r}_j, Z_j) \in \mathbb{R}^3 \times \mathbb{Z}$. An electron in a variational Monte-Carlo simulation would be described by position and spin, $\bx_j = ({\bf r}_j, \sigma_j) \in \mathbb{R}^3 \times \{ \uparrow, \downarrow \}$. 
A particle \emph{configuration} (commonly also called a {\em point cloud}) is a multi-set  on $\Omega$,  denoted by $\mathbf{X} = \msl\bx_j\msr_{j=1}^J$, where each $\bx_j \in \Omega$ and $J \in \mathbb{N}$. We denote the set of all multi-sets on $\Omega$ by $\text{MS}(\Omega)$.

Many applications \cite{thomas2018tensor, lee2019set, Li2022DLDFT, qi2017pointnet, wang2023theoretical} are concerned with approximating a multi-set function
\begin{equation}
    V : \text{MS}(\Omega) \rightarrow \bbF,
\end{equation}
where $\bbF \in \{\bbR, \bbC\}$. The restriction to scalar-valued functions is only for the sake of notational simplicity; we explain in Section \ref{section:G-sym} how to extend all our results to tensor-valued functions. Symmetric functions with fixed dimensionality (as oppposed to multi-set functions) are implicitly included in our framework and will be discussed in Section~\ref{section:orthogonalsymbasis}.

The argument $\bX$ being a multi-set implicitly entails that $V$ is invariant under a permutation (or, relabelling) of particles $\bx_j$.
In addition, one often requires also invariance under a Lie-group action, isometry-invariance being the prototypical example. Let $G$ be such a Lie group acting on $\Omega$, e.g. $G = O(d)$. We say that $V$ is $G$-invariant if
\begin{equation}
\label{eq:V_invar}
V \circ g  = V \qquad \forall g \in G;
\end{equation}
see \cite{thomas2018tensor, mlinvariant, lietransformer} and references therein.
The current work focuses on the atomic cluster expansion (ACE) framework which, in its simplest form,  produces a complete linear model for $G$-invariant functions defined on multi-sets. A similar construction can be used if $V$ is $G$-equivariant instead. We now review this framework in more detail, following \cite{batatia2023general, DrautzACE, ACECompleteness, 2023-acepotentials}.

\subsubsection{Canonical cluster expansion}
\label{section:intro_canonical}
The first step of the ACE framework is to perform a many-body expansion,
\begin{equation}
\label{equa:CanonicalExpansion}
    V(\mathbf{X}) \approx V^{\mathcal{N}}(\mathbf{X}) = \sum^\mathcal{N}_{N=0} \sum_{j_1 < \cdots < j_N} v_N(\bx_{j_1}, \cdots, \bx_{j_N}),
\end{equation}
where $\mathcal{N} \in \mathbb{N}_0 := \mathbb{N} \cup \{0\}$ is the maximum {\em correlation order}, and the expansion components $v_N : \Omega^N \to \mathbb{F}$ are permutation invariant on the fixed dimension subspace $\Omega^N$ of $\text{MS}(\Omega)$. $V^{\mathcal{N}}$ is an approximation to $V$ due to both the truncation in correlation order and the parameterization of the functions $v_N$.

The term $v_0()$ without argument is simply a constant. Each $v_N$ is now expanded using a tensor product basis. To that end, let $\{\phi_k\}_{k \in \mathcal{I}}$ be a countable family of linearly independent one-particle basis functions, i.e., $\phi_k : \Omega \to \bbF$.
Note that $k$ does not necessarily have to be an integer but could be a multi-index. 
We assign to each $\phi_{k}$ a degree $\deg{(\phi_{k})} \in \mathbb{R}_{\geq 0}$ and write $\deg{(\phi_{k})} = \deg(k)$ when there is no confusion. 
Throughout this work we assume that there is a natural total ordering of $\phi_k$ that is consistent with $\deg(\phi_k)$.

One can then parameterize each $v_N$ with finitely many $\bk$,
\begin{equation}
    v_N(\bx_1, \cdots, \bx_N) = \sum_{\bk = (k_1, \cdots, k_N)} c_{\bk} \prod^N_{t = 1} \phi_{k_t}(\bx_t).
\end{equation}
Permutation-invariance of $v_N$ is guaranteed if and only if $c_{\bk} = c_{\sigma \bk}$ for all $\sigma \in S_N$. Note that $v_0() = c_{()} \in \bbF$. Combining the above, we obtain {\it new} parameters $c_{k_1, \cdots, k_N}$ such that
\begin{align}
    \label{eq:exp_pure_1}
    V^{\mathcal{N}}(\mathbf{X}) 
    &=
    \sum^\mathcal{N}_{N=0} \sum\limits_{\substack{\bk = (k_1, \cdots, k_N) \\ \bk \text{ ordered }}} {c_\bk} \mathcal{A}_{\bk}({\bf X}),
\end{align}
where ``ordered'' means lexicographic ordering, and the $\mathcal{A}_{\bk}$ basis is defined by 
\begin{align}
    \label{equa:PureBasis}
    \mathcal{A}_{\bk} = \mathcal{A}_{k_1, \cdots, k_N}(\bX) 
    &:= 
    \sum_{j_1 \neq \cdots \neq j_N} \prod^N_{t = 1} \phi_{k_t}(\bx_{j_t}).    
\end{align}
 We call \eqref{eq:exp_pure_1} the {\em canonical cluster expansion}. This is to contrast it with the self-interacting formulation that was introduced in \cite{DrautzACE} to overcome the prohibitive combinatorial scaling of the computational cost of evaluating \eqref{eq:exp_pure_1} and which we discuss in the next section. After having introduced the self-interacting version, we will return to the canonical formulation in  Section~\ref{section:RecursionAndSpan} to show how to make it equally computationally tractable. 

While \eqref{eq:exp_pure_1} highlights the many-body expansion aspect, it is usually more convenient to employ the more compact expression 
\begin{equation}
    \label{eq:pure_compact}
    V^\mathcal{N}(\bX) = \sum_{\bk \in \bK} c_{\bk} \calA_{\bk}(\bX),
\end{equation}
where $\bK$ is a finite set of ordered tuples $(k_1, \dots, k_N)$ with $0 \leq N \leq \calN$. {We also say that $\mathcal{A}_{\bk}$ is of order $N$, denoted by $\tord(\bk)$, if $\bk = (k_1, \cdots, k_N)$ is of length $N$.}

\subsubsection{Self-interacting cluster expansion}
The computational cost of evaluating the basis $\mathcal{A}_\bk$ scales as $\binom{J}{N}$ due to the naive summation over all unique clusters of $N$ particles chosen from $J$ particles. To obtain a computationally efficient formulation, Drautz \cite{DrautzACE} proposed allowing self-interaction terms in the many-body expansion,
\begin{equation}
\label{equa:selfInteractExpansion}
    U^{\mathcal{N}}(\mathbf{X}) = 
    \sum^\mathcal{N}_{N=0} \sum_{j_1, \cdots, j_N} u_N(\bx_{j_1}, \cdots, \bx_{j_N}).
\end{equation}
In contrast with \eqref{equa:CanonicalExpansion} the summation $\sum_{j_1, \dots, j_N}$ is taken over repeated clusters as well as spurious clusters where one or more particles may be repeated. We therefore call \eqref{equa:selfInteractExpansion} the \emph{self-interacting} cluster expansion. It is sometimes suggested \cite{DrautzACE, ACECompleteness} that this is equivalent to the canonical formulation \eqref{equa:CanonicalExpansion} but we will show in Theorem \ref{theorem:spanKK'} that after discretization this is only true under specific conditions on the basis.

The tensor product structure in the $N$-dimensional sums in \eqref{equa:selfInteractExpansion} can be utilized to obtain a computationally efficient parameterization. Proceeding as above by expanding the $u_N$ terms in a tensor product basis and then interchanging summation we obtain 
\begin{align}
    \label{eq:ex_impure_1}
    U^{{\mathcal{N}}}(\bX) &= \sum_{N = 0}^{\calN} \sum\limits_{\substack{\bk = (k_1, \cdots, k_N) \\ {\bk \text{ ordered }}}} c_{\bk} \bAA_{\bk}(\bX), \qquad \text{where} \\ 
    \label{equa:ImpureBasis}
    \bAA_{k_1, \dots, k_N}(\bX) 
    &:= 
    \sum_{j_1, \cdots ,j_N} \prod^N_{t = 1} \phi_{k_t}(\bx_{j_t}) = 
    \prod^N_{t = 1} \sum^J_{j = 1} \phi_{k_t}(\bx_j).
\end{align}
Substituting 
\begin{equation}
    A_k(\bX) = \sum^J_{j = 1} \phi_{k}(\bx_j), \qquad 
    \bAA_{\bk}(\bX) = \prod_t A_{k_t}(\bX),
\end{equation}
we observe that the evaluation is comprised of two stages: (i) a pooling operation at an $O(J)$ cost per feature $A_k$ and (ii) the $N$-correlations $O(N)$ cost per feature $\bAA_{\bk}$. As a matter of fact, in most scenarios the latter can be reduced to $O(1)$ cost per feature as explained in \cite{RecursiveEvalN-bodyequi}.

\subsection{Efficient Evaluation of the Canonical Cluster Expansion}
\label{section:RecursionAndSpan}
We now discuss our main theoretical results: exposing and analyzing an efficient evaluation algorithm for the canonical cluster expansion. 
The overarching idea is to construct a ``purification operator'' which transforms the self-interaction expansion to the canonical expansion. 
The potential for such an algorithm was hinted at in \cite{ACECompleteness}, but no details or detailed analysis of the purification operator were given. We also note that much of the construction which follows could be directly applied without the symmetrization step to generate a permutation invariant function on $\text{MS}(\Omega)$, or a subspace of $\text{MS}(\Omega)$ containing only a fixed number of particles, which will be discussed further in Section~\ref{section:orthogonalsymbasis}.

For illustrative purposes, we first demonstrate the procedure for a two-correlation. Notice that $\bAA_{k_1 k_2}$ can be rewritten as
$$\mathbf{A}_{k_1k_2} = \sum\limits_{j_1 \neq j_2}\phi_{k_1}(\mathbf{x}_{j_1})\phi_{k_2}(\mathbf{x}_{j_2}) + \sum\limits_{j}\phi_{k_1}(\mathbf{x}_j)\phi_{k_2}(\mathbf{x}_j) = \mathcal{A}_{k_1k_2} + \sum\limits_{j}\phi_{k_1}(\mathbf{x}_{j})\phi_{k_2}(\mathbf{x}_j),$$
or, conversely, 
$$\mathcal{A}_{k_1k_2} = \mathbf{A}_{k_1k_2} - \sum\limits_{j}\phi_{k_1}(\mathbf{x}_j)\phi_{k_2}(\mathbf{x}_j).$$
A key assumption of our following construction is that the pointwise product $\phi_{k_1}(\mathbf{x}_j)\phi_{k_2}(\mathbf{x}_j)$ can be re-expanded, or {\em linearized} as this operation is commonly called, in a finite sum,
\begin{equation} \label{eq:linearize_product}
\phi_{k_1}(\mathbf{x}_j)\phi_{k_2}(\mathbf{x}_j) = \sum_{\kappa} \mathcal{P}^{k_1 k_2}_\kappa \phi_{\kappa}(\mathbf{x}_j),
\end{equation}
with linearization coefficients $\mathcal{P}_{\kappa}^{k_1k_2}$. This assumption is natural for polynomials, including trigonometric polynomials and spherical harmonics; cf. Section~\ref{section:simpleGExamples}.  
Using \eqref{eq:linearize_product} we immediately obtain  
\begin{equation}
\label{eq:PurifyN=2}
\mathcal{A}_{k_1k_2} = \bAA_{k_1k_2} - \sum\limits_{j} \sum_{\kappa} \mathcal{P}_{\kappa}^{k_1k_2} \phi_{\kappa}(\mathbf{x}_j) 
= \bAA_{k_1k_2} - \sum_{\kappa} \mathcal{P}_{\kappa}^{k_1k_2} A_\kappa.
\end{equation}
We note here that for a finite index set $\mathbf{K}$ in~\eqref{eq:pure_compact} and $k_1, k_2 \in \mathbf{K}$, the summation over $\kappa$ has to be closed in $\mathbf{K}$ so that the purification \eqref{eq:PurifyN=2} can be done while only evaluating basis functions corresponding to indices in $\mathbf{K}$. Expanding on this idea one can obtain the following result.
\begin{theorem}
\label{theorem:spanKK'}
    Let $\mathbf{K}$ be a finite ordered index set. Suppose that any pointwise product of $\phi_{k_1}(\mathbf{x})\phi_{k_2}(\mathbf{x})$ can be linearized exactly in terms of a finite sum of $\phi_{\kappa}(\mathbf{x})$ \eqref{eq:linearize_product}. Then there exists a $\mathbf{K}' \supset \mathbf{K}$ such that ${\rm span}(\{\mathbf{A}_{\bk'}\}_{\bk' \in \mathbf{K}'}) \supset {\rm span}(\{\mathcal{A}_{\bk}\}_{\bk \in \mathbf{K}})$.
\end{theorem}

The proof is given in Appendix \ref{Appendix:RecursionFormula}, where we also derive a recursion (Eq.~\eqref{equa:Purify}) analogous to~\eqref{eq:PurifyN=2} which allows the explicit construction of $\mathbf{K}'$. 
Building on the above theorem we obtain the following corollary declaring the existence of the desired purification operator. Its practical implementation is discussed in Section~\ref{sec:implementation_remarks}.

\begin{corollary}
\label{corollary:purifymat}
    Under the conditions in Proposition \ref{theorem:spanKK'}, for each $\bk \in \mathbf{K}$, $\mathcal{A}_{\bk}$ can be 
    expressed uniquely as a linear combination of the $\{\mathbf{A}_{\bk'}\}_{\bk' \in {\bf K}'}$ basis. That is, there exists a unique linear operator $\purifymat = (\mathcal{P}_{\bf k'}^{\bf k})_{{\bf k} \in {\bf K}, {\bf k'} \in {\bf K'}}$, such that 
    \begin{equation}
    \label{eq:linearizek'}
    \mathcal{A}_{\bk} = \sum_{\bk' \in \mathbf{K}'}\mathcal{P}^{\bk}_{\bk'}\mathbf{A}_{\bk'}.
    \end{equation}
    For brevity, we may simply write $\mathcal{A} = \mathcal{P}\mathbf{A}$, where $\mathbf{A} = \{\mathbf{A}_{\bk'}\}_{\bk' \in \mathbf{K}'}$ and $\mathcal{A} = \{\mathcal{A}_{\bk}\}_{\bk \in \mathbf{K}}$.
\end{corollary}

\begin{remark}
As stated in \cite{DrautzACE} and \cite{ACECompleteness}, following their construction, with the decomposition
\begin{equation}
\label{eq:decomposeWN}
\mathcal{A}_{\bk}(\{\br_j\}^J_{j=1}) = \mathbf{A}_{\bk}(\br_{j_1}, \cdots, \br_{j_N}) + W_{N-1}(\{\br_j\}^J_{j=1}),
\end{equation}
$W_{N-1}$ is then of correlation order $N - 1$ and can be absorbed by basis terms of order less than or equal to $N-1$. However, in many cases $\mathbf{K}$ is not rich enough for such an absorption; we give explicit examples in Section~\ref{section:MLIPs} and Table \ref{table:ExtraBasis}. To obtain the canonical expansion via a linear purification, one has to extend $\mathbf{K}$ as discussed in Theorem \ref{theorem:spanKK'}. In those situations, $span(\{\mathbf{A}_{\bk}\}_{\bk \in \mathbf{K}}) \neq span(\{\mathcal{A}_{\bk}\}_{\bk \in \mathbf{K}})$. In the setting of a complete (infinite basis) basis, they are formally always equivalent.
\end{remark}

Two natural questions regarding the efficiency of applying the purification operator $\mathcal{P}$ arise: (1) how does $\mathbf{K'}$ depend on $\mathbf{K}$ in \eqref{eq:linearizek'}, e.g. how much larger is it; and (2) how sparse is the purification matrix, i.e., how many of the entries $\mathcal{P}^{\bk}_{\bk'}$ will be non-zero in \eqref{eq:linearizek'}? The following two propositions (see Appendix \ref{Appendix:Efficiency} for proofs) provide partial answers and illustrate that the purification can be performed efficiently. 

The next proposition illustrates the fact that the purification of $\mathcal{A}_\bk$ preserves properties of the linearization of products of $\phi_{k}$. Extending the definition of degree on $\mathcal{A}_{\bk}$ and $\mathbf{A}_{\bk}$ naturally as $\deg(\mathcal{A}_{\bk}) = \deg(\mathbf{A}_{\bk}) = \tdeg(\bk) = \sum_{t} \tdeg(k_t)$, for $\bk = (k_t)_t$, one obtains the following result. 

\begin{proposition}
\label{proposition:preserving}
    Suppose that the linearization is total-degree preserving,
    \begin{equation}
    \label{equa:prodspan_totdeg}
    \phi_{k_1} \phi_{k_2} \in \text{span}\big\{\phi_\kappa | \tdeg(\kappa) \leq \tdeg(k_1) + \tdeg(k_2)\big\}
    \qquad \forall k_1, k_2. 
    \end{equation}
     Then, the purification \eqref{eq:linearizek'} is also total-degree preserving; that is, 
     \begin{equation}
     \label{eq:totdegzerocoef}
        \mathcal{P}^{\bk}_{\bk'} = 0
        \qquad \text{whenever } 
        \tdeg(\bk')  > \tdeg(\bk). 
     \end{equation}
\end{proposition}

Note that a polynomial basis equipped with the usual definition of polynomial degree satisfies \eqref{equa:prodspan_totdeg}. Importantly, this idea is readily transferable to also show the sparsity of the purification operator in symmetry-based sparsification in Section~\ref{section:simpleGExamples}. Via Proposition \ref{proposition:preserving}, one can see that $\mathcal{P}$ is a triangular matrix with non-zero diagonal entries (cf. Figure \ref{fig:sparsity}) $\mathbf{K}'$ if ordered according the $\deg$ in \eqref{equa:prodspan_totdeg}, hence deriving a strong version of the result in Theorem \ref{theorem:spanKK'}. 

\begin{corollary}
\label{corollary:spanK'K'}
    Let $\mathbf{K}$ be a finite ordered index set.  Suppose that any pointwise product of $\phi_{k_1}(\mathbf{x})\phi_{k_2}(\mathbf{x})$ can be linearized exactly in terms of a finite sum of $\phi_{\kappa}(\mathbf{x})$ \eqref{eq:linearize_product} and is total-degree preserving \eqref{equa:prodspan_totdeg}. Then there exists a $\mathbf{K}' \supset \mathbf{K}$ such that ${\rm span}(\{\mathbf{A}_{\bk'}\}_{\bk' \in \mathbf{K}'}) = {\rm span}(\{\mathcal{A}_{\bk}\}_{\bk \in \mathbf{K'}})$. In addition, we have
    \begin{equation}
    \label{eq:totdegK'subset}
        \mathbf{K'} \subset \big\{\bk' \big| \tdeg(\bk') \leq \max_{\bk \in \mathbf{K}} \tdeg(\bk) \big\}.
    \end{equation}
\end{corollary}

Corollary \ref{corollary:spanK'K'} follows directly from the fact that a triangular matrix with strictly non-zero diagonal entries is invertible and Proposition \ref{proposition:preserving}. The statement \eqref{eq:totdegK'subset} follows from \eqref{eq:totdegzerocoef}. Next we show a general estimate on the sparsity of the transformation.

\begin{proposition}
\label{proposition:bound}
Let $K \geq 1$ be a uniform bound on the number of terms in the linearization of all $\phi_{k_i} \phi_{k_j}$. For each $\bk \in \mathbf{K}$, let $M_N$ be the number of non-zero $\mathcal{P}^{\bk}_{\bk'}$ \eqref{eq:linearizek'}, where $N = \tord(\bk)$. Then $M_{N} \leq \prod \limits^{N-1}_{t=1} (Kt+1) \leq K^N N!$. 
\end{proposition}

The bound $K$ in Proposition \ref{proposition:bound} is basis dependent. For example, if particles are one-dimensional and $\phi_k$ are Chebyshev polynomials, then $K = 2$. For complex trigonometric polynomials, $K = 1$. For spherical harmonics, $K$ grows with the maximum degree. In practice the purification operator $\purifymat$ is actually sparser than the pessimistic bound in Proposition \ref{proposition:bound}; see Figure \ref{fig:sparsity}. The sub-asymptotic behavior of the non-zero coupling coefficients (cf. Figure \ref{fig:ChebBound}) is attributable to the fact that overlapping non-zero terms combine in the purification \eqref{eq:linearizek'}, which is not taken into account in the derivation of the general upper bound approximation in Proposition \ref{proposition:bound}; see also our discussion of Example \ref{example:cancellationCoeffs} below.

\begin{figure}[h!]
     \begin{subfigure}[b]{0.5\textwidth}
        \centering
        \includegraphics[width=3in]{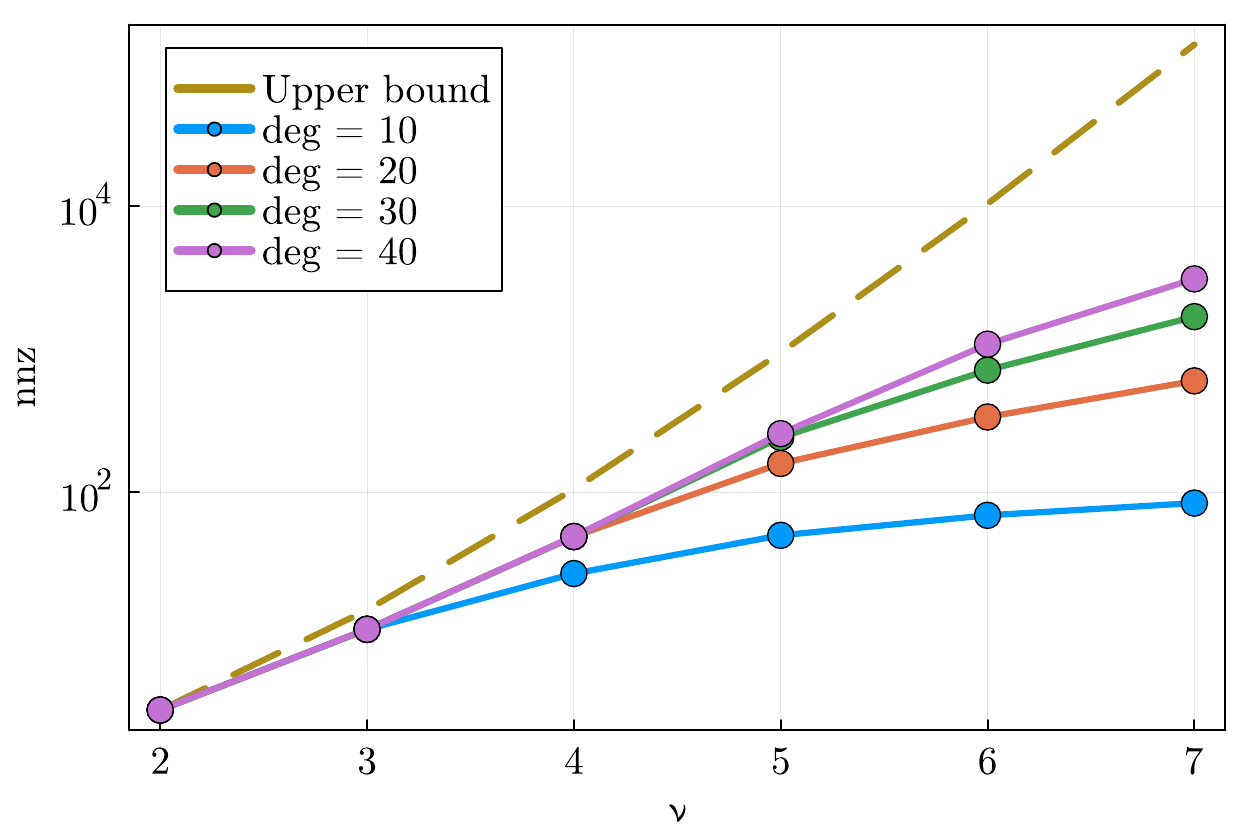}
        \caption{}
        \label{fig:ChebBound}
     \end{subfigure}
     \hfill
     \begin{subfigure}[b]{0.5\textwidth}
        \centering
        \includegraphics[width=2.1in]{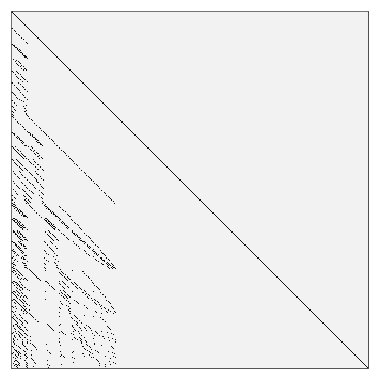}
        \caption{}
        \label{fig:sparsity}
     \end{subfigure}
     \caption{(a) Bound of non-zero terms in $\bk'$ summation in \eqref{eq:linearizek'} when using a Chebyshev basis for embedding one-dimensional particles ${\bf x}_j \in [-1,1]$; The higher the correlation order and the lower the degree, the more the bound is overestimated. (b) Sparsity pattern of $\purifymat$ for correlation order three and total degree = 20. Black pixels indicate non-zeros; sparsity $\approx 1.48\%$. $\mathcal{P}$ is triangular since the Chebyshev basis is total-degree preserving \ref{equa:prodspan_totdeg}.}
\end{figure}

\subsubsection{Implementation remarks}
\label{sec:implementation_remarks}
We now provide some information how the purification framework can be implemented in practice. Populating the matrix $\purifymat$ can be an expensive operation, but it needs to be done only once in a preprocessing step, and is highly parallelizable. A prototype implementation is provided in \cite{ACE1x}.

The first step is to evaluate the linearization coefficients $\mathcal{P}^{k_i k_j}_{\kappa}$ of $\phi_{k_i}\phi_{k_j}$. For some classes of polynomials, including the general family of Jacobi polynomials~\cite{LinearizeProductJacobi} and spherical harmonics~\cite{byerly1893elemenatary, yutsis1962mathematical}, they can be computed from known formulas. This is straightforward to generalize to products of polynomials.
However, in practice, we found it most straightforward to obtain the $\mathcal{P}^{k_i k_j}_{\kappa}$ coefficients by solving a least squares problem,
\begin{equation}
\min_{\mathcal{P}^{k_i k_j}_{\kappa}} \sum_{x \in \mathcal{X}}
\Big| \sum_{\kappa} \mathcal{P}^{k_i k_j}_{\kappa} \phi_{\kappa}(x) - \phi_{k_i}(x)\phi_{k_j}(x) \Big|^2,
\label{equa:coefls}
\end{equation}
where $\mathcal{X}$ is taken as random samples from a suitable distribution that makes this problem well posed \cite{cohen2013stability}. 
Note that all instances of \eqref{equa:coefls} for different $(k_i, k_j)$ pairs utilize the same design matrix, and hence one requires only a single QR factorization. 
If the $\{\phi_{k}\}_{k \in \mathcal{I}}$ are, to name one important example, a set of orthogonal polynomials, then (\ref{equa:coefls}) can be solved in a numerically stable and efficient manner.

As the second step we can now populate the matrix $\purifymat$ which represents the purification operator (cf. Corollary \ref{corollary:purifymat}).
The first step already gives us all the entries $\mathcal{P}^{k_1 k_2}_{\bk}$. For $N > 2$ we can proceed recursively. For each order $N = 2, 3, \cdots, \mathcal{N}$, we update the list of non-zero indices in $\mathcal{P}$ using the recursion relation Equation \eqref{equa:Purify}. We note that these update are highly parallelizable due to the following result, the proof of which is given in Appendix \ref{Appendix:Efficiency}. 

\begin{proposition}
\label{Proposition:parallelizable}
    For every ${\bk}$ of order $N \geq 2$, the purification operator entries $\mathcal{P}^{\bk}_{\bk'}$, where $\bk'$ has order $N'$, can be non-zero only if $\bk' = \bk$ or if $N' < N$.
\end{proposition}

\subsection{Orthogonal Symmetric Basis}
\label{section:orthogonalsymbasis}
In this section we discuss how the construction of the foregoing sections can be used to construct an orthogonal basis for symmetric functions, in particular highlighting the connection between purification and orthogonality. We now consider the simpler setting of approximating a {\em symmetric} function $f : \Omega^N \to \mathbb{F}$, i.e., it is invariant under permutations of its arguments,
\[
    f(x_1, \dots, x_N) = f(x_{\sigma 1}, \dots, x_{\sigma N}) \qquad \forall \sigma \in S_N;
\]
or $f \circ \sigma = f$ in short. 
It is straightforward to see that, if $\{\phi_k\}_{k\in \mathcal{I}}$ is a basis for one-particle functions, then 
\[
    \mathcal{A}^{(N)} := \big\{ \mathcal{A}_{\bf k} \big| {\bf k} = (k_1, \dots, k_N)
            \text{ ordered} \big\}, 
\]
is a basis of symmetric functions on $\Omega^N$; see \cite{bachmayr2023polynomial} for a detailed discussion. The purification operation $\mathcal{A} = \mathcal{P} {\bf A}$ yields an efficient evaluation scheme for this basis. 

Likewise, it is easy to show~\cite{bachmayr2023polynomial} that the set
\[
    \mathbf{A}^{(N)} := \big\{ \mathbf{A}_{\bf k} \big| {\bf k} = (k_1, \dots, k_N)
            \text{ ordered} \big\}
\]
forms a basis of symmetric functions on $\Omega^N$ if $1 \in {\rm span}(\phi_{k})$. The ``self-interactions'' can be absorbed in this case, 
implying ${\rm span}(\mathbf{A}^{(N)}) = {\rm span}(\mathcal{A}^{(N)})$.

Suppose now that $\phi_k$ are orthonormal with respect to the $L^2_\mu(\Omega)$-inner product, 
\begin{equation}
\label{equa:standL2innerprod}
    \langle \phi_k, \phi_{k'} \rangle_{L^2_\mu(\Omega)} = \int_{L^2(\Omega)} \phi_{k} \phi_{k'} \,d \mu \ = \delta_{kk'}.
\end{equation}
Although $\mathbf{A}^{(N)}$ and $\mathcal{A}^{(N)}$ have the same spanning space provided $1 \in {\rm span}(\phi_{k})$, only $\mathcal{A}^{(N)}$ inherits the orthogonality \eqref{equa:standL2innerprod}: it is orthogonal (though not orthonormal) with respect the induced $L^{2}_{\mu^{\otimes N}}(\Omega^N)$-inner product. 
For notational simplicity we demonstrate this for the special case of $N = 2$: if $\bk = (k_1, k_2)$ and $\bk' = (k_1', k_2')$ then 
\begin{equation} \label{equa:innerproductL2N=2} 
\begin{split}
    & \big\langle \mathcal{A}_{\bk}, \mathcal{A}_{\bk'} \big\rangle_{L^2_{\mu^{\otimes 2}}(\Omega^2)}\\
    & = \big\langle \phi_{k_1}(x_1)\phi_{k_2}(x_2) + \phi_{k_1}(x_2)\phi_{k_2}(x_1), \phi_{k'_1}(x_1)\phi_{k'_2}(x_2) + \phi_{k'_1}(x_2)\phi_{k'_2}(x_1) \big\rangle_{L^2_{\mu^{\otimes 2}}(\Omega^2)} \\
    & = 2\delta_{k_1 k'_1}\delta_{k_2 k'_2} + 2\delta_{k_1 k'_2}\delta_{k_2 k'_1},
\end{split}
\end{equation}
which equals $0$, $2$ or $4$ depending on the values of $\bk$ and $\bk'$. Indeed, following a similar argument in \eqref{eq:exp_pure_1} and the above calculation of $\langle \mathcal{A}_{\bk}, \mathcal{A}_{\bk'} \rangle_{L^2_{\mu^{\otimes 2}}(\Omega^2)}$ \eqref{equa:innerproductL2N=2} one can show that a series expansion of symmetric function with $\mathcal{A}^{(N)}$ is equivalent to a multivariate symmetric orthogonal polynomial series with orthogonality $L^2_{\mu^{\otimes N}}(\Omega^ N)$. In contrast $\mathbf{A}^{(N)}$ is clearly not orthogonal in the $L^2$ sense, a fact which we explore numerically in more detail in Section~\ref{section:ConditionNumber}.

\subsubsection{Regularization of cluster expansions for multisets} \label{section:PurificationAndOrthogonality}
Our observations about orthogonality for parameterizing symmetric functions in the previous section naturally lead to a discussion of regularization choices for atomic cluster expansion models for multi-sets. We now explore the relationship between the canonical and self-interacting expansions from the point of view of regularization, leading to the observation that one may equivalently think of the purification operator in \eqref{eq:linearizek'} as a special regularizer on the self-interacting expansion. In what follows we consider classical Tikhonov regularizations of the form
\begin{equation}
\label{eq:genericRegLSQ}
   \min_{\bf c} \|\Psi \mathbf{c} - \mathbf{y}\|^2_2 + \|\Gamma \mathbf{c}\|^2_2,
\end{equation}
where $\Psi$ is the design matrix corresponding to $\mathcal{A}$ or $\mathbf{A}$ respectively, $\Gamma$ is the regularizer and $\mathbf{c}$ are the coefficients in the corresponding expansions. Consider an inner product for the $\phi_k$ defined by \eqref{equa:standL2innerprod}. Given a maximum correlation order $\mathcal{N}$ one can define the following natural inner product for the canonical expansion \eqref{equa:CanonicalExpansion} by considering each $v_N$ as an independent function on $\Omega^N$:
\begin{equation}
\label{eq:canonical_expansion_innerproduct}
\langle V^{\mathcal{N}}, V'^{\mathcal{N}} \rangle = \sum^{\mathcal{N}}_{N = 0} \langle v_N, v'_N \rangle_{L^2_{\mu^{\otimes N}}(\Omega^N)}.
\end{equation}
We remind ourselves here that one straightforwardly finds that the canonical $\mathcal{A}_{\bk}$ basis is orthogonal with respect to the inner product in the sense of \eqref{equa:innerproductL2N=2}. We can thus assume without loss of generality that $\mathcal{A}_{\bk}$ is orthonormal with respect to this inner product, as we can always rescale the basis. With $\mathcal{A}_{\bk}$ assumed orthonormal the Tikhonov regularization term with $\Gamma = I$ is found to simply be
\begin{align*}
    \|\mathbf{c}\|^2_{2} & = \sum_{k} c^2_k + \sum_{k_1, k_2} c^2_{k_1, k_2} + \cdots \sum_{k_1, \cdots, k_{\mathcal{N}}} c^2_{k_1, \cdots, k_\mathcal{N}} \\
    & = \|v_1\|^2_{L^2_\mu(\Omega)} + \cdots + \|v_\mathcal{N}\|^2_{L^2_{\mu^{\otimes \mathcal{N}}(\Omega^\mathcal{N})}},
\end{align*}
which is consistent with the inner product in \eqref{eq:canonical_expansion_innerproduct}. An analogous definition for the self-interacting expansion yields an inner product
\begin{align*}
\langle U^{\mathcal{N}}, U'^{\mathcal{N}} \rangle 
& = \int_{L^2(\Omega)} u_1 \bar{u}'_1(x) + u_2\bar{u}'_2(x, x) + \cdots + u_{\mathcal{N}}\bar{u}'_{\mathcal{N}}(x, \cdots, x) \,d \mu\ \\
 & + \int_{L^2_{\mu^{\otimes 2}} (\Omega^2)} u_2\bar{u}'_2(x_1, x_2) + \cdots + \int_{L^2_{\mu^{\otimes \mathcal{N}}}(\Omega^{\mathcal{N}})} u_{\mathcal{N}}\bar{u}'_{\mathcal{N}} \,d \mu^{\otimes \mathcal{N}} \ .
\end{align*}
However, applying $\Gamma = I$ Tikhonov regularization for the self-interacting basis one also obtains a Tikhonov term of the form
\begin{equation}
    \|\mathbf{c}\|^2_{2} = \|u_1\|^2_{L^2_\mu(\Omega)} + \cdots + \|u_\mathcal{N}\|^2_{L^2_{\mu^{\otimes \mathcal{N}}}(\Omega^\mathcal{N})},
\end{equation}
where the self-interacting terms in $\langle U^{\mathcal{N}}, U'^{\mathcal{N}} \rangle$ are completely absent. This discussion motivates an alternative view of the linear purification operator $\mathcal{P}$ with $\mathcal{A} = \mathcal{P} \mathbf{A}$ in \eqref{eq:linearizek'} as a regularizer in the Tikhonov sense: To regularize $U^{\mathcal{N}}$ in the sense of inner products of $V^{\mathcal{N}}$, we may set $\Gamma = \mathcal{P}^{-\top}$ such that all self-interacting terms are accounted for. We explore the improved behavior resulting from this alternative regularization choice in the self-interacting basis in the numerical experiments in Sections~\ref{Section:NumericalExp} and~\ref{section:MLIPs}, where among other results we show that the purification regularizer $\Gamma = \mathcal{P}^{-\top}$ applied to the self-interacting expansion produces equivalently improved results as using the canonical expansion. 

As a final remark, we note that while we used $\Gamma = I$ for the above observation for simplicity, it is straightforwardly generalized to the case where $\Gamma$ is any invertible diagonal matrix which notably includes the important case of the smoothness prior which we discuss in more detail in the numerical experiments in Section~\ref{Section:NumericalExp}.
\subsection{Purification and $G$-Symmetrization}
\label{section:simpleGExamples}
The discussion up to this point suggests that the canonical cluster expansion basis $\mathcal{A}_{\bk}$ can be evaluated efficiently in a way that overcomes the naive $O(\binom{J}{N})$ cost. However, in practice one usually applies a further transformation of the $\mathcal{A}_{\bk}$ basis to impose a Lie group symmetry \cite{ACECompleteness}. This typically results in a structural sparsification of $\mathcal{A}$. 
In the following paragraphs, we demonstrate that the purification operator preserves many aspects of this symmetry-based sparsification, analogous to Proposition \ref{proposition:preserving}. We will focus on the orthogonal group $O(n)$ with $n \leq 3$ (the most common cases for applications) and consider an index set $\mathbf{K}^D_{\mathcal{N}}$, where $D$ is a given bound on the total degree ${\rm deg}(\bk)$ and $\mathcal{N}$ is the maximum correlation order. 

\subsubsection{Review of $G$-symmetrization}
\label{section:G-sym}
Suppose that a target function we wish to approximate satisfies the $G$-invariance \eqref{eq:V_invar}, then it would be preferrable in many scenarios if the approximation scheme can exactly preserve this invariance. 
We therefore review how to extend the ACE construction to an invariant basis. 

In the abstract we assume that $\phi_{k}$ is chosen such that we have a representation $\rho = (\rho_{k k'})$ of the group action, 

\begin{equation}
\phi_{k} \circ g = \sum_{k'} \rho_{k k'}(g) \phi_k,    
\end{equation}
where $g \in G$ and for each $k$ the sum over $k'$ is a finite sum. For all classical Lie groups, such representations are known. 
Then $\mathcal{A}_{\bk'}$ and $\mathbf{A}_{\bk'}$ can be symmetrized 
to obtain the invariant $\mathcal{B}_{\alpha}$ or $\mathbf{B}_{\alpha}$ bases (i.e., $\mathcal{B}_{\alpha} \circ g = \mathcal{B}_{\alpha}$ and
    $\mathbf{B}_{\alpha} \circ g = \mathbf{B}_{\alpha}$), 
\begin{equation}
\label{equa:symmetrization}
    \mathcal{B}_{\alpha} := \sum_{\mathbf{k'}} \mathcal{C}_{\mathbf{k'}}^{\alpha} \mathcal{A}_{\mathbf{k'}}, 
    \qquad \text{and} 
    \qquad
    \mathbf{B}_{\alpha} := \sum_{\mathbf{k'}} \mathcal{C}_{\mathbf{k'}}^{\alpha} \mathbf{A}_{\mathbf{k'}},
\end{equation}
where $\mathcal{C}_{\mathbf{k'}}^{\alpha}$ are sometimes called \textit{generalized Clebsch--Gordan coefficients} and the indices $\alpha$ are simply indexing all possible invariant basis functions that can be generated by symmetrizing $\mathcal{A}_{\bk'}$ or $\mathbf{A}_{\bk'}$. 
We refer to \cite{ACECompleteness, batatia2023general, batatia2023equivariantmat} for the details of this construction. 

This construction forms the basis of a number of highly successful machine learning architectures thanks to its flexibility which allows a vast design space, from linear models \cite{ACECompleteness, LinearACE4OrgMol} and a straightforward extension to non-linear forms \cite{DrautzACE, drautzcarbon} to models using message passing neural networks \cite{Batatia2022mace, batatia2022e3design}. This leads to a wide range of applications of ACE including featurization for jet-tagging \cite{Munoz_2022}, parameterization of wave functions \cite{drautzwavefunc, ACESchrodinger}, learning Hamiltonian operators \cite{ACEHam} and modeling molecular dynamics for materials and molecules~\cite{Batatia2022mace, PACE, MACEFF}.

A convenient feature of our proposed purification framework is that the symmetrization and the purification operations are both represented as sparse matrix multiplications and hence can be merged into a single sparse matrix operation. More precisely, we can pre-compute sparse matrices $\mathcal{C}, \mathcal{P}$ and $\mathcal{C}_{\rm p} = \mathcal{C} \cdot \mathcal{P}$, such that
\begin{equation}
    \mathcal{B} = \mathcal{C}\mathcal{P}\mathbf{A}
        =: \mathcal{C}_{\rm p} \mathbf{A},
\end{equation}
where $\mathbf{A} = \{\mathbf{A}_{\bk'}\}_{\bk' \in \mathbf{K}'}$ and $\mathcal{B} = \{\mathcal{B}_{\alpha}\}_{\alpha}$. 

\subsubsection{Effect of symmetrization for $G = O(1)$}
\label{section:G=O(1)}
A function is $O(1)$ invariant if it is invariant under reflection through the origin. We consider this simple case before moving on to the more challenging $O(2)$ and $O(3)$ cases. 
We assume that particles belong to the closed interval $\Omega = [-1, 1]$ and use  the set of monomials as one particle basis, i.e., 
\begin{equation}
\phi_{k}(x) = x^k.
\end{equation}
The group $O(1)$ consists of only two elements: the identity $g_1 x = x$ and the reflection $g_{-1} x = -x$. We have the representation $\phi_k \circ g_\sigma = \sigma^k \phi_k$. From this observation it can be readily seen that the symmetrization becomes simply a filtering operation, i.e., we can identify the invariant basis indices $\alpha$ with all those tuples $\alpha = {\bf k}$ such that $\sum \bk := \sum_{k \in \bk} k$ is even. 
An $O(1)$-invariant parameterization can now be simply written as  
\begin{equation}
V^{{\mathcal{N}}}(\bx_1, \cdots, \bx_J) = \sum_{\sum \bk = \even} c_{\bk} \mathcal{A}_{\bk}(\bx_1, \cdots, \bx_J).
\end{equation}
Therefore, the index set of interest in the many-body expansion is reduced to 
\begin{equation}
    \mathbf{K} := \mathbf{K}^{D}_{\mathcal{N}} \cap \big\{\bk \,\big|\, {\textstyle \sum \bk = \even } \big\}.
\end{equation}
Observe that the linearization of $\phi_{k}\phi_{k'}$ contains only a single term $\phi_{k+k'}$ with $k + k' \in \mathbf{K}^{D}_{\mathcal{N}}$. We also note that the condition \eqref{equa:prodspan_totdeg} in Lemma \ref{proposition:preserving} is satisfied and the result applies. In this case we have
\begin{equation}
\label{equa:prodspan_exacttotdeg}
\phi_{k_1} \phi_{k_2} \in \text{span}\{\phi_\kappa | \tdeg(\kappa) = \tdeg(k_1) + \tdeg(k_2)\}.
\end{equation}
Following the proof of Proposition \ref{proposition:preserving} one can similarly observe that the linearization \eqref{eq:linearizek'} with monomials is exactly total-degree preserving. In other words \eqref{eq:linearizek'} can be performed with indices $\bk'$ such that $\tdeg(\bk') = \tdeg(\bk)$. Importantly, the number of non-zero entries of the transformation $\mathcal{P}$ hits exactly the lower bound of the proven estimation in Proposition \ref{proposition:bound} (i.e. $K = 1$), attaining the best possible sparsity.

\subsubsection{Effect of symmetrization for   $G = SO(2)$}
\label{section:G=O(2)}
We now consider another academic example, taking $G = SO(2)$ and $\Omega$ the unit circle. This setting can equivalently be understood as the translation group on the torus. Thus, we take 
\[
    \Omega = \mathbb{T} = (-\pi, \pi], 
\]
supplied with periodic boundary conditions. The translation group $G$ can be identified with $\mathbb{T}$ itself. To emphasize the connection with the unit circle and the rotation group we write $x = \theta$.

We then define the one-particle basis 
\begin{equation}
    \phi_{k}(\theta) = e^{ik \theta},
\end{equation}
where $k \in \mathbb{Z}$. For any $g_{\theta_0} \in G$ we have the representation 
\[
    \phi_k \circ g_{\theta_0} = e^{ik \theta_0}\phi_k(\theta).
\]
To satisfy invariance, the symmetrization once again becomes a simple filtering operation. To see this more clearly, we have:
$$
\begin{aligned}
    \prod^N_{t = 1} \phi_{k_t}(\theta_t + \theta_0)& = \prod^N_{t=1} e^{ik_t \theta_0}\phi_{k_t}
    = e^{i (\sum\mathbf{k}) \theta_0}\prod^N_{t = 1} \phi_{k_t}(\theta_t).\\
\end{aligned}
$$
Hence, the invariant basis indices $\alpha$ are tuples $\bk$ such that $\sum \bk = 0$. Consequently, multi-set functions $V^{\mathcal{N}}$ which are permutation and $SO(2)$ invariant can be parameterized by
\begin{equation}
V^{\mathcal{N}}(\theta_1, \cdots, \theta_J) = \sum_{\sum \bk = 0} c_{\bk} \mathcal{A}_{\bk}(\theta_1, \cdots, \theta_J).
\end{equation}
Analogous to monomials in the $O(1)$ case, the linearization of $\phi_{k}\phi_{k'}$ contains only a single term $\phi_{k+k'}$ with $k+k' \in \mathbf{K}^{D}_{\mathcal{N}}$ (i.e. it is exactly total degree preserving). By a similar observation as in Section~\ref{section:G=O(1)}, following the proof of Proposition \ref{proposition:preserving},
one can show that every non-zero term $\mathcal{P}_{\bf k'}^{\bf k}$ in the purification satisfies $\sum {\bf k}' = 0$.
This entails that the purification can be performed with only basis elements from the index set of interest $\mathbf{K} := \mathbf{K}^{D}_{\mathcal{N}} \cap \big\{\bk | \sum \bk = 0\big\}$ and again the number of non-zero entries of the transformation $\mathcal{P}$ attains the lower bound of the proven estimate in Proposition \ref{proposition:bound}.

\subsubsection{Effect of symmetrization for   $G = O(3)$}
\label{section:G=O(3)}
Finally, we turn to the case of isometry invariance in three dimensions, i.e. $G = O(3)$,  which is the most relevant for applications. 

Following \cite{DrautzACE, ACECompleteness}, we use the one particle functions $\phi_{nlm}: \Omega \rightarrow \mathbb{C}$ defined by 
\begin{equation}
    \label{equa:O(3)1pbasis}
    \phi_{nlm}(\mathbf{r}) = R_n(r)Y^{m}_l(\mathbf{\hat{r}}), 
\end{equation}
where $ \Omega = [0, r_{\text{cut}}] \times \mathbb{S}^2 \subset \mathbb{R}^3$ with $r_{\text{cut}}>0$. Here $r = |\mathbf{r}|$, $\mathbf{r} = r\mathbf{\hat{r}}$ and $k = (n, l, m)$ is a multi-index. In \eqref{equa:O(3)1pbasis} $\{R_n\}_{n \in \mathbb{N}_0}$ is a basis of linearly independent polynomials on $[0, r_{\text{cut}}]$ and $Y^{m}_l$ are standard complex spherical harmonics \cite{thomas2018tensor, gerken2021geometric}. It is possible to generalize the construction to allow radial bases $R_{nl}$ (e.g., three-dimensional Zernike polynomials), but we omit this as it introduces additional notational complexity. 

From the representation of spherical harmonics, following \cite{ACECompleteness}, any permutation and $O(3)$ invariant multi-set functions $V^{\mathcal{N}}$ can then be parameterized by
\begin{equation*}
\begin{aligned}
V^{\mathcal{N}}(\br_1, \cdots, \br_J) 
& = \sum_{\substack{\bn \\ \sum \bl = \even}} c_{\bn \bl i} \sum_{\sum \bm = 0} \mathcal{C}^{\bn \bl i}_{\bm}\mathcal{A}_{\bn \bl \bm}(\br_1, \cdots, \br_J) \\
& =: \sum_{\substack{\bn, \bl, i : \\ \sum \bl = \even}} c_{\bn \bl i} \mathcal{B}_{\bn \bl i}(\br_1, \cdots, \br_J),
\end{aligned}
\end{equation*}
where $i$ enumerates all possible symmetric couplings of the generalized Clebsch--Gordan coefficients \cite{ACECompleteness} and $\mathcal{C}^{\bn \bl i}_{\bm}$ are constructed from the representation of the spherical harmonics \cite{DrautzACE, ACECompleteness, RecursiveEvalN-bodyequi}. The constraint $\sum \bl = \text{even}$ and $\sum \bm = 0$ are, respectively, due to the reflection symmetry and rotation symmetry. This corresponds to $\alpha = (\bn \bl i)$ in \eqref{equa:symmetrization}, which entails the index set of interest,
\begin{equation}
\label{equa:O(3)sparsify}
\mathbf{K} := \mathbf{K}^{D}_{\mathcal{N}} \cap \big\{(\bn, \bl, \bm) \,\big|\, {\textstyle \sum \bm = 0, \sum \bl = \even} \big\},
\end{equation}
where $\deg(k) = \deg(nlm)$ is defined as the sum $n + l$ and $\deg(\bk)$ can be naturally extended as $\sum_{t} n_t + l_t$. In this setting, since $|m| \leq l$, it is common to not include $m$ in the definition of total degree. This also ensures that the basis is closed under the symmetrization operator.

As in Sections ~\ref{section:G=O(1)} and \ref{section:G=O(2)}, the key to showing that no additional basis functions are needed for purification lies in the properties of the linearization of $\phi_{n_1 l_1 m_1}\phi_{n_2 l_2 m_2}$. 

The following proposition is a special case of Corollary \ref{corollary:purifymat} when $\mathbf{K}=\mathbf{K'}$. 

\begin{proposition}
    With definition of $\phi_{nlm}$ and $\mathbf{K}$ in \eqref{equa:O(3)1pbasis} and \eqref{equa:O(3)sparsify}, for each $\bk \in \mathbf{K}$, $\mathcal{A}_{\bk}$ can be 
    expressed uniquely as a linear combination of the $\{\mathbf{A}_{\bk'}\}_{\bk' \in {\bf K}}$ basis. That is, there exists a unique (invertible) linear operator $\purifymat = (\mathcal{P}_{\bf k'}^{\bf k})_{{\bf k} \in {\bf K}, {\bf k'} \in {\bf K}}$, such that 
    \begin{equation}
    \label{eq:linearizek'_O(3)}
    \mathcal{A}_{\bk} = \sum_{\bk' \in \mathbf{K}}\mathcal{P}^{\bk}_{\bk'}\mathbf{A}_{\bk'}.
    \end{equation}
\end{proposition}

\begin{proof}
We first show that any product of $\phi_{n_1 l_1 m_1}(\mathbf{r})\phi_{n_2 l_2 m_2}(\mathbf{r})$ can be linearized exactly in terms of a finite sum of $\phi_{\kappa}(\mathbf{r})$ \eqref{eq:linearize_product} with $\kappa \in \mathbf{K}$. Let $\phi_{n_1 l_1 m_1} = R_{n_1}Y^{m_1}_{l_1}$ and $\phi_{n_2 l_2 m_2} = R_{n_2}Y^{m_2}_{l_2}$. Since $\{R_n\}_{n \in \mathbb{N}}$ is a basis of polynomials, there exist finitely many $P^{n_1 n_2}_{\nu}$ such that $R_{n_1}R_{n_2} = \sum_\nu P^{n_1 n_2}_{\nu}R_{\nu}$. Then the linearization coefficients can be obtained by expanding $P^{n_1 n_2}_{\nu}$ and the Clebsch-Gordon coefficients $C^{L M}_{l_1 l_2 m_1 m_2}$ from the contraction rule of spherical harmonics \cite{byerly1893elemenatary, yutsis1962mathematical}. Therefore, Corollary \ref{corollary:purifymat} applies. It remains to show that $\mathbf{K} = \mathbf{K'}$. The total degree preserving property follows from the fact that the expansion over radial basis and spherical harmonics preserves the total degree. Conservation of $\sum \bm = 0$ follows from the fact that $C^{L M}_{l_1 l_2 m_1 m_2}$ is non-zero only if $m_1 + m_2 + M = 0$, with similar deduction as in Section~\ref{section:G=O(2)}. For $\sum \bl = \text{even}$, we first notice that $C^{L M}_{l_1 l_2 m_1 m_2}$ is non-zero only if $l_1 + l_2 + L = 0$. This entails that the linearization of any $\phi_{k_1}\phi_{k_2}$ preserves parity over the $l$ index, precisely:
\begin{equation}
\label{equa:prodspan_parity}
\phi_{n_1 l_1 m_1} \phi_{n_2 l_2 m_2} \in \text{span}\{\phi_{n_\kappa l_\kappa m_\kappa} | l_1 + l_2 \equiv l_\kappa \text{ mod } 2 \}.
\end{equation}
By a similar argument as in Proposition \ref{proposition:preserving}, one can see that the purification \eqref{eq:linearizek'} can be performed with indices preserving the original parity. Therefore, no basis outside $\mathbf{K}$ are needed for the transformation.
\end{proof}

\section{Numerical Experiments: Academic Examples}
\label{Section:NumericalExp}
In this section we present numerical experiments comparing basic properties of the canonical and self-interacting atomic cluster expansions including conditioning, effect of regularization choices and coefficient decay in the simplified regression context. We defer comparisons in the context of machine learning interatomic potentials, the primary application of the atomic cluster expansion framework, to Section~\ref{section:MLIPs}.

\subsection{Conditioning of the gram matrix}
\label{section:ConditionNumber}
In this section we explore the conditioning of the gram matrix for the canonical and self-interacting bases when $G = O(3)$ (cf. Section~\ref{section:G=O(3)}), the most interesting case for applications. Following the discussion in Section~\ref{section:PurificationAndOrthogonality} and using the fact that $\mathcal{B}_{\bn \bl i}$ is simply an orthogonal transformation of $\mathcal{A}_{\bn \bl \bm}$ the condition number of the gram matrix of the $\mathcal{B}_{\bn \bl i}$ basis with respect to the standard $L^2$ inner product should be 1 (up to diagonal scaling) when $N=J$. We present a numerical comparison of the obtained condition numbers for the gram matrices of the respective bases in Table \ref{table:CondTable3D}. Each sample $\bX = \msl x_j \msr_{j = 1}^J$ is drawn independently from $\mu^{\otimes J}$ where $\mu$ is chosen such that $\langle \phi_{nlm}, \phi_{n'l'm'} \rangle_{L^2_{\mu}(\Omega)} = \delta_{nn'}\delta_{ll'}\delta_{mm'}$, where $\Omega = [0, r_{\text{cut}}] \times \mathbb{S}^2$ as in Section~\ref{section:G=O(3)}. For each total degree, $(70 \times \text{basis size})$ data samples are used to form the full design matrix. The submatrices corresponding to each order $N$ are then extracted to compute the gram matrix, followed by a scaling with respect to the diagonal of each submatrix. We observe that the condition numbers of the $\mathcal{B}_{\bn \bl i}$ basis are all close to 1 and the improvement in conditioning is substantial when compared with the self-interacting $\mathbf{B}_{\bn \bl i}$ basis. Similar behavior is observed for $J>N$ but the significance of the effect decreases as $J$ increases  with fixed $N$, as seen in Figure \ref{fig:cond_graph}. The critical point observed in the condition number corresponding to the gram matrix of the self-interacting expansion in Figure \ref{fig:cond_graph} is observed to be present even without the diagonal rescaling of the Gram matrix. We have no satisfactory explanation for this observed phenomenon.

\begin{table}[h!]
\centering
\begin{tabular}{ p{1.0cm} p{2.4cm} p{1.3cm} p{1.3cm} p{1.3cm}
 p{1.3cm} p{1.3cm} }
 \hline
 \multicolumn{2}{c}{\hspace{7mm} Total degree} & 10 &  12 & 14 & 16 \\
 \hline
 $N = 2$ & Canonical & $1.3 \times 10^0$  & $1.3 \times 10^0$  & $1.3 \times 10^0$  & $1.2 \times 10^0$ \\
 
        & Self-interacting &  $1.8 \times 10^2$  & $3.8 \times 10^2$ & $7.8 \times 10^2$ & $1.2 \times 10^3$\\
\hline
 $N = 3$ & Canonical & $1.6 \times 10^0$ & $1.4 \times 10^0$ &  $1.6 \times 10^0$  & $1.5 \times 10^0$ \\
 
           & Self-interacting & $1.3 \times 10^3$ &$5.7 \times 10^3$ &$1.6 \times 10^4$ & $4.8 \times 10^4$ \\
\hline
 $N = 4$ & Canonical & $1.6 \times 10^0$ & $1.6 \times 10^0$ & $2.1 \times 10^0$ & $2.2 \times 10^0$\\
 
           & Self-interacting & $1.0 \times 10^3$ & $1.2 \times 10^4$ & $7.1 \times 10^4$ &$3.6 \times 10^5$ \\
\hline           
 $N = 5$ & Canonical & $1.9 \times 10^0$ & $2.2 \times 10^0$ & $2.2 \times 10^0$ & $2.9 \times 10^0$ \\
 
           & Self-interacting & $4.8 \times 10^2$ & $4.6 \times 10^3$ & $6.3 \times 10^4$ & $4.7 \times 10^5$ \\
\hline        
 $N = 6$ & Canonical & $1.7 \times 10^0$ & $2.0 \times 10^0$ & $2.6 \times 10^0$ & $3.8 \times 10^0$ \\
 
           & Self-interacting & $6.5 \times 10^1$ &$2.6 \times 10^3$ & $1.9 \times 10^4$ & $2.0 \times 10^5$ \\

\end{tabular}
    \caption{Condition numbers for gram matrices of the self-interacting and canonical bases with indicated total degree and basis sizes using the default parameters in {\tt ACE1x.jl},  $\mathcal{N} = 6$ and $ (70 \times \text{basis size})$ data samples to form the design matrix and extract the corresponding order $N$. A diagonal scaling is applied s.t. the expected canonical condition numbers are $1$.}
    \label{table:CondTable3D}
\end{table}

\begin{figure}[h!]
      \centering
     \captionsetup[subfigure]{justification=centering}
     \begin{subfigure}[b]{0.32\textwidth}
     \centering
         \includegraphics[width=\textwidth]{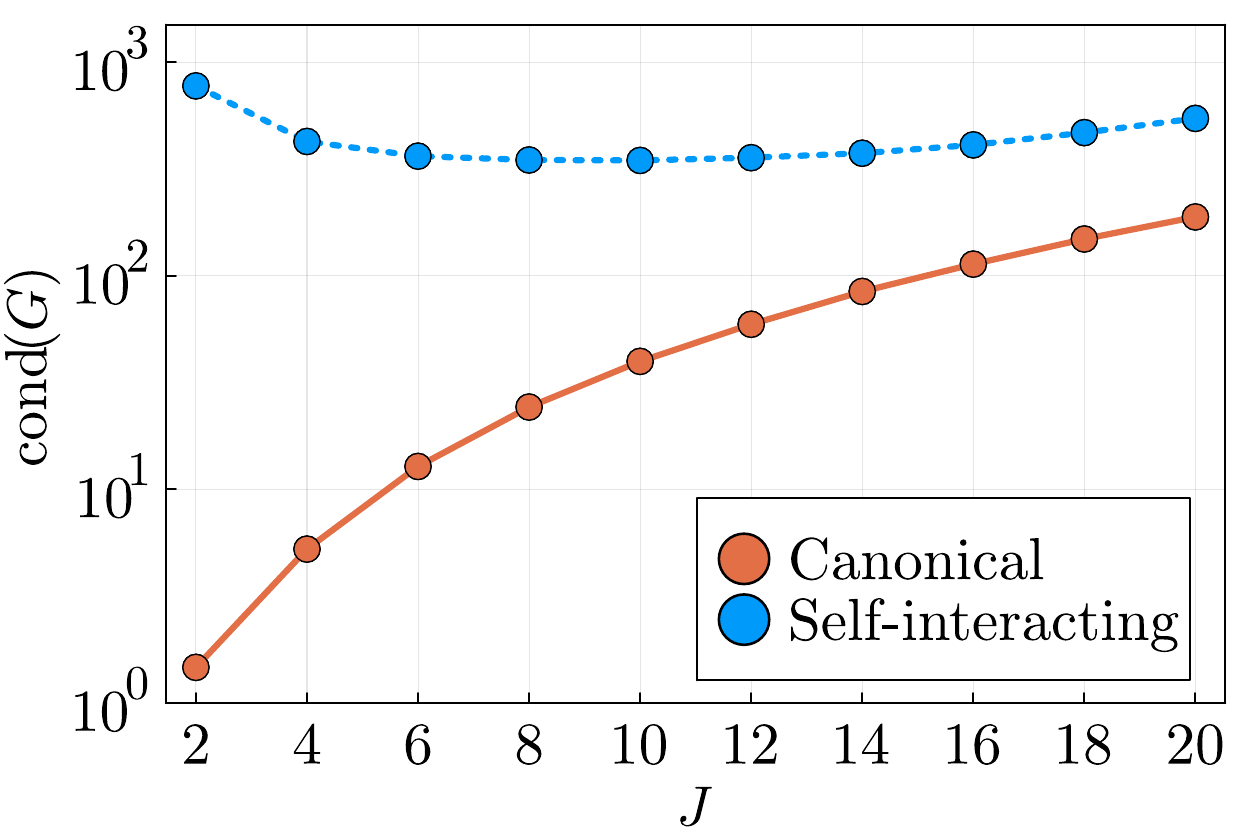}
         \caption{$N = 2$}
         \label{fig:cond_graph_2}
     \end{subfigure}
     \begin{subfigure}[b]{0.32\textwidth}
     \centering
          \includegraphics[width=\textwidth]{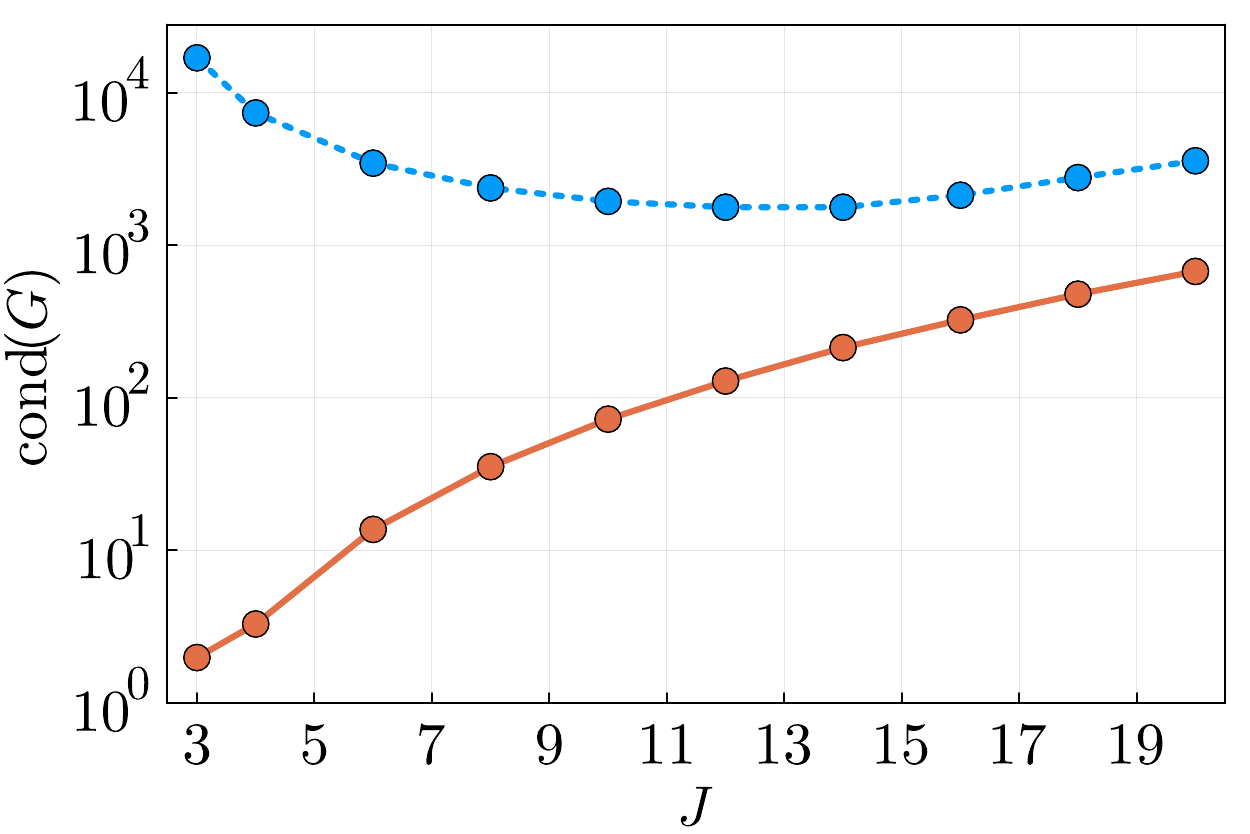}
          \caption{$N = 3$}
         \label{fig:cond_graph_3}
     \end{subfigure}
      \begin{subfigure}[b]{0.32\textwidth}
      \centering
         \includegraphics[width=\textwidth]{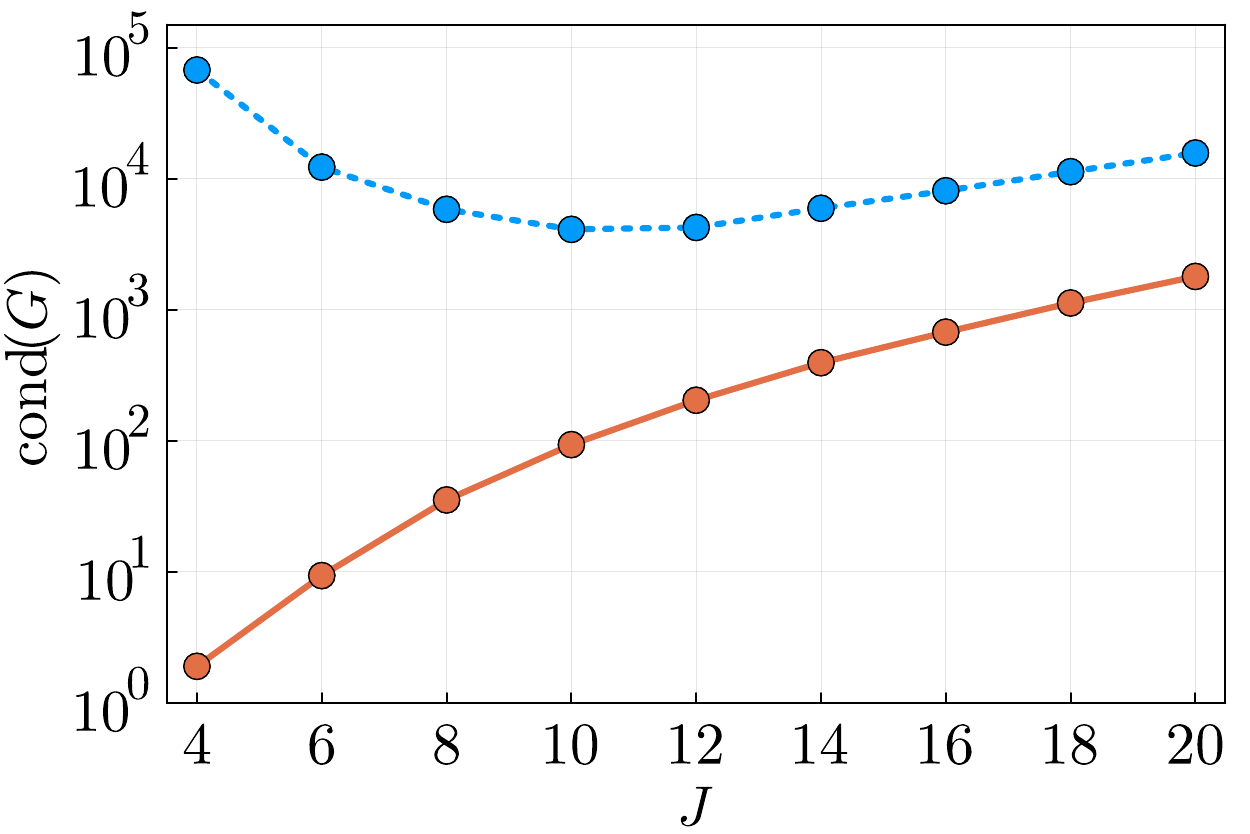}
     \caption{$N = 4$}
     \label{fig:cond_graph_4}
     \end{subfigure}
     \caption{Change of gram matrix condition numbers of the canonical and self-interacting bases with total degree $14$ for the indicated basis orders. A diagonal scaling is applied to the gram matrices consistent with Table \ref{table:CondTable3D}.}
     \label{fig:cond_graph}
\end{figure}

\subsection{Coefficient decay}
\label{section:decaycoeffs}
In the context of function approximation, the coefficient decay of a polynomial approximation is widely used as a proxy for error estimation, especially in the absence of other metrics. A noteworthy feature of orthogonal polynomial expansions is their predictable exponential coefficient decay for sufficiently well-behaved functions, which we investigate numerically for the canonical and self-interacting atomic cluster expansions. We define the following difficult to approximate test functions with $a>0$ and ${\bf x} \in [-1,1]^J$ (though we take $J = N$ in the initial tests):
\begin{align}
\label{eq:testfunctioncoeffs1}
f_{a}(\mathbf{x}) &= \frac{1}{1 + a\|\mathbf{x}\|^2},
\qquad 
\end{align}
\\

$f_{a}$ is a multi-variate generalization of the Runge function, a famously ill-conditioned function to approximate with polynomials.  Notice that $f_{a}$ is permutation invariant and analytic on the $N$-dimensional hypercube $[-1, 1]^{N}$ and increasingly difficult to approximate with increasing $a$, making them good test cases for the discussed ACE bases.

It was shown in \cite{lloydHypercube} that the natural notion of degree for problems of this kind is not the total degree but the {\em Euclidean degree}, defined by 
\[
    {\rm eucl}(\bk) := \big( {\textstyle \sum_{k \in \bk} k^2} \big)^{1/2}. 
\]
that is, the coefficients decay as $O(\rho^{-{\rm eucl}(\bk)})$, for some $\rho > 1$.

In Figure \ref{fig:decaycoeffs} we show representative examples of the observed coefficient decay. Comparing the two bases in this way, we observe that the slowest decaying coefficients corresponding to the product terms in the self-interacting basis are resolved with improved decay in the canonical expansion. Furthermore, while both bases notably show exponential coefficient decay, as seen by the linear behavior on the semi-logarithmic plots, the canonical basis agrees better with the decay behavior expected from quasi-optimal $N$-dimensional Chebyshev polynomial expansions \cite{lloydHypercube} indicated in Figure \ref{fig:decaycoeffs} by a dashed line.

This can potentially be exploited in effectiveness of sparsification and of regularization via a smoothness prior, as we demonstrate in the next section.

\begin{figure}[h!]
     \begin{subfigure}[b]{0.49\textwidth}
         \centering
         \includegraphics[width=\textwidth]{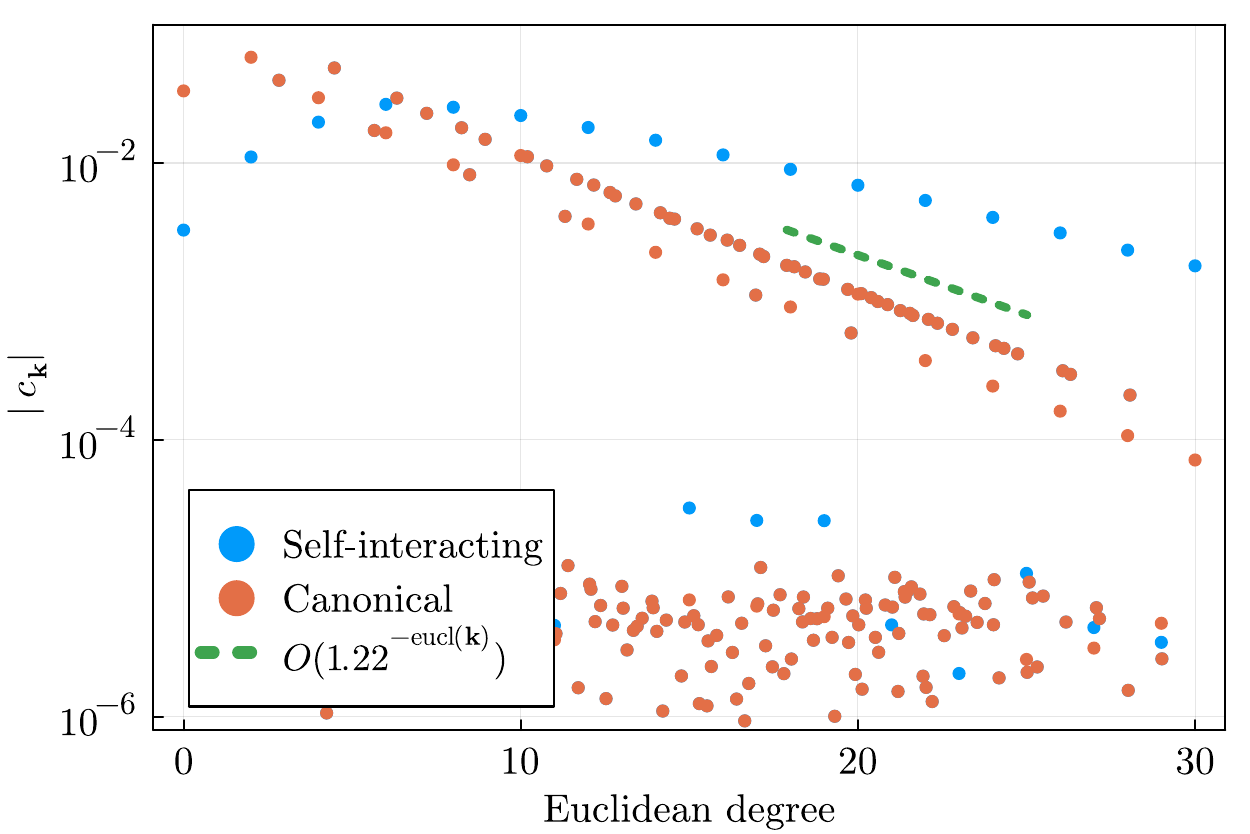}
         \caption {$N = J = 2$}
         \label{fig:decaycoeffs_J=N_1}
     \end{subfigure}
     \hfill
     \begin{subfigure}[b]{0.49\textwidth}
         \centering
         \includegraphics[width=\textwidth]{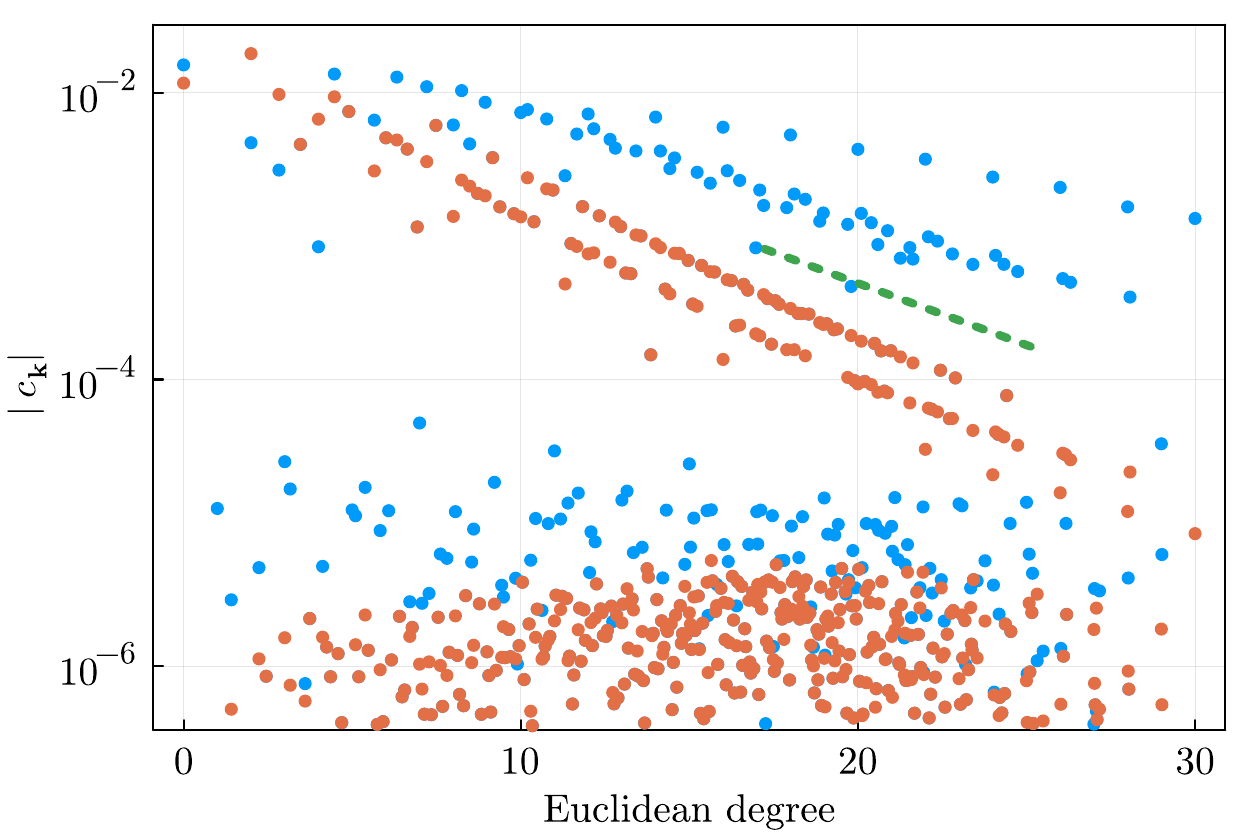}
        \caption{$N = J = 3$}
        \label{fig:decaycoeffs_J=N_2}
     \end{subfigure}
      \caption{Coefficient decay in the canonical and self-interacting bases for $f_{25}$ as defined in \eqref{eq:testfunctioncoeffs1}. All data was sampled independently from a tensor product of distributions with density $\frac{1}{\sqrt{1-x^2}}$ over $[-1,1]^N$ and regressed against an $N$--dimensional ACE basis. The coefficient decay estimate was obtained following \cite{lloydHypercube} and is the same in (a) and (b).}
      \label{fig:decaycoeffs}
\end{figure}

\subsection{Regularization, smoothness priors and the Runge effect}
\label{section:Runge}
Priors and regularization are of critical importance for applications involving linear models \cite{reg_cas, 2023-acepotentials}. In this section we explore the effects of regularization on the two discussed ACE bases and observe that smoothness priors in particular produce better results with the canonical basis. To observe the effects of regularization numerically we again use the difficult to resolve functions defined in \eqref{eq:testfunctioncoeffs1}. We use classical Tikhonov regularization throughout this section, that is we minimize systems of the form
\begin{equation}
\label{equa:RegLeastSquareSystem_lambda}
\min_{\bf c} \|\Psi \mathbf{c} - \mathbf{y}\|^2 + \lambda \|\Gamma \mathbf{c}\|^2,
\end{equation}
where $\Psi$ is the design matrix, $\mathbf{y}$ the target values, $\Gamma$ is a regularizer, $\mathbf{c}$ are the desired coefficients or model parameters and $\lambda$ is a hyperparameter over which we optimize in the range $(10^{-15},10^3)$ using a logarithmic grid search. The regularization choices we consider in this section are the identity prior $\Gamma = I$ and a diagonal regularizer of the form 
\begin{equation}
\label{eq:Gamma_prior}
    \Gamma_{\bk\bk} = \gamma(\bk).
\end{equation}
Taking $\gamma$ increasing with increasing frequency (e.g., polynomial degree) of the basis, such a regularizer encourages smoothness of the fitted target function and is therefore called a {\em smoothness prior}. The chosen growth in the smoothness prior parameter $\gamma$ enforces a corresponding asymptotic coefficient \emph{decay} in $\mathbf{c}$ as readily seen by inspecting the Tikhonov term in  \eqref{equa:RegLeastSquareSystem_lambda}, see also the discussion of coefficient decay in the previous section. A common alternative motivation for the smoothness prior relates it to the gradients $\nabla^p \phi_{k}$ \cite{reg_cas}.

In Section~\ref{section:PurificationAndOrthogonality} we discussed the relationship of these two regularization choices with the canonical and self-interacting basis respectively and showed that they can be rigorously motivated for the canonical basis only, leading to the natural conjecture that they would perform better in that context. In the experiments below we use Chebyshev and Legendre polynomials, i.e. $\phi_k(x) = T_k(x)$, where $T_k$ are Chebyshev or Legendre polynomials of degree $k$, to construct the self-interacting and canonical expansions along with smoothness priors as in \eqref{eq:Gamma_prior} with $\gamma(\bk) = \sum_{t} (1 + k_t)^2$. This choice is made in order to avoid assuming prior knowledge on the coefficient decay of the target function but in practice other more aggressive regularizers may be used. 

Figure \ref{fig:cheblegendre+uniform+cheb} shows some representative examples of the RMSE for approximations of the functions in \eqref{eq:testfunctioncoeffs1} obtained via a Tikhonov regression when $N = J = 4$, while Figure \ref{fig:f12_supnorm} shows examples for the maximum error under identical setting. Several interesting observations can be made from Figures \ref{fig:cheblegendre+uniform+cheb} and \ref{fig:f12_supnorm} which appear to hold with some generality. First, analogous to classical approximation theory results we find that Chebyshev distributed sampling leads to improved errors in both the canonical and self-interacting basis compared to uniformly distributed data. While neither Chebyshev nor uniformly distributed data are very likely to be encountered in a practical application setting, a uniform distribution more closely corresponds to the expected generic approximation theoretic behavior of real data as opposed to the very special case of Chebyshev nodes. As such, the improved behavior observed for the canonical basis in the uniformly distributed cases is noteworthy.

Figure \ref{fig:f12_supnorm} also includes the different regularization of the self-interacting basis suggested by the discussion in Section~\ref{section:PurificationAndOrthogonality}, demonstrating that instead of a change of basis one can equivalently make a special choice of regularization based on the purification transform to obtain the observed improvements. As seen in Section~\ref{section:PurificationAndOrthogonality}, this equivalence is general and not limited to the numerical experiments discussed in this section. Depending on implementation details this may in fact be the more attractive option in many scenarios as a change of regularizer is usually a trivial drop-in replacement compared to a change of basis.

We observe in Figure \ref{fig:cheblegendre+uniform+cheb-normxvsf} that the smoothness prior outperforms the identity prior and the canonical basis outperforms the self-interacting basis when smoothness priors are applied to both. Similar observations were made for a simple extension of the $J > N$ Runge function in Figure \ref{fig:normxvsf-JgtN} and \ref{fig:normxvsgradf-JgtN}.

To further examine the performance of the canonical and self-interacting expansions in terms of the RMSE in regression tasks we consider the following test function $F_{a, \epsilon}: \mathbb{R}^J \rightarrow \mathbb{R}$ defined by 
\begin{equation}
\label{eq:J>Nsumpermtesting}
                F_{a, \epsilon}(\bx) = \sum_{x \subset \bx}f_{a}(x) + \epsilon \prod^J_{i = 1}{x_i},
\end{equation}
where $f_{a}:\mathbb{R}^N \rightarrow \mathbb{R}$ is defined as in \eqref{eq:testfunctioncoeffs1} and $\sum_{x \subset \bx}$ denotes the sum over all unique ``sub-clusters'' $\{x_{j_1}, \dots, x_{j_N}\}$ of length $N$. For $\epsilon > 0$ this test function has qualitative similarities to the form of site energies in the context of interatomic potentials with a higher-body-order terms that the model cannot resolve. 

Figure \ref{fig:JgtN_toyregression} shows the learning curves (RMSE) for approximating example functions of the form \eqref{eq:J>Nsumpermtesting} for $J = 8$ and $N = 4$ with uniformly distributed data over a range of values for $a$ and $\epsilon$, using Tikhonov regression. We observe that the two described bases perform similarly in all tested cases. The reason for the marked difference to the observed behavior in Figures \ref{fig:normxvsf-JgtN} and \ref{fig:normxvsgradf-JgtN} is that, while approximating $f_{a}$ of fixed dimension $J$ using a basis with $J>N$ can be viewed as a truncation of an approximation of order $J$, no notable Runge phenomenon seems to occur for $F_{a, \epsilon}$. This suggests that the self-interacting expansion has similar performance in terms of regression task errors for multi-set functions. This observation is also consistent with the condition number tests of Section~\ref{section:ConditionNumber}.

\begin{figure}[h!]
     \centering
     \captionsetup[subfigure]{justification=centering}
     \begin{subfigure}[b]{0.49\textwidth}
         \centering
         \includegraphics[width=\textwidth]{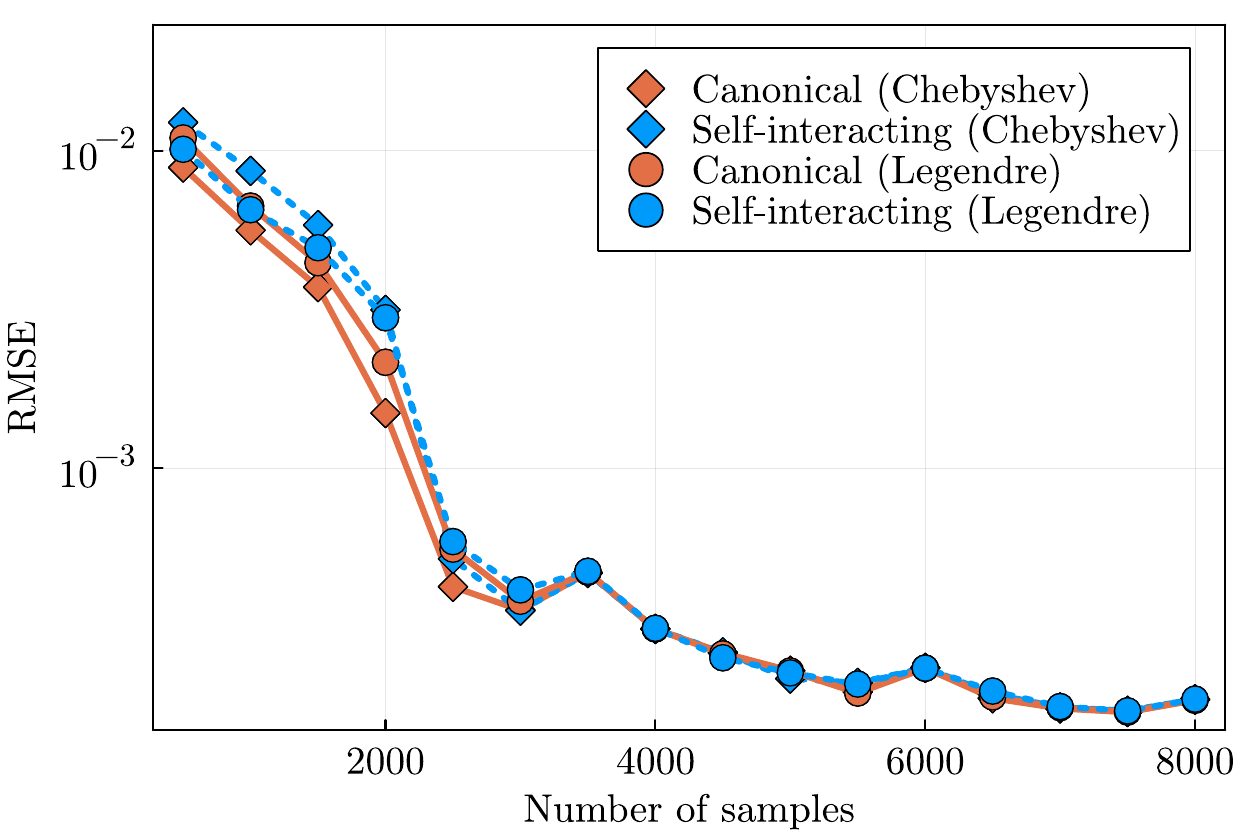}
         \caption{$f_{5}$, distribution with density  $1 / \sqrt{1 - x^2}$}
          \label{fig:f1_cheblegen+cheb_RMSE}
     \end{subfigure}
     \hfill
     \begin{subfigure}[b]{0.49\textwidth}
         \centering
         \includegraphics[width=\textwidth]{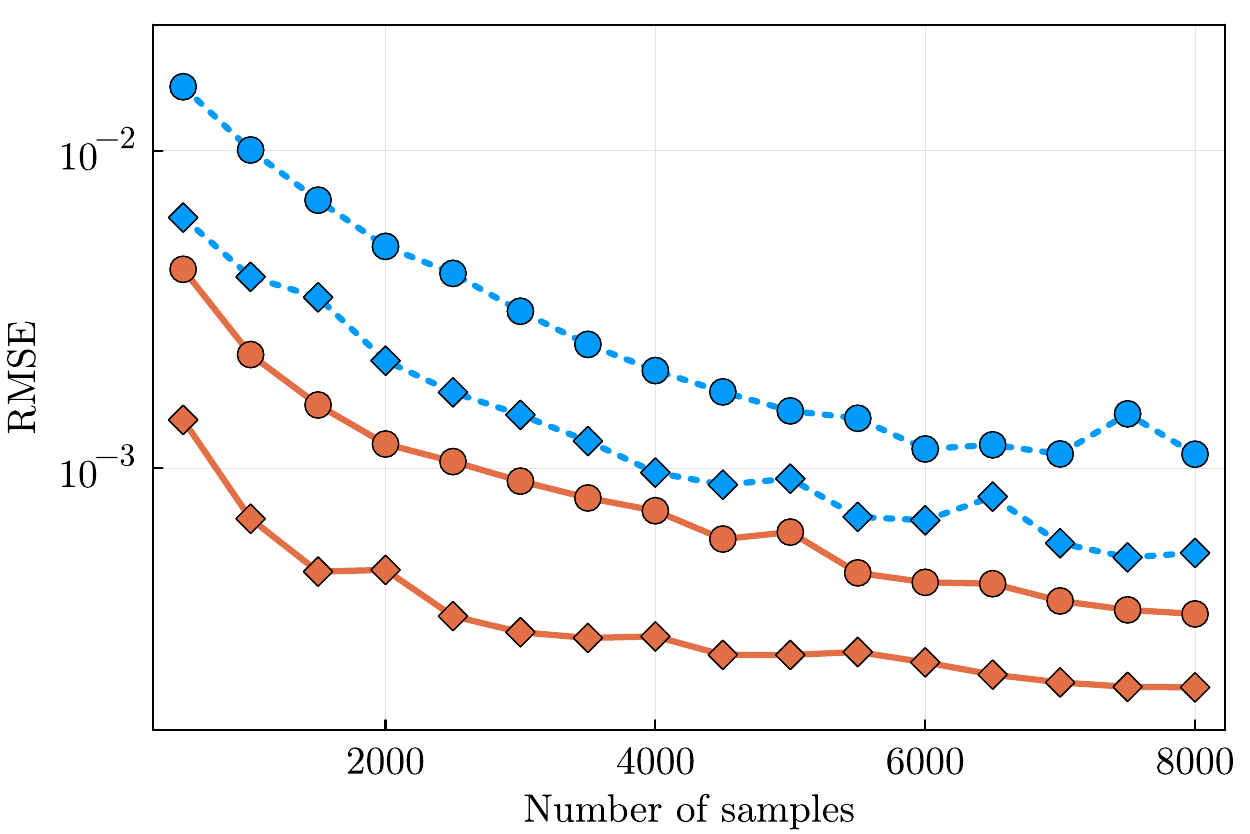}
         \caption{$f_{5}$, uniform distribution}
         \label{fig:f1_cheblegendre+uniform_RMSE}
     \end{subfigure}
     \caption{RMSE when approximating the functions $f_{5}$ in \eqref{eq:testfunctioncoeffs1} with the indicated basis functions of total degree 30 and differently distributed data.}
    \label{fig:cheblegendre+uniform+cheb}
\end{figure}

\begin{figure}[h!]
     \centering
     \captionsetup[subfigure]{justification=centering}
      \begin{subfigure}[b]{0.49\textwidth}
         \centering
         \includegraphics[width=\textwidth]{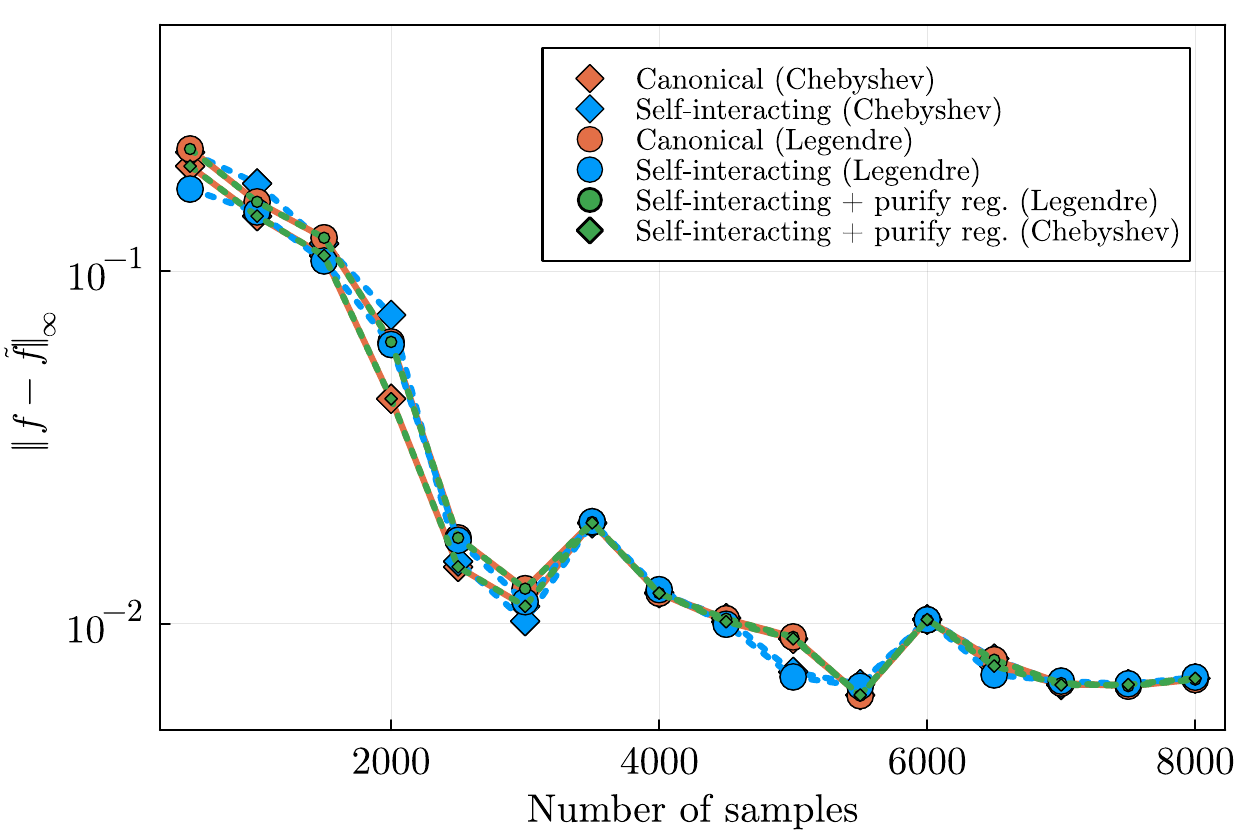}
         \caption{
         $f_{5}$, distribution with density  $1 / \sqrt{1 - x^2}$}
         \label{fig:f1_cheblegendre+cheb_supnorm}
     \end{subfigure}
     \hfill
      \begin{subfigure}[b]{0.49\textwidth}
         \centering
         \includegraphics[width=\textwidth]{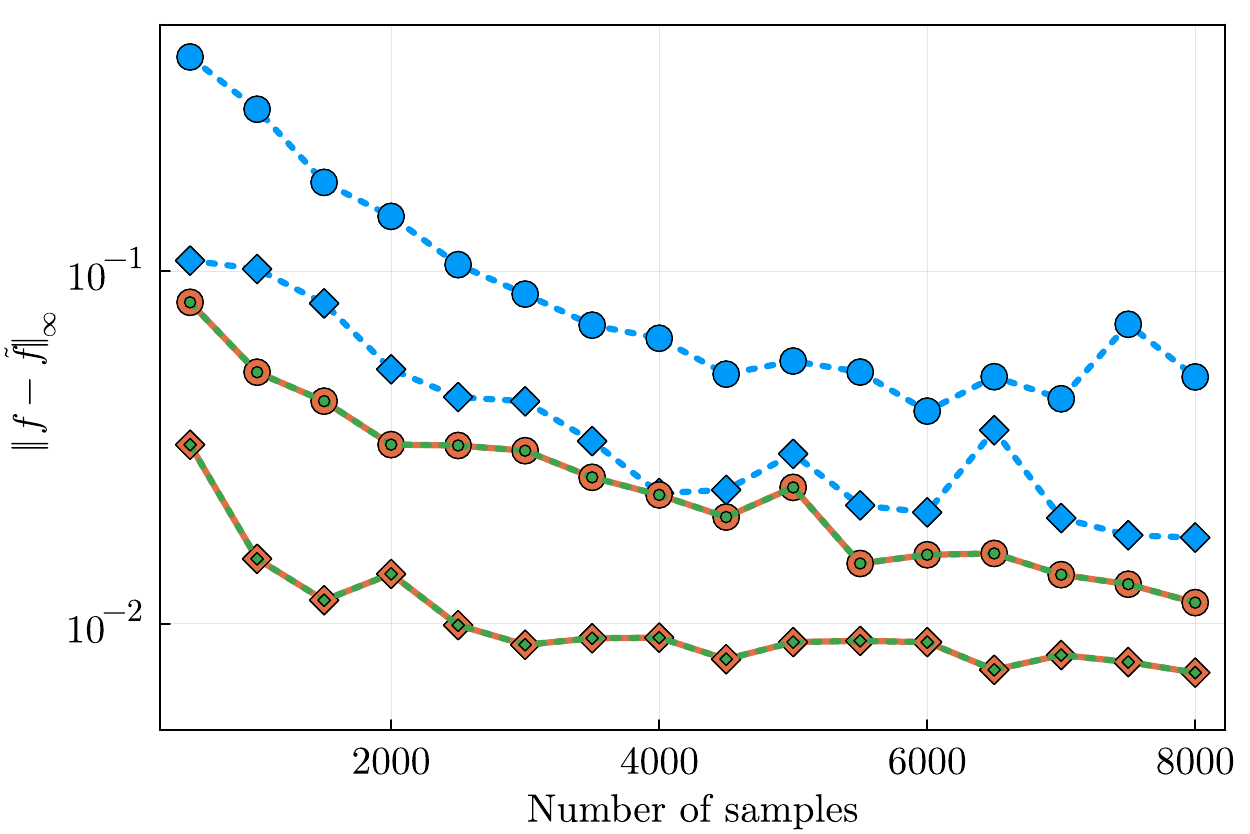}
         \caption{$f_{5}$, uniform distribution}
     \end{subfigure}
     \caption{Maximum error when approximating the functions $f_{5}$ in \eqref{eq:testfunctioncoeffs1} with the indicated basis functions of total degree 30 using the uniform distribution with identical setting as in Figure \ref{fig:cheblegendre+uniform+cheb}.}
     \label{fig:f12_supnorm}
\end{figure}

\begin{figure}[h!]
     \centering
     \captionsetup[subfigure]{justification=centering}
     \begin{subfigure}[b]{0.32\textwidth}
     \centering
         \includegraphics[width=\textwidth]{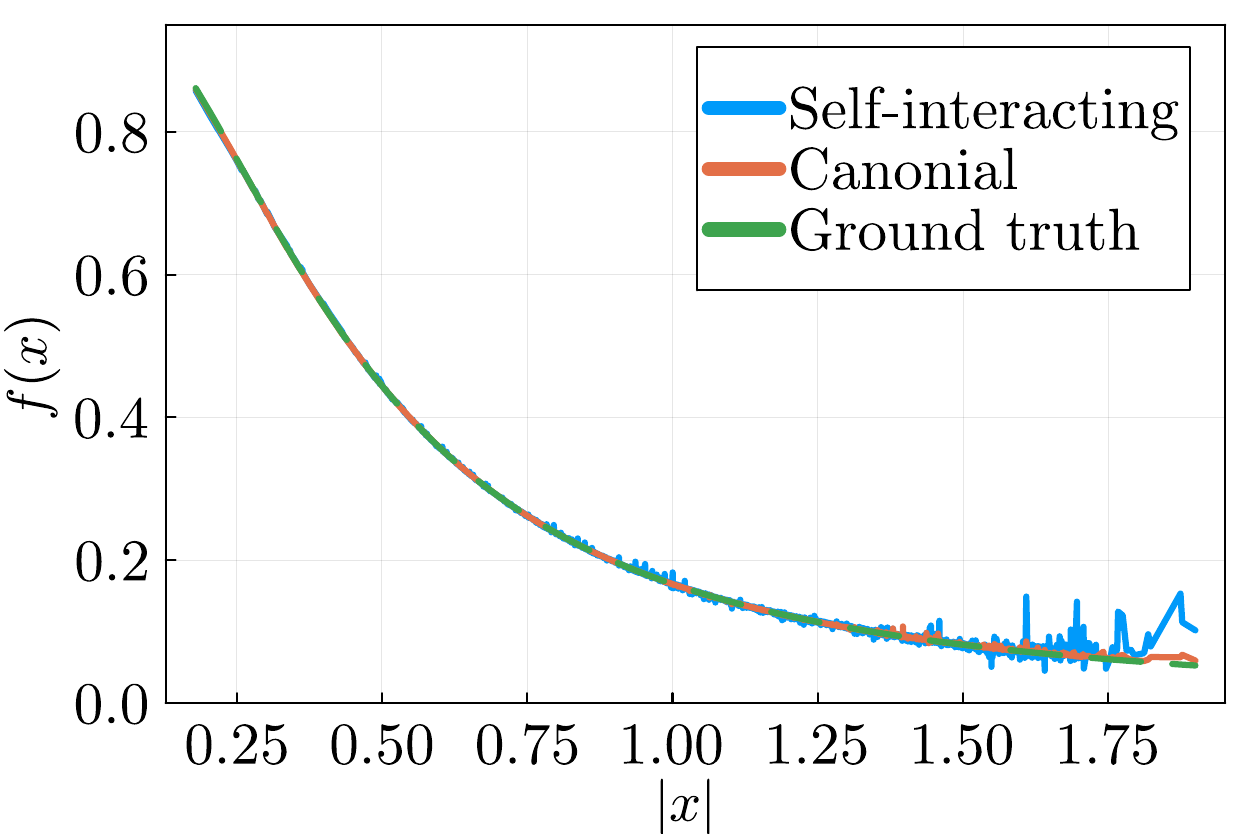}
         \caption{$f_{5}$ with $N = J = 4$\\\hspace{5mm} and smoothness prior}
         \label{fig:normxvsf-smooth-legendre}
     \end{subfigure}
     \begin{subfigure}[b]{0.32\textwidth}
     \centering
         \includegraphics[width=\textwidth]{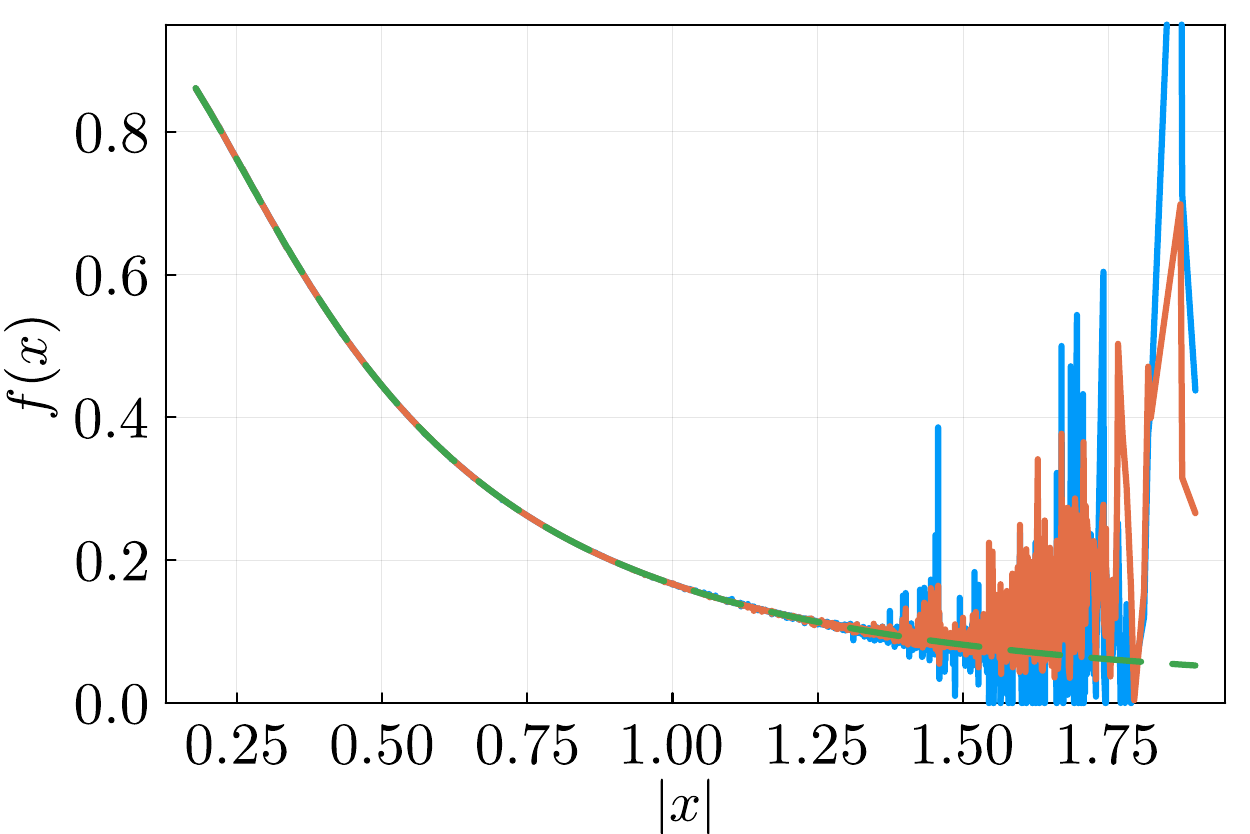}
         \caption{$f_{5}$ with $N = J = 4$\\ and identity prior}
         \label{fig:normxvsf-identity-legendre}
     \end{subfigure}
      \begin{subfigure}[b]{0.32\textwidth}
      \centering
         \includegraphics[width=\textwidth]{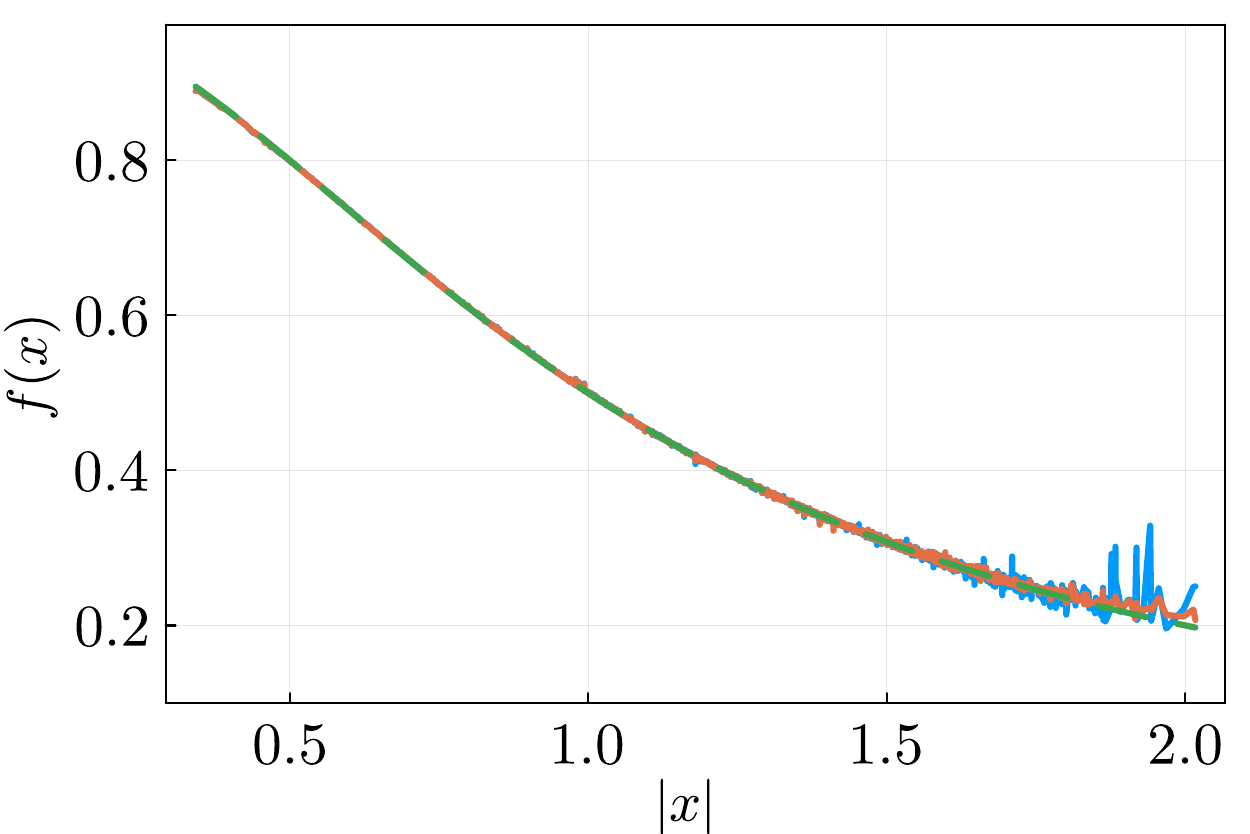}
         \caption{$f_{1}$ with $N = 4$, $J = 5$\\ \hspace{5mm} and smoothness prior}
         \label{fig:normxvsf-JgtN}
     \end{subfigure}
     \begin{subfigure}[b]{0.32\textwidth}
     \centering
         \includegraphics[width=\textwidth]{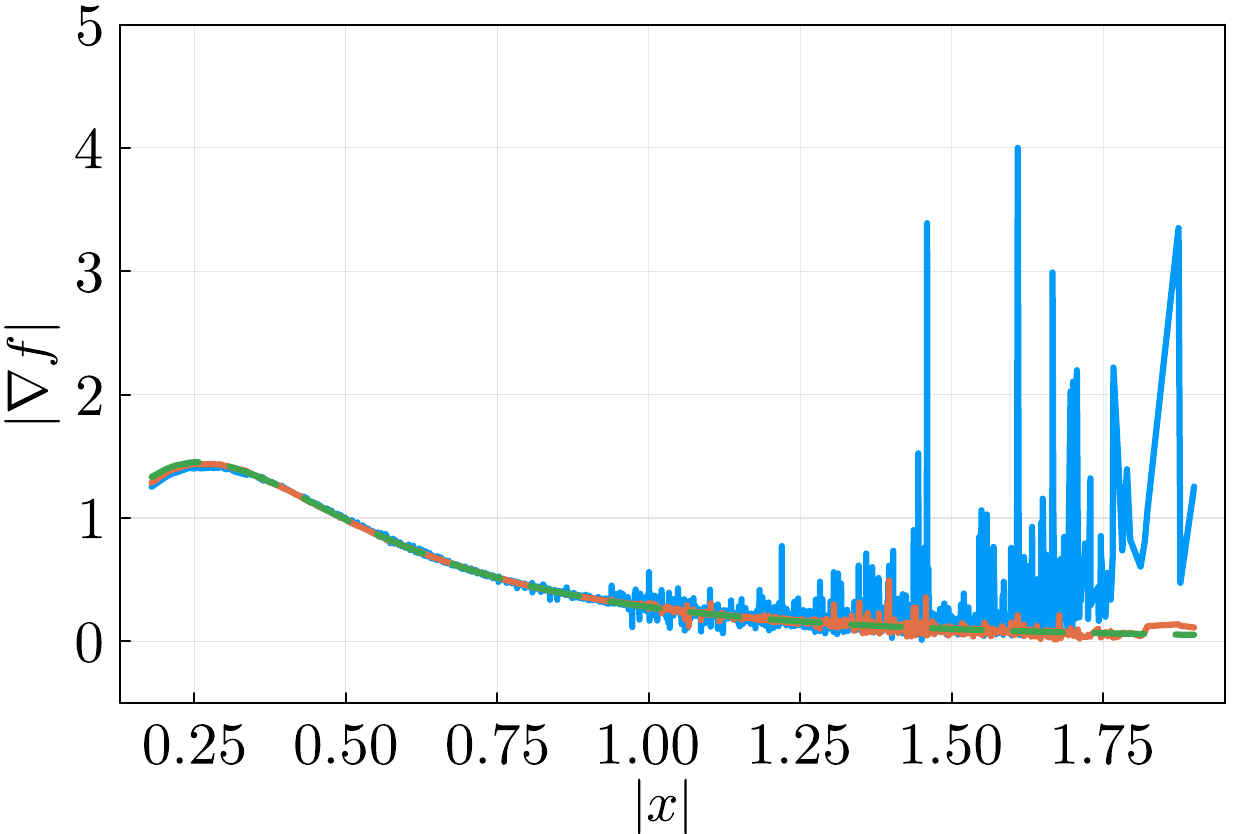}
         \caption{$f_{5}$ with $N = J = 4$\\\hspace{5mm} and smoothness prior}
         \label{fig:normxvsgradf-smooth-legendre}
     \end{subfigure}
     \begin{subfigure}[b]{0.32\textwidth}
     \centering
         \includegraphics[width=\textwidth]{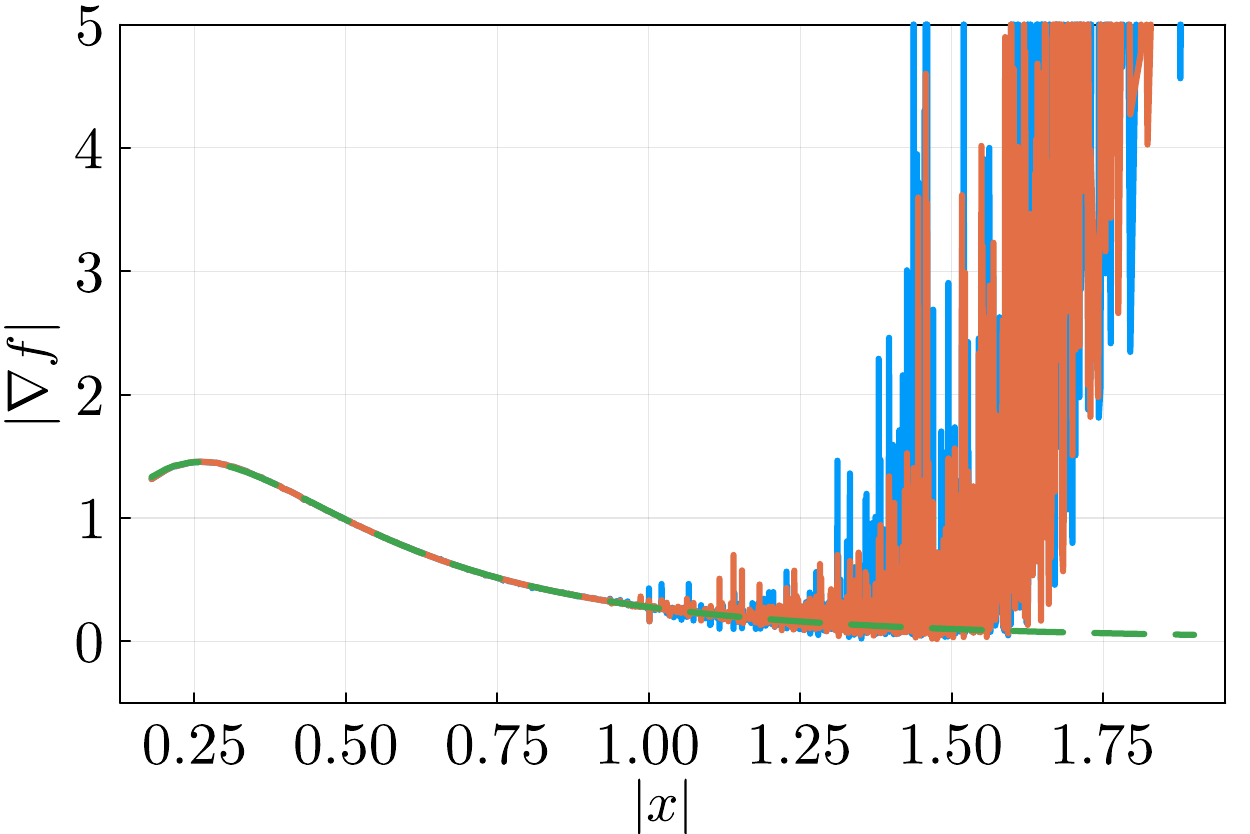}
         \caption{$f_{5}$ with $N = J = 4$\\ and identity prior}
         \label{fig:normxvsgradf-identity-legendre}
     \end{subfigure}
      \begin{subfigure}[b]{0.32\textwidth}
      \centering
         \includegraphics[width=\textwidth]{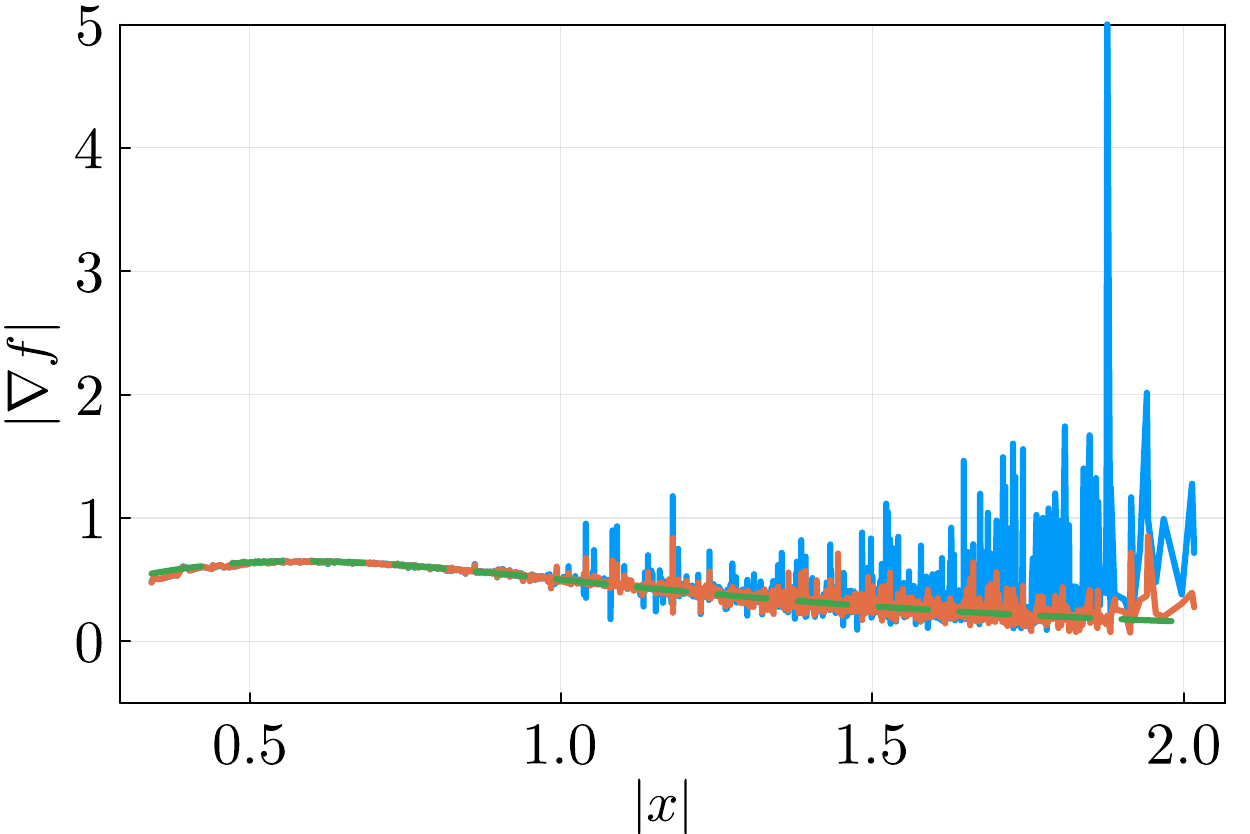}
         \caption{$f_{1}$ with $N = 4$, $J = 5$\\ \hspace{5mm} and smoothness prior}
         \label{fig:normxvsgradf-JgtN}
     \end{subfigure}
     \caption{Radial plots comparing approximations in the canonical and self-interacting bases with the approximated function and its gradient in different settings. Comparison shows the superior effect of the smoothness prior for the canonical basis as well as better performance of the smoothness prior compared to a simpler $\Gamma = I$ identity prior. We use $f_{1}$ in (c) and (f) instead of $f_{5}$ to reduce the difficulty of the regression and obtain a meaningfully accurate fit.} 
    \label{fig:cheblegendre+uniform+cheb-normxvsf}
\end{figure}

\begin{figure}[h!]
     \centering
     \captionsetup[subfigure]{justification=centering}
     \begin{subfigure}[b]{0.32\textwidth}
         \centering
         \includegraphics[width=\textwidth]{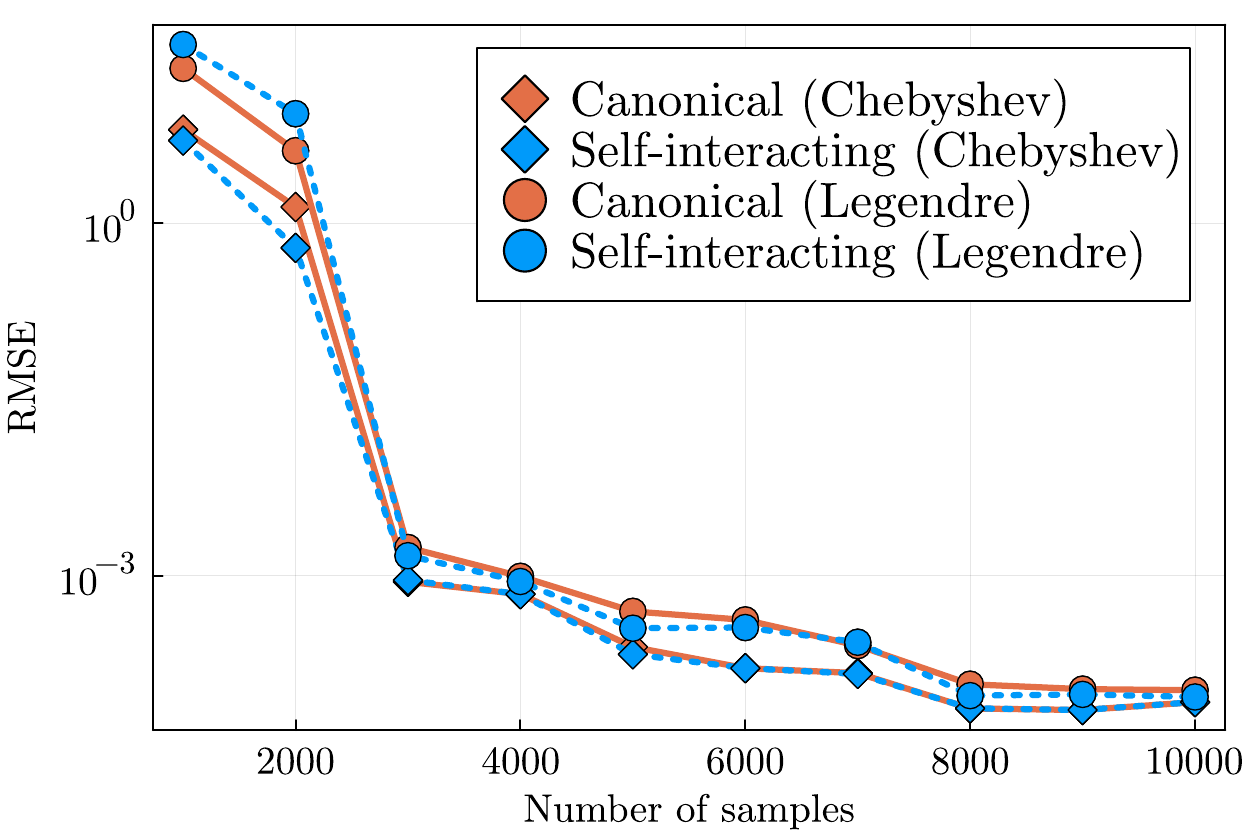}
         \caption{$F_{1, 0}$}
     \end{subfigure}
     \hfill
     \begin{subfigure}[b]{0.32\textwidth}
         \centering
         \includegraphics[width=\textwidth]{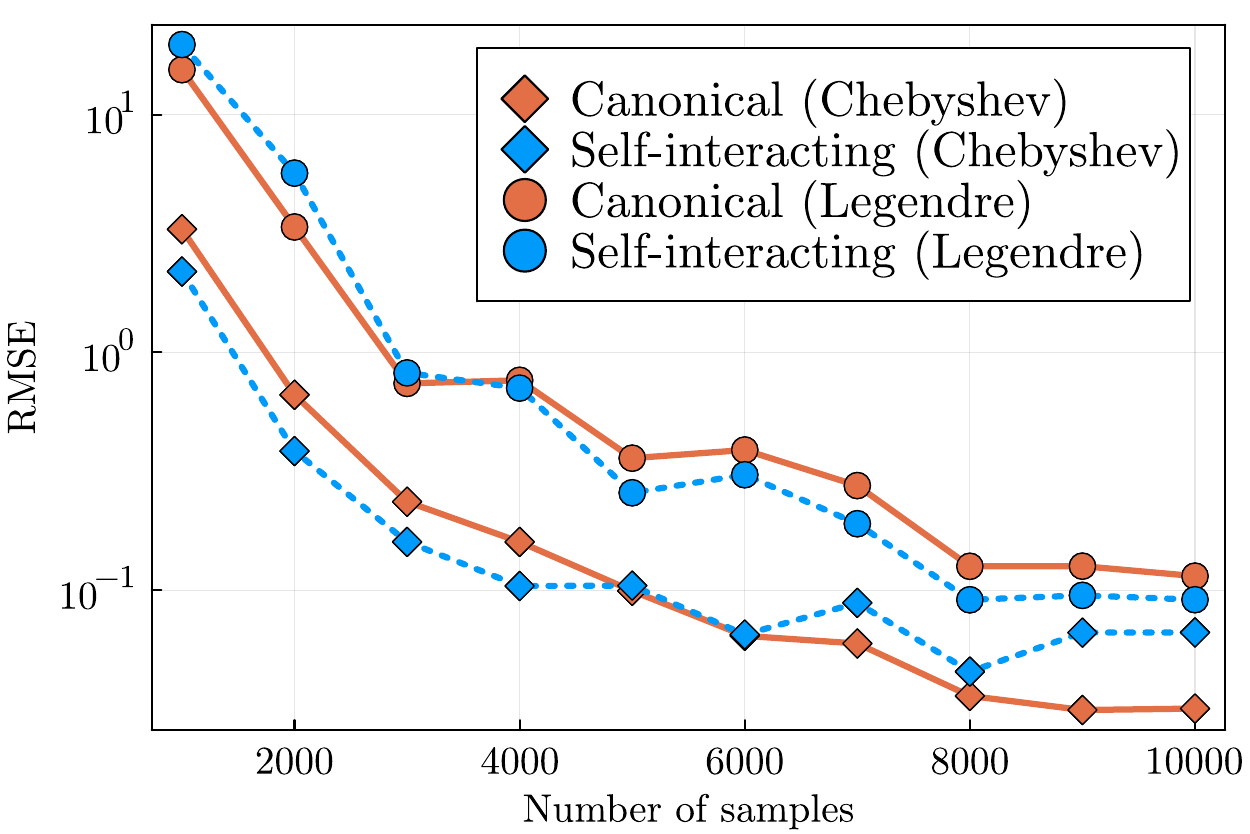}
         \caption{$F_{5, 0}$}
     \end{subfigure}
     \hfill
      \begin{subfigure}[b]{0.32\textwidth}
         \centering
         \includegraphics[width=\textwidth]{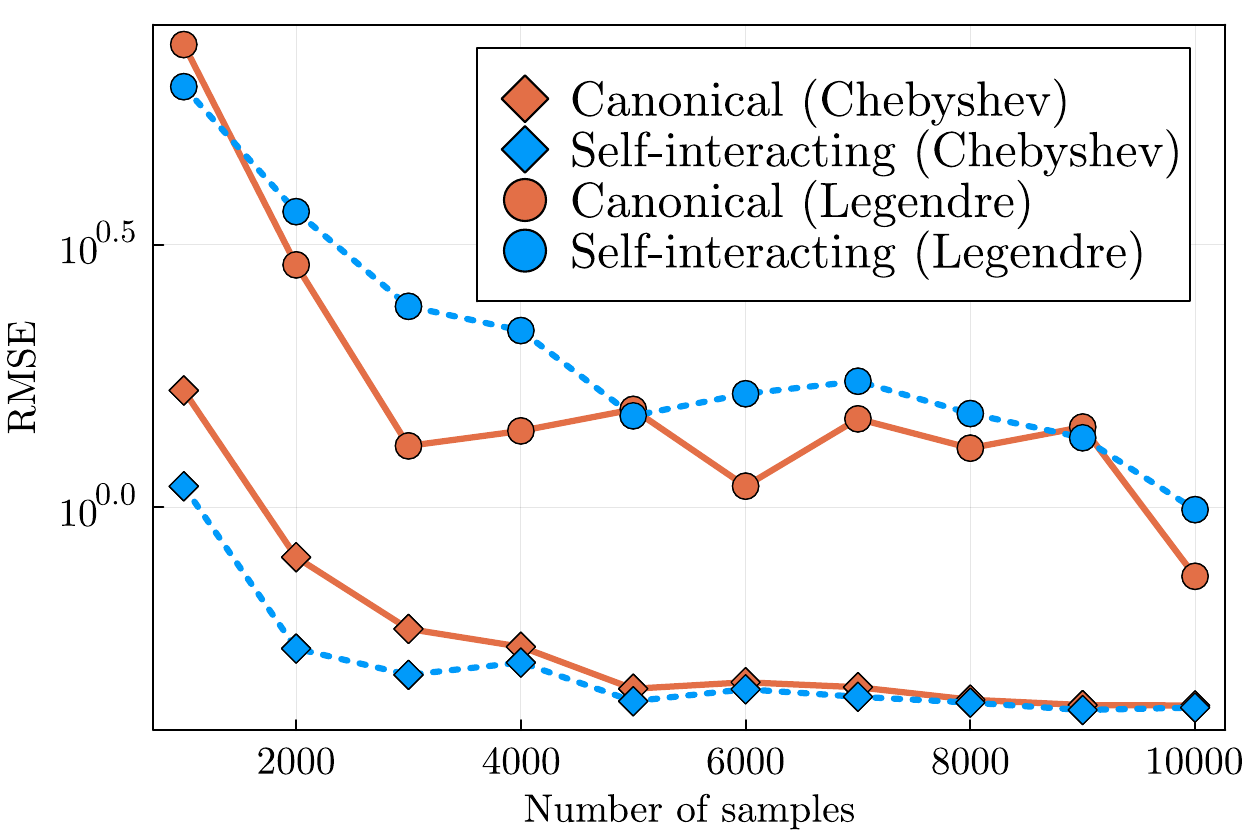}
         \caption{$F_{25, 0}$}
     \end{subfigure}
     \hfill
      \begin{subfigure}[b]{0.32\textwidth}
         \centering
         \includegraphics[width=\textwidth]{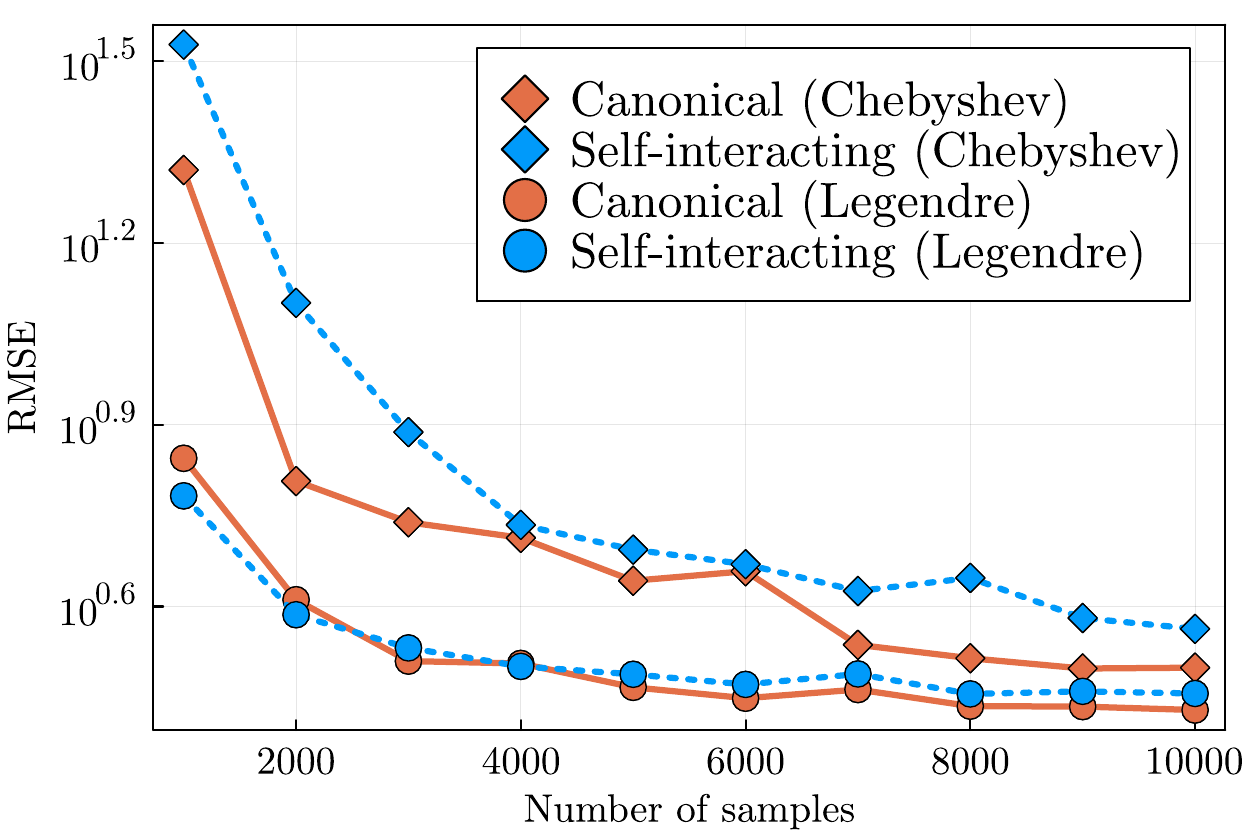}
         \caption{$F_{1, 0.1}$}
     \end{subfigure}
     \hfill
      \begin{subfigure}[b]{0.32\textwidth}
         \centering
         \includegraphics[width=\textwidth]{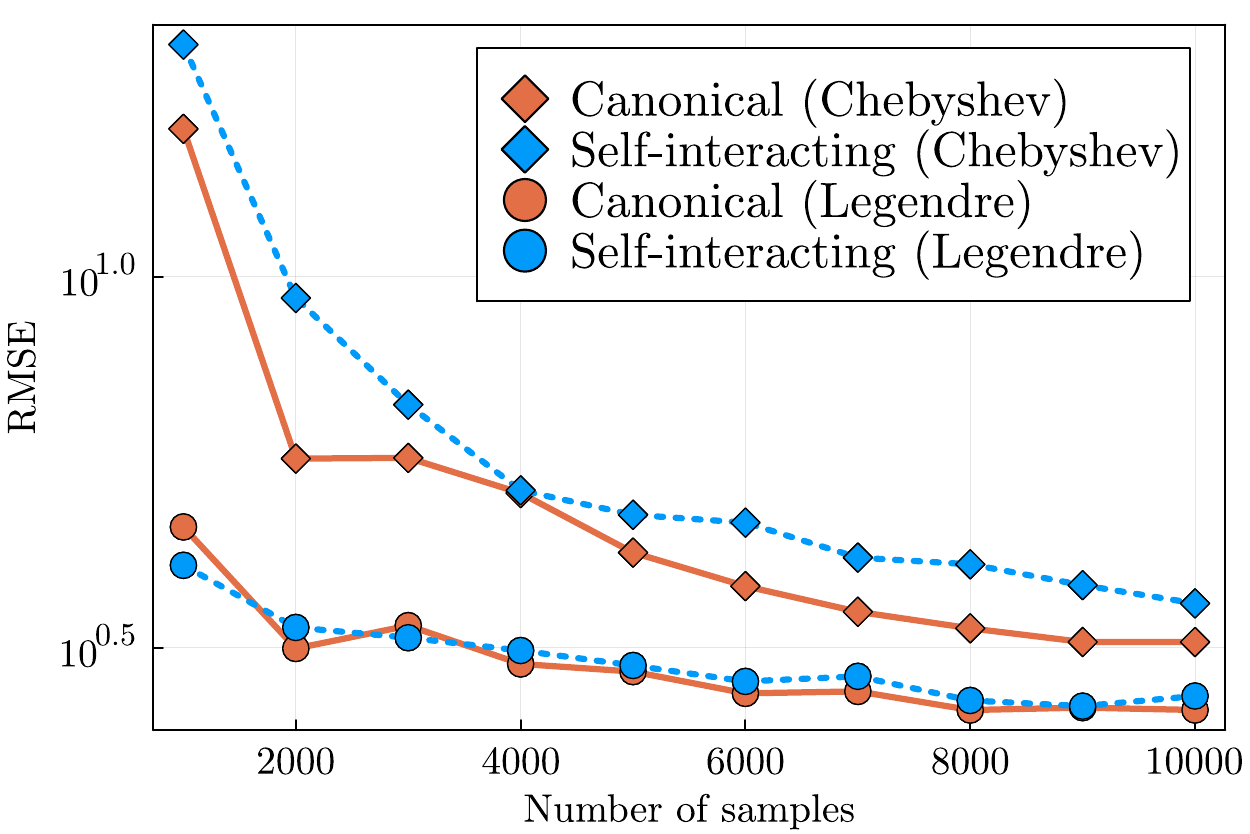}
         \caption{$F_{5, 0.1}$}
     \end{subfigure}
     \hfill
      \begin{subfigure}[b]{0.32\textwidth}
         \centering
         \includegraphics[width=\textwidth]{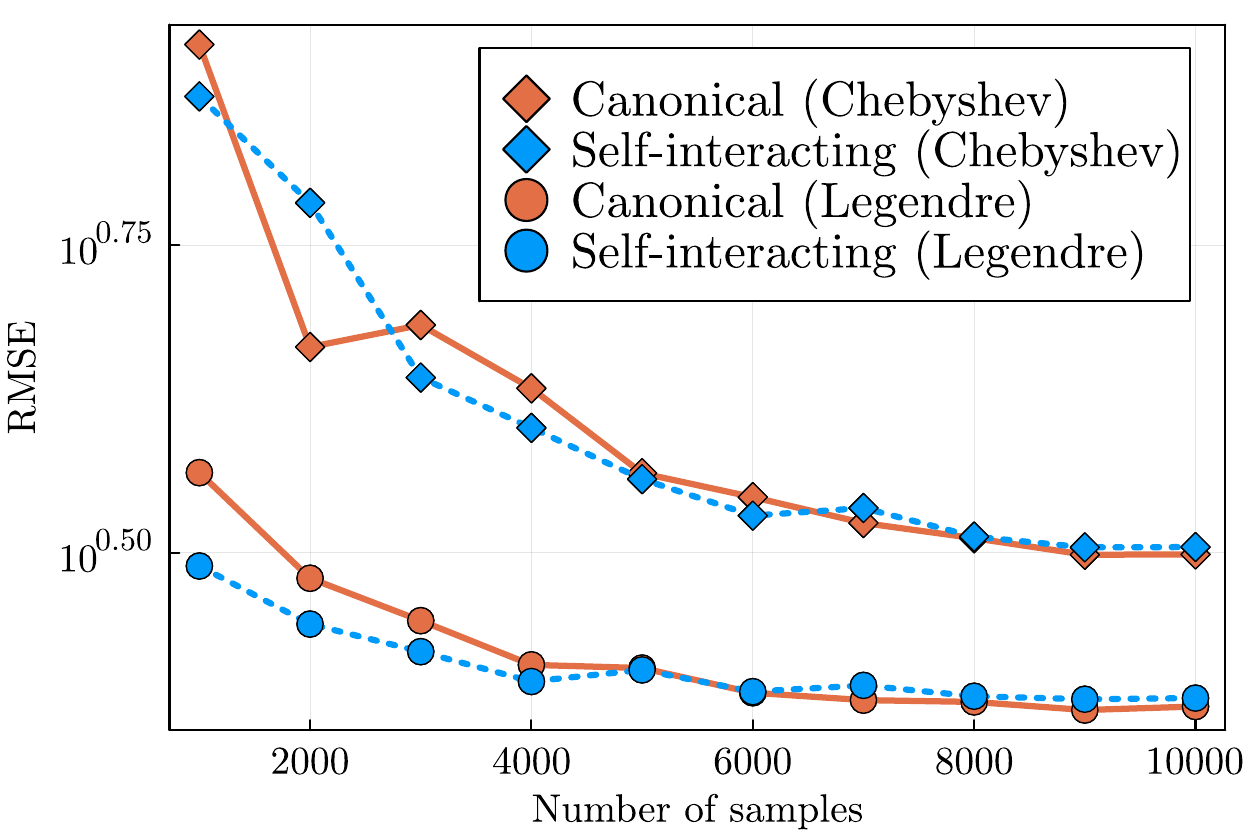}
         \caption{$F_{25, 0.1}$}
     \end{subfigure}
     \caption{RMSE when approximating the function $F_{a, \epsilon}$ in $\eqref{eq:J>Nsumpermtesting}$ with indicated $a$ and $\varepsilon$ using the uniform distribution with the indicated basis functions of total degree 30.}
    \label{fig:JgtN_toyregression}
\end{figure}

\clearpage 

\section{Application to machine learning interatomic potentials}
\label{section:MLIPs}
The original and still primary application of the (self-interacting) atomic cluster expansion has been the parameterization of machine learning interatomic potentials (MLIPs) \cite{2023-acepotentials,DrautzACE,vandermause2020flyal,ACECompleteness}.
In this section we demonstrate the advantages of replacing the standard self-interacting ACE formulation with the canonical cluster expansion in this application.  

\subsection{Review of ACE for MLIPs}
The task is the parameterization of the potential energy $E(\msl \br_i, Z_i \msr_{i = 1}^{N_{\rm at}})$ of an atomic structure $\msl \br_i, Z_i \msr_{i = 1}^{N_{\rm at}} \subset \mathbb{R}^{3} \times \mathbb{Z}$, where $\br_i$ and $Z_i$ are, respectively, the position and atomic number of atom $i$. For the sake of simplicity, we will only consider elemental systems, hence we ignore the atomic number; an atomic structure then becomes simply a point cloud $\msl \br_i \msr_{i = 1}^{N_{\rm at}} \subset \mathbb{R}^3$. Furthermore, we will ignore periodic boundary conditions which are commonly employed in this setting. (Although they do appear in the datasets we use below, they are unimportant to explain the model parameterization via ACE.)

Within the ACE framework the potential energy surface is not parameterized directly but it is first decomposed into site energies, 
\[
    E\big(\{\br_i\}_{i =1}^{N_{\rm at}} \big) 
        = 
    \sum_{i = 1}^{N_{\rm at}} \varepsilon\big( \mslb \br_{ij} \msrb_{j \sim i} \big) 
\]
where $\varepsilon$ is called the {\it site energy potential}, $\br_{ij} = \br_j - \br_i$ and $j \sim i$ if $|\br_{ij}| < r_{\rm cut}$. The cutoff radius $r_{\rm cut} > 0$ is a model hyperparameter which we take as fixed. The spatial decomposition into site energies can be thought of as a physical constraint, or prior~\cite{chen2016qm, 2019-tbloc0T}. The reason to parameterize $\varepsilon$ instead of $E$ directly is size-extensivity of the model. We can guarantee  $E(3)$ invariance (Euclidean group) of $E$ by enforcing $O(3)$ invariance of $\varepsilon$.

Since the distance between atoms changes during a simulation, the atomic environment $\msl\br_{ij}\msr_{j \sim i}$ of an atom $i$ is a multi-set. Setting $\Omega \subset \R^3$ to be a ball with radius $r_{\rm cut}$ and $x_j := \br_{ij}$ brings us into the general abstract context of the foregoing sections to parameterize $\varepsilon$ via the canonical or self-interacting $O(3)$-invariant cluster expansion. To specify the parameterization of $\varepsilon$ we will closely (but not exactly) follow the construction in Section~\ref{section:G=O(3)}. 

There is a small change in the construction of the one-particle basis due to the presence of the cutoff radius $r_{\rm env}$. 
Following  \cite{ACECompleteness, 2023-acepotentials}, to guarantee smoothness of $\varepsilon$ as atoms enter and exit the atomic neighbourhood, we introduce an envelope function $f_{\rm cut}$. The radial basis $R_n$ in \eqref{equa:O(3)1pbasis} is thus replaced by 
\begin{equation}
\label{equa:MLIP_1pbasis}
    R_n(r_{ij}) = f_{\text{env}}(y(r_{ij}))P_n(y(r_{ij})),
\end{equation}
where $y$ is a coordinate transform, $P_n$ is an orthogonal basis in $y$-coordinate with weight $f_{\text{env}}(y(r_{ij}))$ and $f_{\text{env}}(y(r_{ij}))$ is a polynomial that vanishes smoothly outside $[0, r_{\text{cut}}]$. By choosing $f_{\rm env}$ to be a polynomial (normally at most quartic) allows exact linearization of $\phi_{n_1 l_1 m_1}\phi_{l_2 m_2 l_2}$, which is required for the construction of the purification operator. Further details of the radial basis construction are given in Appendix \ref{Appendix:ModelConstruction} and in \cite{2023-acepotentials}, which also gives a more in-depth explanation for the choice of radial basis.

After constructing the one particle basis, the general ACE framework outlined in the previous sections can be used to construct the $O(3)$-invariant basis ${\bf B}(\msl \br_{ij} \msr_{j \sim i})$ and the resulting parameterization of the site energy potential, 
\begin{align*}
    \varepsilon\big({\bf c}; \msl \br_{ij} \msr_{j \sim i}\big)
    = {\bf c} \cdot {\bf B}(\msl \br_{ij} \msr_{j \sim i}), \qquad \text{or,} \qquad  
    \varepsilon\big({\bf c}; \msl \br_{ij} \msr_{j \sim i}\big)
    = {\bf c} \cdot {\cal B}(\msl \br_{ij} \msr_{j \sim i}). 
\end{align*}
A key point to note here is that, due to the addition of an envelope function, the expansion on the radial part is not total degree preserving (cf. Proposition \ref{proposition:preserving}), and hence two models (canonical and self-interacting) may be no longer equivalent. We refer to Appendix \ref{Appendix:ExtraBasis} for a discussion of the effect of the envelope function on the spanning space of the two expansions.

From $\varepsilon$ we can evaluate total energies $E$, and hence also forces and potentially other observations. Training datasets are {\em not} provided in the form of atomic environment and site energy pairs, but in the form of a list of atomic structures, and corresponding target total energies and forces obtained from ab-initio level calculations (Kohn-Sham density function theory in all our examples). 
The parameters ${\bf c}$ are estimated by minimizing a mean square cost function matching all available observations. Since energies and forces are linear functions of the parameters ${\bf c}$, this results in a linear regression problem. There are many different methods to solve such a linear least squares problem. 
We will not use the very general Tikhonov formulation in \eqref{eq:genericRegLSQ} and instead rescale the design matrix $\Psi$ by $\Gamma^{-1}$ to enforce an asymptotic coefficient decay to obtain:
\begin{align*}
\min_{\bf b} \left \| \Psi \Gamma^{-1}\mathbf{b} - \mathbf{y} \right \|^2,
\end{align*}
which we solve for given relative tolerances via a truncated SVD approach, after which we revert the scaling via $\mathbf{c} = \Gamma^{-1} \mathbf{b}$. 
All the experiments are performed using the {\tt ACEpotentials.jl} code~\cite{2023-acepotentials}. We refer again to \cite{2023-acepotentials, ACECompleteness} for further details and further discussion of the parameter estimation. 

The canonical parameterization ${\bf c} \cdot {\cal B}(\msl \br_{ij} \msr_{j \sim i})$ has the {\em potential advantage} over the self-interacting parameterization that the energy contribution from every order is separated as clearly seen in \eqref{equa:CanonicalExpansion}. However, since the two expansions span the same or nearly the same space, it is unclear whether this chemically intuitive property translates into actual benefits in accuracy or qualitative properties of the fitted site energy potential $\varepsilon$. Our experiments in the following section are a preliminary investigation of this question. 
We perform this analysis on two data sets: the ``Zou {\it et al} (2020) data set'' taken from an established MLIPs benchmark \cite{Zuo2020} and a more recent and more challenging ``iron data set'' published in \cite{zhang2020Fe}.

\subsection{Zuo et al (2020) data set}
\label{section:Zuo}
The Zou {\it et al} (2020) data set \cite{Zuo2020} contains data sets for six elements, Li, Mo, Ni, Cu, Si and Ge, spanning various crystal structures (bcc, fcc, and diamond) and bonding types (metallic and covalent). It is a relatively large dataset but with limited diversities in terms of its coverage of the configuration space for each of these materials. Test and training data splits are provided explicitly and are not generated randomly. In previous tests~\cite{2023-acepotentials}, we found the Mo dataset to be the most challenging, hence we use this for our tests, but similar observations were made on the other materials. A self-interacting and canonical model using approximately 1200 basis functions is used (See \ref{Appendix:hyperparameters_zuo} for details) to construct the design matrix. 

In our first test, we explore the data efficiency of the two parameterizations. To that end, we plot the learning curves for both total energy and forces in Figure~\ref{fig:Mo_svd_rmse}, for different regularization parameters. The canonical model achieves only slightly better force errors and similar energy errors as the self-interacting model. However, the canonical model is more robust in this test in two important ways: first, it is capable of achieving the best accuracy simultaneously for energy and forces; and secondly, its convergence is more stable under changes in the regularization parameter.

\begin{figure}[h!]
\captionsetup[subfigure]{justification=centering}
     \begin{subfigure}[b]{0.49\textwidth}
         \centering
        \includegraphics[width=\textwidth]{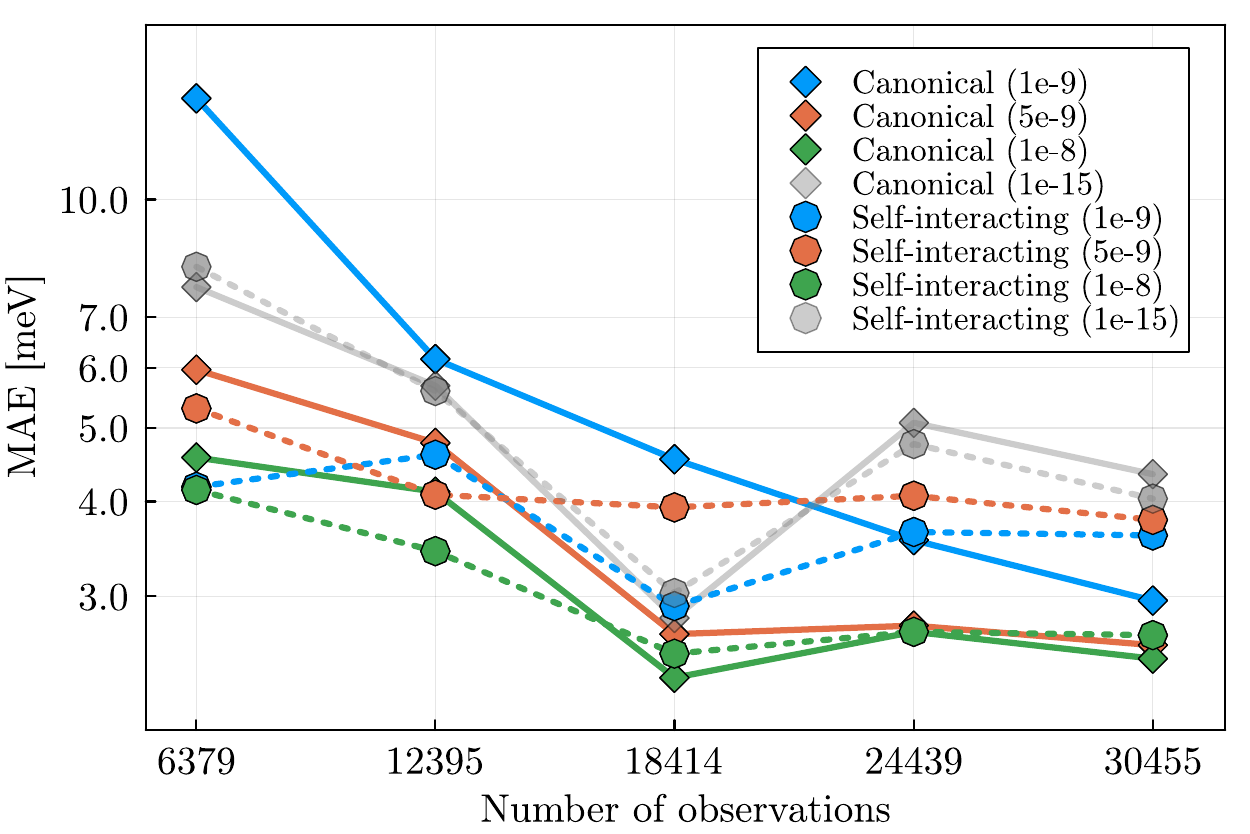}
        \caption{Predicted energy error}
     \end{subfigure}
     \hfill
     \begin{subfigure}[b]{0.49\textwidth}
        \centering
        \includegraphics[width=\textwidth]{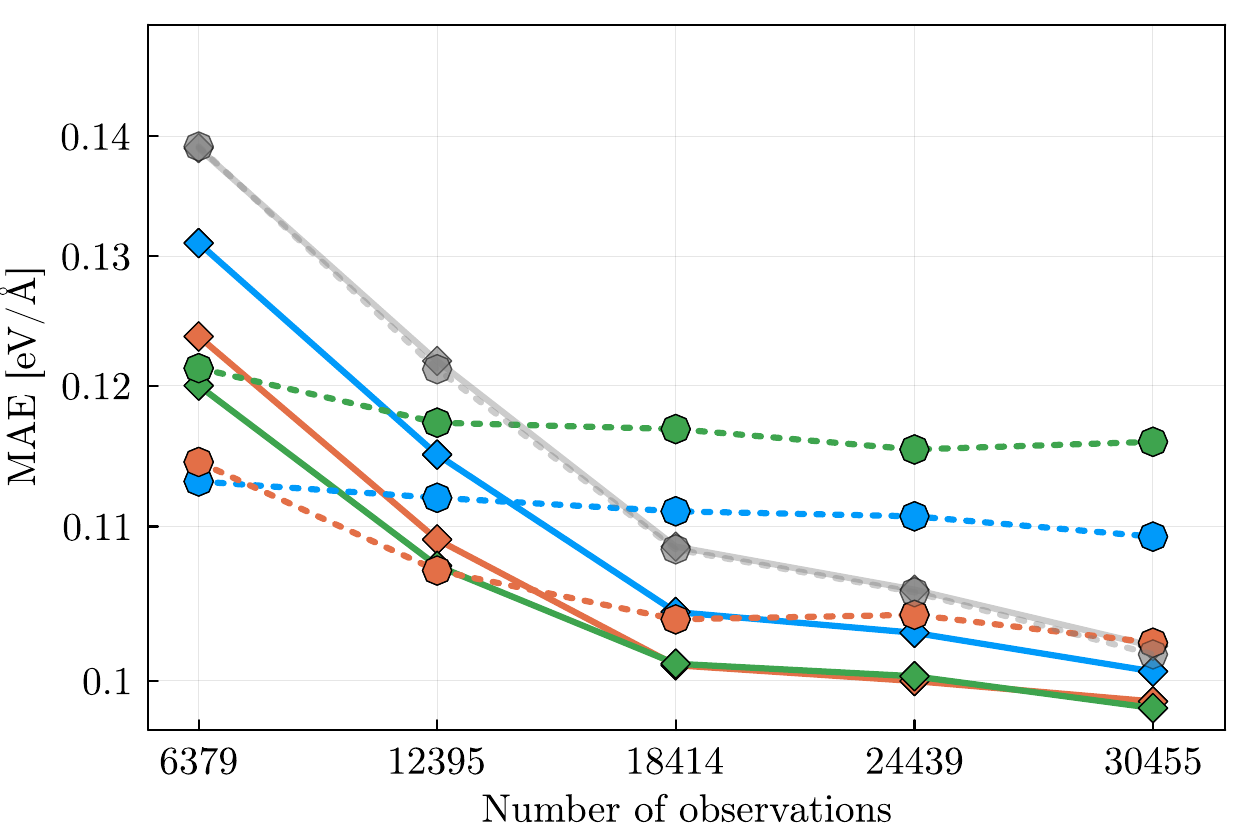}
        \caption{Predicted force error}
     \end{subfigure}
     \caption{Lerning curves for canonical and self-interacting ACE models for the Mo data set of Zuo et al~\cite{Zuo2020}. We use MAE instead of RMSE for consistency with \cite{Zuo2020}.}
     \label{fig:Mo_svd_rmse}
\end{figure}

The second part of our exploration concerns qualitative properties. A common test~\cite{PACE,2023-acepotentials} is to visually inspect the regularity of the dimer curve, since it is a one-dimensional but chemically important quantity. Empirically, the regularity and qualitatively convex-concave shape of the dimer curve has been found to be directly related to the stability of molecular dynamics simulations. 
In Figure~\ref{fig:Mo_svd_dimer} we plot the dimer curves (and their derivatives) for the canonical and self-interacting ACE models, both in a low-data and large-data regime, and for different choices of regularization parameters. With minimal regularization, even though both models have acceptable energy and force errors on the test set, their dimer shapes are highly oscillatory which is unacceptable in simulations~\cite{PACE,2023-acepotentials}.
With sufficient regularization, both models provide qualitatively acceptable dimer curves. However the canonical model shows markedly smoother results (most clearly visible in subfigures c, d) and in the high-data regime the dimer energy minimum is an accurate representation of the ground state bondlength. 

\begin{figure}[h!]
\captionsetup[subfigure]{justification=centering}
     \begin{subfigure}[b]{0.49\textwidth}
         \centering
        \includegraphics[width=\textwidth]{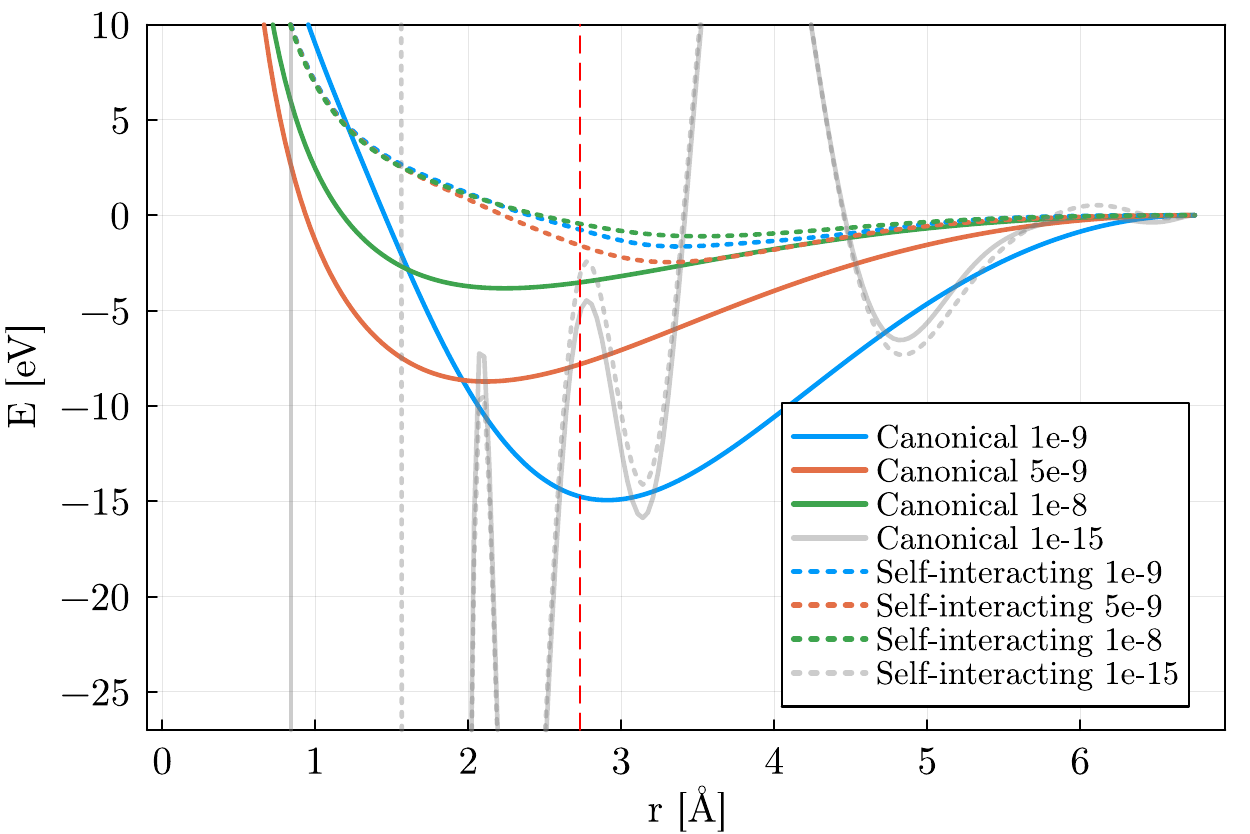}
        \caption{Number of observations: 6379 (20\%)}
     \end{subfigure}
     \hfill
     \begin{subfigure}[b]{0.49\textwidth}
        \centering
        \includegraphics[width=\textwidth]{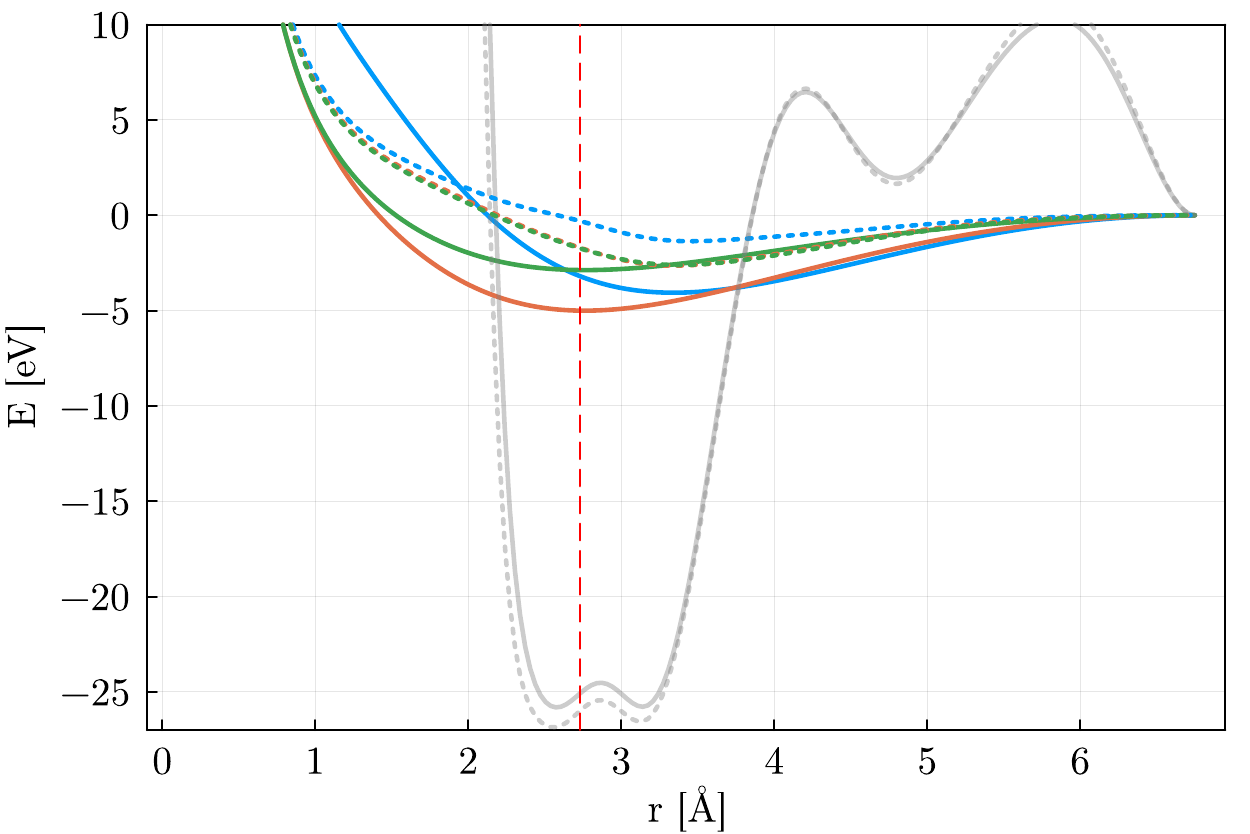}
        \caption{Number of observations: 30455 (full)}
     \end{subfigure}
      \hfill
     \begin{subfigure}[b]{0.49\textwidth}
        \centering
        \includegraphics[width=\textwidth]{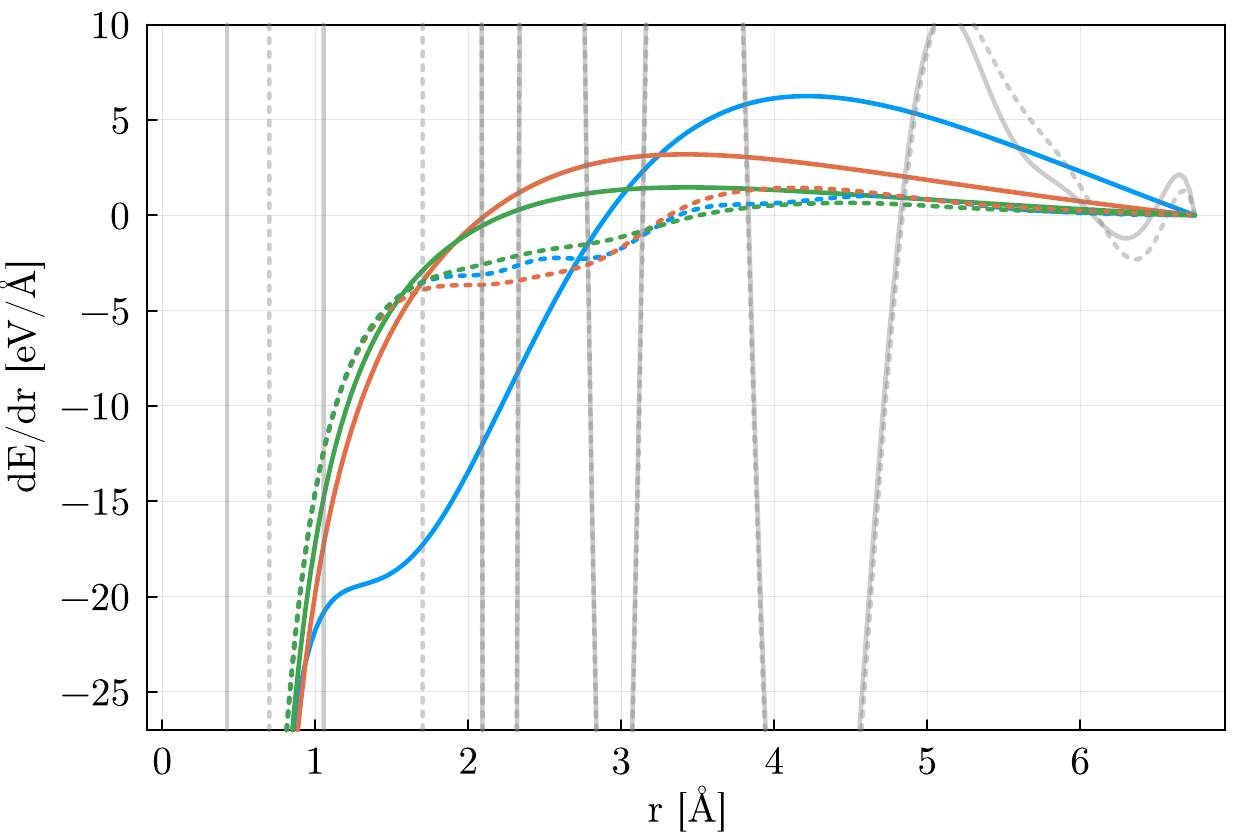}
        \caption{Number of observations: 6379 (20\%) }
     \end{subfigure}
      \begin{subfigure}[b]{0.49\textwidth}
        \centering
        \includegraphics[width=\textwidth]{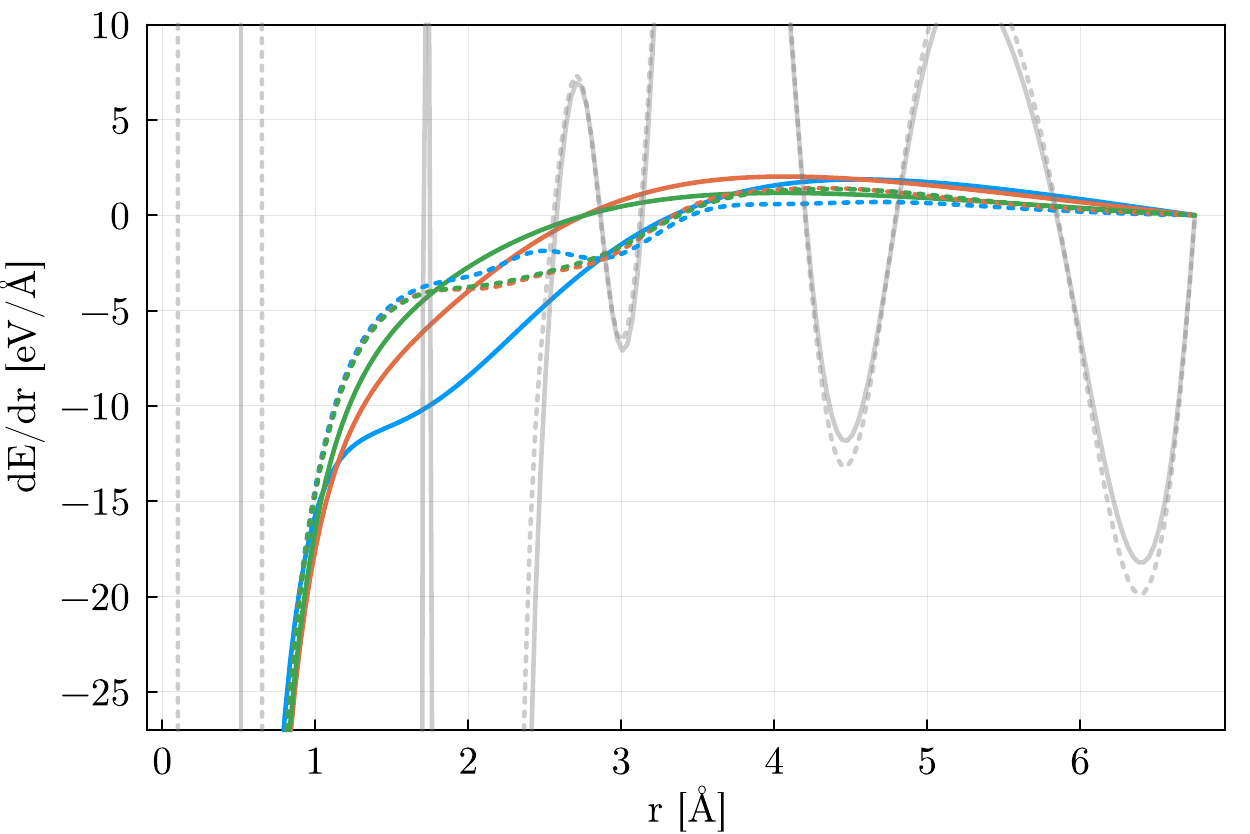}
        \caption{Number of observations: 30455 (full) }
     \end{subfigure}
     \caption{Plots of the dimer and dimer derivative of Mo for the indicated number of observations. The red vertical line indicates the nearest neighbour distance of Mo; A key observation is that the dimer of the canonical model is smooth, repulsive and attains its minimum close to the nearest neighbour distance, while the self-interacting model is more oscillatory.}
     \label{fig:Mo_svd_dimer}
\end{figure}

\subsection{Fe data set}
\label{section:Fedata}
Another example we consider is a Fe data set constructed via active learning with the Gaussian Approximation Potential (GAP) \cite{zhang2020Fe}. The purpose of the GAP model was to simulate the deformation mechanisms of the crack-tip to demonstrate the semi-brittle nature of bcc iron. Such predictions involves a broad range of material behavior, hence the Fe dataset is very diverse and therefore challenging to fit. 
In the following experiment, we follow a similar set up as in the Zuo (2020) data set but using larger model sizes. We label a sequence of ACE models via {\em levels} 1 through 5; higher levels correspond to larger basis sizes.; see Appendix \ref{Appendix:hyperparameters_Fe} for details.  A weighted least squares system is set up using weights on individual observations exactly as described in \cite{zhang2020Fe}. 

Since the dataset is very diverse, but there are limited structures within each subgroup, it is more interesting to explore the convergence of the model accuracy as we increase the basis size rather than the number of observations. This is shown in Figure~\ref{fig:RMSE_Fe} for a range of regularization parameters. We observe that for all regularization parameters, the canonical model has significantly better accuracy, and the accuracy gap is particularly large for stronger regularization. This is a very promising result since it suggests that the canonical cluster expansion parameterization can be more strongly regularized (hence leading to better generalization) without losing too much accuracy. 

\begin{figure}[h!]
\captionsetup[subfigure]{justification=centering}
     \begin{subfigure}[b]{0.49\textwidth}
        \centering
        \includegraphics[width=\textwidth]{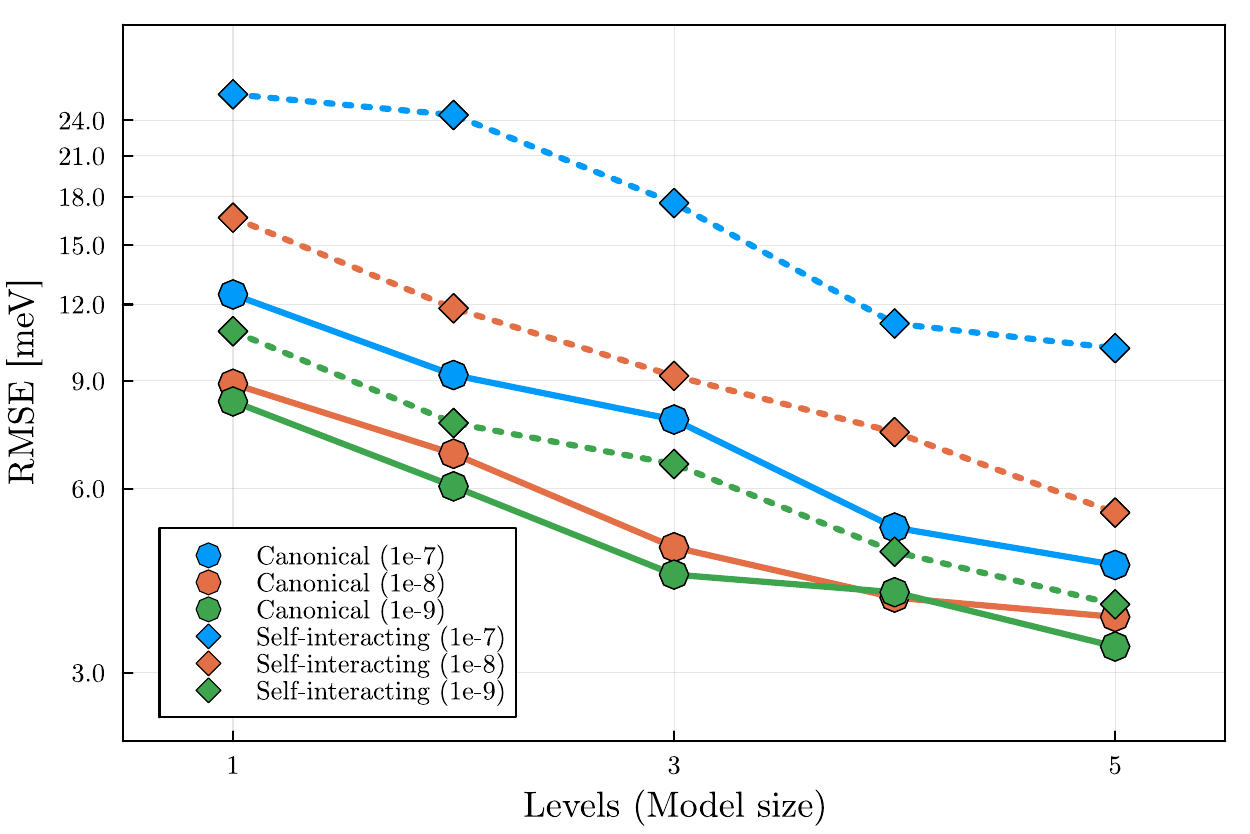}
        \caption{Predicted energy error}
     \end{subfigure}
     \hfill
     \begin{subfigure}[b]{0.49\textwidth}
        \centering
        \includegraphics[width=\textwidth]{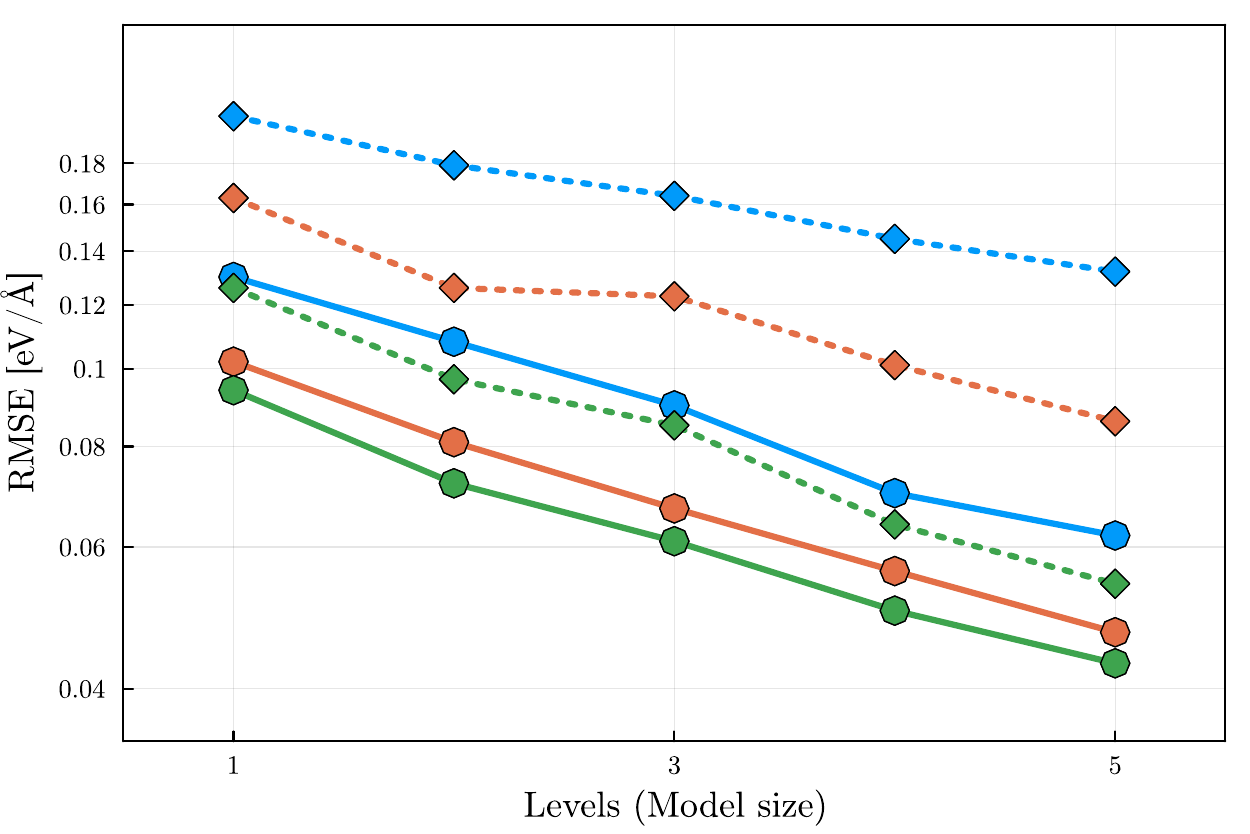}
        \caption{Predicted force error}
     \end{subfigure}
        \caption{Convergence plots of the RMSE with respect to model levels for total energy (a) and forces (b) for the Fe data set with the indicated relative tolerances in the truncated SVD.}
        \label{fig:RMSE_Fe}
\end{figure}

Secondly, we consider again the ``physicality'' of the fit by exploring whether the canonical and self-interacting models are able to recover a realistic dimer curve shape. In this test, both the canonical and self-interacting dimer curves are very smooth (likely due to the significantly more diverse dataset). Both models appear to converge stably to a chemically intuitive dimer curve shape as the regularization parameter increases. However, similarly as for the Zuo (2020) dataset, the canonical model still captures the potential energy minimum accurately while the self-interacting model does not.

\begin{figure}[h!]
\captionsetup[subfigure]{justification=centering}
      \begin{subfigure}[b]{0.33\textwidth}
        \centering
        \includegraphics[width=\textwidth]{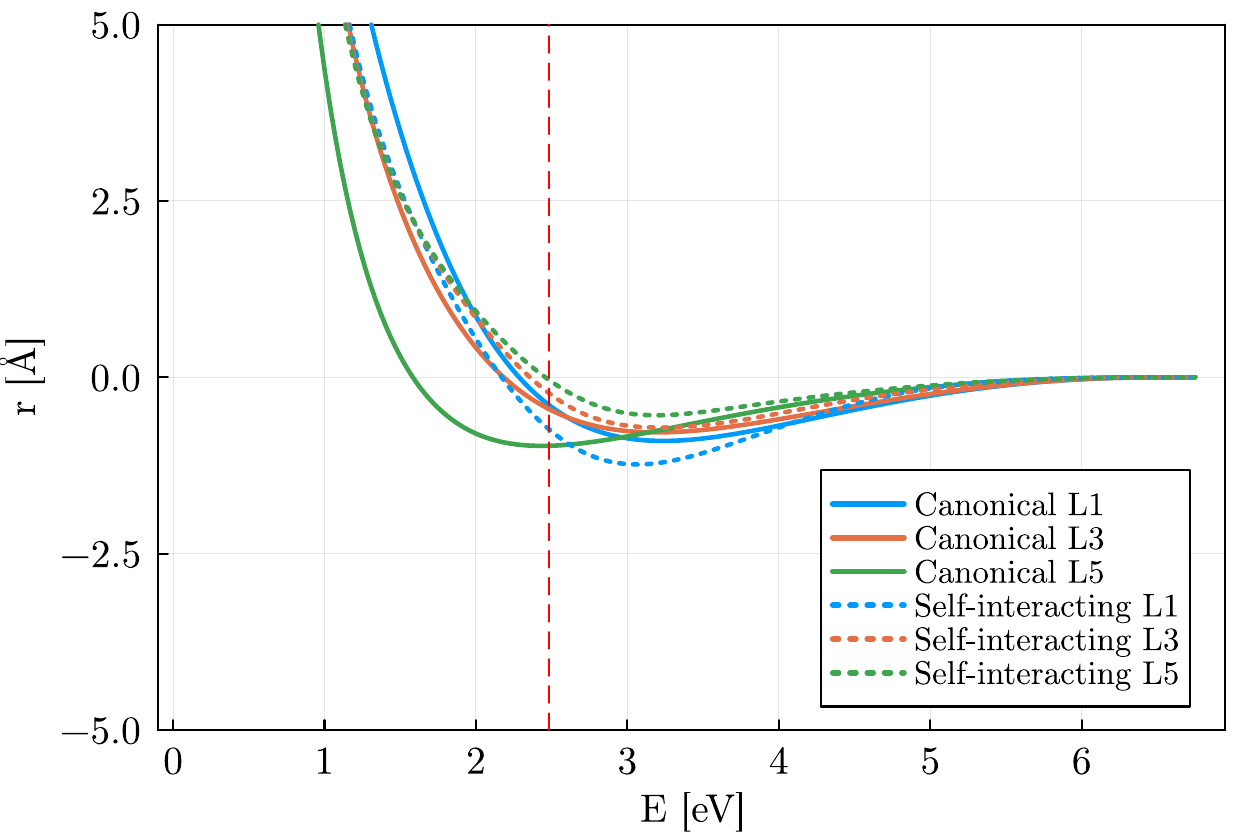}
        \caption{Relative tolerance: $10^{-8}$}
     \end{subfigure}
      \begin{subfigure}[b]{0.33\textwidth}
        \centering
        \includegraphics[width=\textwidth]{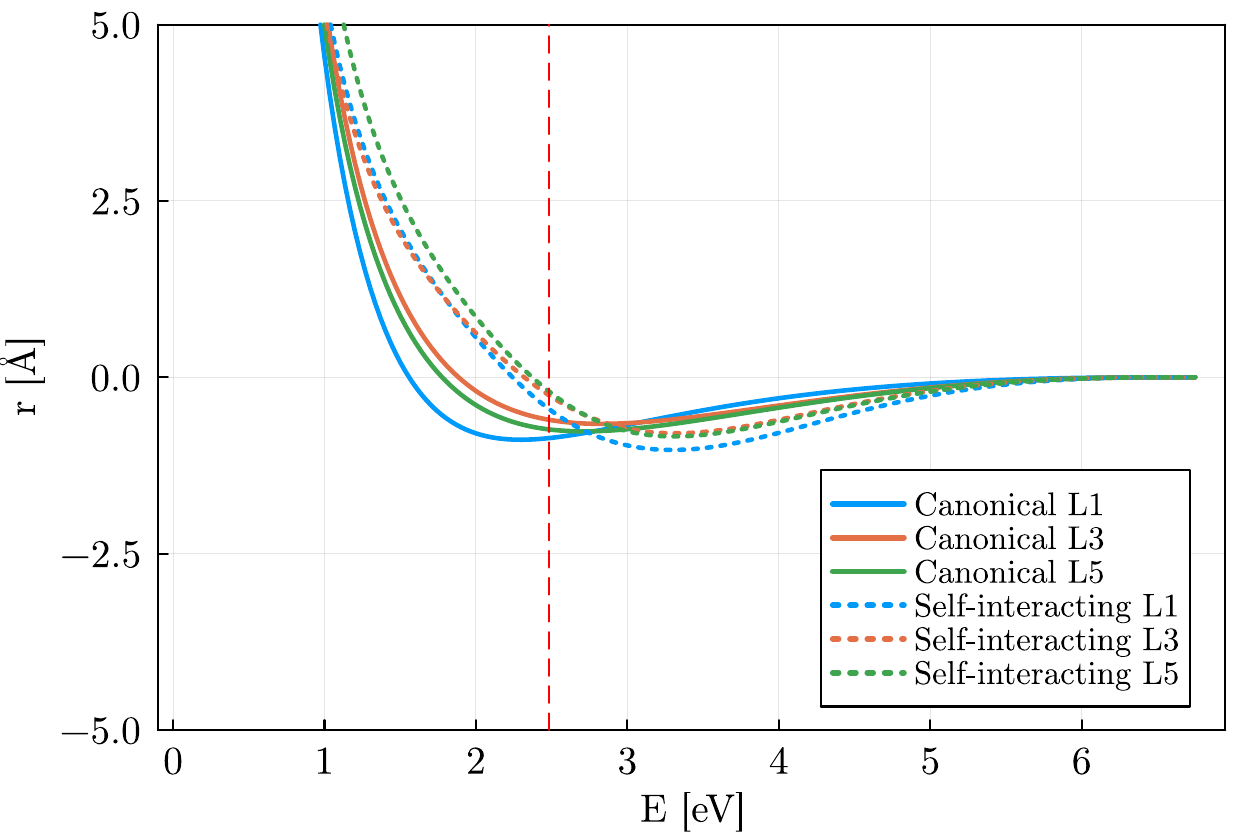}
        \caption{Relative tolerance: $5 \times 10^{-8}$}
     \end{subfigure}
      \begin{subfigure}[b]{0.33\textwidth}
        \centering
        \includegraphics[width=\textwidth]{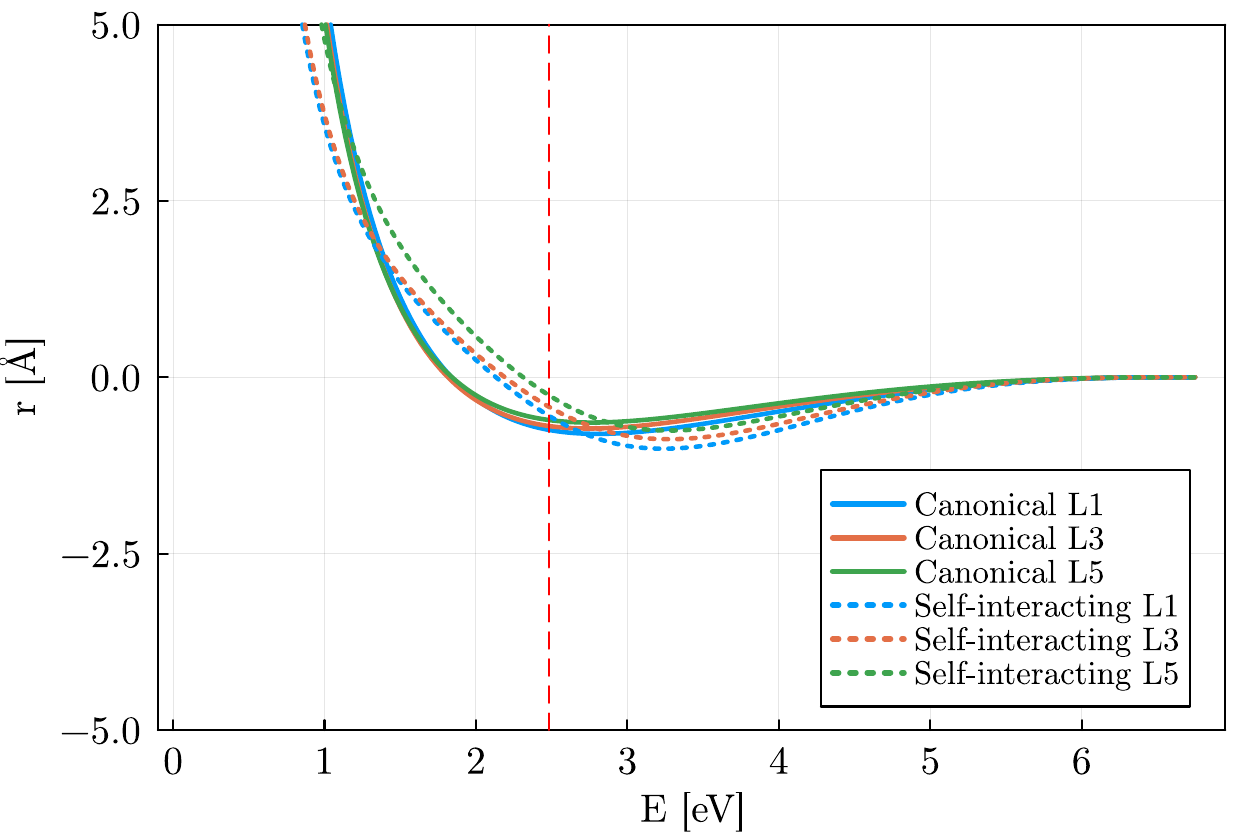}
        \caption{Relative tolerance: $10^{-7}$}
     \end{subfigure}
      \caption{Dimer curves for the Fe data set with the indicated relative tolerances in the truncated SVD. The red vertical line indicates the nearest neighbour distance of the Fe ground state crystal.}
      \label{fig:Fe_dimer}
\end{figure}

\section{Conclusion}
We developed a framework to transform the computationally efficient  {\em self-interacting cluster expansion} into the physically motivated {\em canonical cluster expansion}. 
We provided details how to generate the transform in practice as well as an in-depth analysis of the cost of applying the transform and of the benefits of using the canonical cluster expansion instead of the self-interacting cluster expansion in a range of academic and realistic examples. 

Our numerical evidence suggests that typical smoothness priors perform better when applied to the canonical cluster expansion and it is also less sensitive to changes in the regularization parameter, allowing more efficient hyper parameter tuning. For regression of certain ill-conditioned {\em symmetric functions}, we observe notable improvements in prediction accuracy over the self-interacting expansion. For regression of {\em multi-set functions} the accuracy improvements are relatively minor or even negligible. This is explained by the condition number tests in Section~\ref{section:ConditionNumber}.
These observations empirically support the considerable success of the existing ACE framework, which employs the self-interacting cluster expansion and which is normally applied to the regression of set functions. 

On the other hand, we did observe that the canonical cluster expansion can lead to improved qualitative properties of the regressed target function. Our overall conclusion is therefore, that the canonical cluster expansion is a worthwhile drop-in replacement for the self-interacting cluster expansion in a wide variety of modeling tasks. Our results suggest that it will {\em generally} lead to improved numerical robustness and sometimes also improved prediction accuracy. 
Further insight beyond the scope of this paper can be gained by more stringent and rigorous testing in specific application domains, in particular analyzing the effect of different regularization and regression techniques.

\section*{Acknowledgements}
This work was supported by NSERC Discovery Grant GR019381 and NFRF Exploration Grant GR022937. 
TSG was also supported by a PIMS-Simons postdoctoral fellowship, jointly funded by the Pacific Institute for the Mathematical Sciences (PIMS) and the Simons Foundation. 

\appendix
\section*{Appendix}
\section{The purification operator and conservative properties}
\label{Appendix:RecursionFormula}

\begin{lemma}
\label{Lemma:RecursionFormula}
 Let $\{\phi_k\}_{k \in \mathcal{I}}$ be a countable family of linearly independent one-particle basis functions i.e., $\phi_k : \Omega \to \bbF$. Suppose that any pointwise product of $\phi_{k_1}(\bx)\phi_{k_2}(\bx)$ can be linearized exactly in terms of a finite sum of $\phi_{\kappa}(\bx)$ \eqref{eq:linearize_product}. Let $\mathcal{A}_{\mathbf{k}}$ and $\mathbf{A}_{\mathbf{k}}$ be defined as in \eqref{equa:PureBasis} and \eqref{equa:ImpureBasis}. Then for every $\mathbf{k} = (k_1, k_2, \dots, k_{N})$ and $k_{N+1}$, there exist coefficients $\mathcal{P}^{k_{\beta}k_{N+1}}_{\kappa}$ such that 
\begin{equation}
    \label{equa:Purify}
    \mathcal{A}_{(\mathbf{k}, k_{N+1})} = \mathcal{A}_{\mathbf{k}}A_{k_{N+1}} - \sum\limits^N_{\beta = 1}\sum_{\kappa} \mathcal{P}^{k_{\beta}k_{N+1}}_{\kappa}\mathcal{A}_{\mathbf{k\left[\beta, \kappa\right]}},
\end{equation}
where $\mathbf{k}\left[\beta, \kappa\right] := (k'_1, k'_2, \dots, k'_N)$ with:
$$
k'_\alpha =
\begin{cases}
k_\alpha,  \quad \alpha \neq \beta \\
\kappa, \quad \hspace{2mm} \alpha = \beta
\end{cases}
.
$$
In particular, $\mathcal{P}^{k_{\beta} k_{N+1}}_{\kappa}$ are the linearization coefficients appearing in the linearization of products of the form $\phi_{k_{\beta}} \phi_{k_{N+1}}$.
\end{lemma}
\begin{proof}
By assumption, for each $\beta = 1, \cdots, N$, there exist finitely many $\mathcal{P}^{k_{\beta}k_{N+1}}_{\kappa}$ such that
\begin{equation}
\label{equa:linearize1p_2}
 \sum_{\kappa} \mathcal{P}^{k_{\beta}k_{N+1}}_{\kappa} \phi_{\kappa}(\mathbf{x}) = \phi_{k_\beta}(\mathbf{x})\phi_{k_{N+1}}(\mathbf{x}),
\end{equation}
for all $\mathbf{x} \in \Omega$. We first notice that $\mathcal{A}_{(\mathbf{k}, k_{N+1})} - \mathcal{A}_{\mathbf{k}}A_{k_{N+1}}$ is exactly the sum of all self-interacting terms between $\phi_{k_{N+1}}$ and $\phi_{k_\beta}$ for $\beta = 1, \dots, N$. More precisely, the self-interacting term between $\phi_{k_{N+1}}$ and $\phi_{k_\beta}$ reads
\[
\sum\limits_{j_1 \neq \dots \neq j_N} \phi_{k_1}(\mathbf{x}_{j_1}) \dots \phi_{k_\beta}(\mathbf{x}_{j_\beta}) \dots \phi_{k_N}(\mathbf{x}_{j_N})\phi_{k_{N+1}}(\mathbf{x}_{j_{\beta}})
\]
for each pair $(k_\beta, k_{N+1})$. Now for each $\beta$, we have
\begin{align*}
    \sum_{\kappa} \mathcal{P}^{k_{\beta}k_{N+1}}_{\kappa}\mathcal{A}_{\mathbf{k[\beta, \kappa]}}
   & = \sum_{\kappa} \mathcal{P}^{k_{\beta}k_{N+1}}_{\kappa}\sum\limits_{j_1 \neq \dots \neq j_N} \phi_{k_1}(\mathbf{x}_{j_1}) \dots \phi_{\kappa}(\mathbf{x}_{j_{\beta}}) \dots \phi_{k_N}(\mathbf{x}_{j_N})\\
   & = \sum\limits_{j_1 \neq \dots \neq j_N} \phi_{k_1}(\mathbf{x}_{j_1}) \dots \big(\sum_{\kappa} \mathcal{P}^{k_{\beta}k_{N+1}}_{\kappa} \phi_{\kappa}(\mathbf{x}_{j_{\beta}})\big) \dots \phi_{k_N}(\mathbf{x}_{j_N})\\
   & = \sum\limits_{j_1 \neq \dots \neq j_N} \phi_{k_1}(\mathbf{x}_{j_1}) \dots \phi_{k_\beta}(\mathbf{x}_{j_\beta}) \dots \phi_{k_N}(\mathbf{x}_{j_N})\phi_{k_{N+1}}(\mathbf{x}_{j_{\beta}}). \qquad [\text{ by } (\ref{equa:linearize1p_2})]
\end{align*}
From the above calculation we can see that $\sum_{\kappa} \mathcal{P}^{k_{\beta}k_{N+1}}_{\kappa}\mathcal{A}_{\mathbf{k\left[\beta, \kappa\right]}}$ is exactly the self-interaction corresponding to bases $\phi_{k_\beta}$ and $\phi_{{k}_{N+1}}$ which is left from $\mathcal{A}_{(\mathbf{k}, k_{N+1})} - \mathcal{A}_{\mathbf{k}}A_{k_{N+1}}$. Therefore, by summing over $\beta$ for each $\beta = 1, 2, \dots. N$, we have:
$$\mathcal{A}_{(\mathbf{k}, k_{N+1})} = \mathcal{A}_{\mathbf{k}}A_{k_{N+1}} - \sum\limits^N_{\beta = 1}\sum_{\kappa} \mathcal{P}^{k_{\beta}k_{N+1}}_{\kappa}\mathcal{A}_{\mathbf{k[\beta, \kappa]}},$$
which completes the proof.
\end{proof}

Following from the above lemma, we now prove Theorem \ref{theorem:spanKK'}.

\begin{proof}[Proof of Theorem \ref{theorem:spanKK'}]
    For each $\bk \in \mathbf{K}$, by recursively expanding \eqref{equa:Purify}, one can readily observe that each $\mathcal{A}_{\bk}$ can be expressed as linear combination of the $\mathbf{A}$-basis \eqref{eq:linearizek'}. The result follows from collecting all non-zero indices from the recursive expansion to form $\mathbf{K}'$.
\end{proof}

\section{Computational efficiency of the purification transform}
\label{Appendix:Efficiency}

\begin{proof}[Proof of Proposition \ref{proposition:preserving}]
The proof proceeds by induction, with the $N=2$ case being obtained via \eqref{eq:PurifyN=2}. Assume the statement holds for order $N$, i.e. that for each $\bk$ of order $N$ the expansion
$$
\mathcal{A}_{\bk} = \sum_{\bk'} \mathcal{P}^{\bk}_{\bk'}\mathbf{A}
$$
is total-degree preserving. Then, for order $N + 1$, we have
$$
\begin{aligned}
\mathcal{A}_{(\mathbf{k}, k_{N+1})} 
& = \mathcal{A}_{\mathbf{k}}A_{k_{N+1}} - \sum\limits^N_{\beta = 1}\sum_{\kappa} \mathcal{P}^{k_{\beta}k_{N+1}}_{\kappa}\mathcal{A}_{\mathbf{k\left[\beta, \kappa\right]}} \\
& = \sum_{\bk'} \mathcal{P}_{\bk'} \mathbf{A}_{\bk'} A_{k_{N+1}} - \sum\limits^N_{\beta = 1}\sum_{\kappa} \mathcal{P}^{k_{\beta}k_{N+1}}_{\kappa}\mathcal{A}_{\mathbf{k\left[\beta, \kappa\right]}} \\
& = \sum_{\bk'} \mathcal{P}_{\bk'} \mathbf{A}_{(\bk', k_{N+1})} - \sum\limits^N_{\beta = 1}\sum_{\kappa} \mathcal{P}^{k_{\beta}k_{N+1}}_{\kappa}\mathcal{A}_{\mathbf{k\left[\beta, \kappa\right]}}. \\
\end{aligned}
$$
In the above, the indices of every term in the first summations sums up to degree $\leq \sum \bk + k_{N+1}$ by the induction hypothesis. The second term is a sum of the $\mathcal{A}$ basis with indices $\sum \mathbf{k}\left[\beta, \kappa \right]$, satisfying assumption \eqref{equa:prodspan_totdeg} and the induction hypothesis. 
\end{proof}

\begin{proof}[Proof of Proposition \ref{proposition:bound}]
It is straightforward to observe that $M_{N} \leq M_{N-1} + K(N-1)M_{N-1} = (K(N-1)+1)M_{N-1}$ from (\ref{equa:Purify}). The result follows from applying this recursively.
\end{proof}

\begin{exmp}
\label{example:cancellationCoeffs}
\textup{To observe additional cancellations in the coefficients, we use an $\mathcal{A}_{k_1 k_2 k_3}$ formed from Chebyshev polynomials as an example. Recall that the product of Chebyshev polynomials $T_m$ and $T_n$ can be written as a weighted sum of $T_{m+n}$ and $T_{m-n}$ \cite{giorgi2011polynomial_prodcheb}. Directly applying \eqref{equa:Purify} gives}
\begin{align*}
\mathcal{A}_{k_1 k_2 k_3} &= \mathcal{A}_{k_1 k_2}A_{k_3} - \sum_{\kappa} \mathcal{P}^{k_1 k_3}_{\kappa}\mathcal{A}_{\kappa k_2} - \sum_{\kappa} \mathcal{P}^{k_2 k_3}_{\kappa}\mathcal{A}_{k_1 \kappa} \\
& = (A_{k_1 k_2} - \mathcal{P}^{k_1 k_2}_{k_1 + k_2} A_{k_1 + k_2} - \mathcal{P}^{k_1 k_2}_{k_2 - k_1} A_{k_2 - k_1})A_{k_3} - \sum_{\kappa} \mathcal{P}^{k_1 k_3}_{\kappa}\mathcal{A}_{\kappa k_2} - \sum_{\kappa} \mathcal{P}^{k_2 k_3}_{\kappa}\mathcal{A}_{k_1 \kappa}    
\end{align*}
\textup{The part without $\kappa$-summation contributes the following non-zero indices to $C$: $(k_1, k_2, k_3), (k_1 + k_2, k_3), (k_2 - k_1, k_3)$. The remaining two summations each contribute six terms:}
$$
\begin{cases}
(k_1 + k_3, k_2), (k_1 + k_2 + k_3), (k_1 + k_3 - k_2) \\
(k_3 - k_1, k_2), (k_3 - k_1 + k_2), (|k_3 - k_1 - k_2|) \\
\end{cases}
\begin{cases}
(k_2 + k_3, k_1), (k_1 + k_2 + k_3), (k_2 + k_3 - k_1) \\
(k_3 - k_2, k_1), (k_3 - k_2 + k_1), (|k_3 - k_2 - k_1|) \\
\end{cases}
$$
\textup{Notice that there are duplications in indices and there are thus only at most $11$ non-zero elements (which happens when $k_i \neq k_j$ for all $i \neq j$) in total for the purification of an order 3 basis using Chebyshev polynomials. This explains the sub-asymptotic behavior in Figure \ref{fig:ChebBound}.}
\end{exmp}

\begin{proof}[Proof of Proposition \ref{Proposition:parallelizable}]
For clarity we write the index as $(\bk, k_{N+1})$ and show that purification of an order $N+1$ basis only depends on $\mathbf{A}_{\bk'}$ basis functions of order $N$. By \eqref{equa:Purify}
    \begin{align*}
   \mathcal{A}_{(\mathbf{k}, k_{N+1})} & = \mathcal{A}_{\mathbf{k}}A_{k_{N+1}} - \sum\limits^N_{\beta = 1}\sum_{\kappa} \mathcal{P}^{k_{\beta}k_{N+1}}_{\kappa}\mathcal{A}_{\mathbf{k\left[\beta, \kappa\right]}}\\
   & = \left(\mathcal{A}_{\mathbf{k-1}}A_{k_{N}} - \sum\limits^{N-1}_{\beta = 1}\sum_{\kappa} \mathcal{P}^{k_{\beta}k_{N}}_{\kappa}\mathcal{A}_{\mathbf{k-1\left[\beta, \kappa\right]}}\right) A_{k_{N+1}} - \sum\limits^N_{\beta = 1}\sum_{\kappa} \mathcal{P}^{k_{\beta}k_{N+1}}_{\kappa}\mathcal{A}_{\mathbf{k\left[\beta, \kappa\right]}}\\
   & = \mathcal{A}_{\mathbf{k-1}}A_{k_{N} k_{N+1}} - \underbrace{\sum\limits^{N-1}_{\beta = 1}\sum_{\kappa} \mathcal{P}^{k_{\beta}k_{N}}_{\kappa}\mathcal{A}_{\mathbf{k-1\left[\beta, \kappa\right]}}A_{k_{N+1}} - \sum\limits^N_{\beta = 1}\sum_{\kappa} \mathcal{P}^{k_{\beta}k_{N+1}}_{\kappa}\mathcal{A}_{\mathbf{k\left[\beta, \kappa\right]}}}_{\leq \text{order $N$}}\\
   & = A_{k_1 \cdots k_{N+1}} - \text{L.O.T}.
\end{align*}
where L.O.T denotes terms of order $\leq N$ and the last line follows from recursively expanding all $\mathcal{A}$ terms as in the second equality.
\end{proof}

\section{The ACE model for MLIPs}
\subsection{Model construction}
\label{Appendix:ModelConstruction}
The ACE model applied throughout our experiments in Section~\ref{section:MLIPs} slightly differs from simply the $O(3)$ invariant basis in Section~\ref{section:G=O(3)} with an additional pair potential and is implemented as {\tt acemodel} in {\tt ACEpotentials.jl} \cite{2023-acepotentials}. More precisely, {\tt acemodel} is a sum of a pair (order $= 1$) and many-body $\mathcal{B}$ or $\mathbf{B}$-basis of maximum order $\mathcal{N}$, where $\mathcal{N}$ corresponding to the maximum correlation order of the potential terms, determined by the user, as in \eqref{equa:CanonicalExpansion} and throughout the paper. A cutoff $r_{\rm cut} > 0$ is specified to enforce decay of pair as $r \rightarrow r_{\rm cut}$ and of many body terms as $r \rightarrow 0$ and $r \rightarrow r_{\rm cut}$. The default choice for the many-body basis is
\begin{equation}
\label{equa:default_env}
f_{\rm env}(r_{ij}) = y^2 (y - y_{\rm cut})^2    
\end{equation}
where $y_{\rm cut} = y(r_{\rm cut})$ and $y$ is a coordinate transformation as in \eqref{equa:MLIP_1pbasis}. The default choice of envelope for the paired basis is a modified Coulomb potential to ensure a smooth cut off as $r \rightarrow r_{\rm cut}$ and repulsive behavior as $r \rightarrow 0$,
$$
f_{\rm env}(r_{ij}) = \Big(\frac{r_{ij}}{r_0}\Big)^{-1} - \Big(\frac{r_{\rm cut}}{r_0}\Big)^{-1} + \Big(\frac{r_{\rm cut}}{r_0}\Big)^{-2}\Big(\frac{r_{ij}}{r_0} - \frac{r_{\rm cut}}{r_0}\Big).
$$
Here, $r_0$ is an estimate of the equilibrium bond-length, which is system-dependent. We also note here that the choice of coordinate transformation $y$ is independent of the purification framework and beyond the focus of this work. We refer to \cite{2023-acepotentials} and the documentation of {\tt ACEpotentials.jl} \cite{ACEpotentials_sw} for more details.

\subsection{The effect of envelope functions}
\label{Appendix:ExtraBasis}
The incorporation of envelope functions results in the basis being no longer total degree deserving. Table \ref{table:ExtraBasis} shows an example on the number of extra basis elements required for the different sparsifications. For flexibility, an additional type sparsification, which generalize the idea of total degree, is commonly applied. One specifies a tuple of total degree $D = (D_1, \cdots, D_\mathcal{N})$, which means that all the basis elements of correlation order $N$ has a total degree $\leq D_N$ for $1 \leq N \leq \mathcal{N}$. For clarity, the table contains an extra column for the total degree in {\tt ACEpotentials.jl}, denoted by $D^{\text{ACEpot}}$, which starts with $1$ instead of $0$. The actual total degree of order $N$ (regardless of envelope function) in another column of the table is calculated by $D_N = D_N^{\text{ACEpot}}-N$. An envelope function of degree = $4$ (as in \eqref{equa:default_env}) is used and the actual total degree with envelope function $D^{\text{env}}$ is also computed in a separate column.

We make three key observations from Table \ref{table:ExtraBasis}:
\begin{enumerate}[(i)]
    \item The analysis of the two conditions: $\sum \bm = 0$ and $\sum \bl = \text{even}$ on the angular part (i.e. spherical harmonics) of the basis in Section~\ref{section:G=O(3)} is not affected by the radial envelope function; 
    \item With the introduction of the envelope function, Proposition \ref{proposition:preserving} cannot be applied directly with basis sparsification as in Section~\ref{section:G=O(3)} and extra basis elements supplementing insufficiency in linearization of the radial part of $\phi_{n_1 l_1 m_1}\phi_{n_2 l_2 m_2}$ have to be evaluated;
    \item In some cases, if the difference in total degree for each order is large enough, the two bases nevertheless have the same span since the lower order term is already sufficient for removing the self-interacting terms without extra basis. This corresponds to the cases where number of extra $\mathbf{A}_{\bk}$ basis is zero in Table \ref{table:ExtraBasis}.
\end{enumerate}

\begin{table}[h!]
\centering
\begin{tabular}{ p{2.0cm} p{2.0cm} p{2.0cm} p{1.25cm} p{1.55cm} p{1.25cm} p{1.6cm} }
 \hline
$D^{\text{ACEpot}}$ & $D$ & $D^{\text{env}}$ & \# of $\mathbf{A}_{\bk}$ basis & \# of extra $\mathbf{A}_{\bk}$ basis  & \# of $\mathcal{B}_{\alpha}$ and $\mathbf{B}_{\alpha}$ basis & Equivalent? \\
 \hline
 22, 18, 14, 10 & 21, 16, 11, 6 & 25, 24, 23, 22 & 1709 & 0 & 539 & Yes \\
 24, 20, 16, 12 & 23, 18, 13, 8 & 27, 26, 25, 24 & 3704 & 0 & 929 & Yes \\
 26, 22, 18, 14 & 25, 20, 15, 10 & 29, 28, 27, 26 & 8034 & 0 & 1608 & Yes \\
 \hline
 18, 15, 12, 9 & 17, 13, 9, 5 & 21, 21, 21, 21 & 736 & 0 & 299 & Yes\\
 20, 17, 14, 11 & 19, 15, 11, 7 & 23, 23, 23, 23 & 1669 & 0 & 538 & Yes\\
 22, 19, 16, 13 & 21, 17, 13, 9 & 25, 25, 25, 25 & 3842 & 0 & 969 & Yes\\
 \hline
 18, 16, 14, 12 & 17, 14, 11, 8 & 21, 22, 23, 24 & 1718 & 362 & 559 & No\\
 20, 18, 16, 14 & 19, 16, 13, 10 & 23, 24, 25, 26 & 4226 & 886 & 1056 & No\\
 22, 20, 18, 16 & 21, 18, 15, 12 & 25, 26, 27, 28 & 10363 & 1771 & 1969 & No\\
 \hline
 17, 16, 15, 14 & 16, 14, 12, 10 & 20, 22, 24, 26 & 3168 & 1383 & 862 & No\\
 19, 18, 17, 16 & 18, 16, 14, 12 & 22, 24, 26, 28 & 8245 & 2907 & 1677 & No\\
 21, 20, 19, 18 & 20, 18, 16, 14 & 24, 26, 28, 30 & 20845 & 5855 & 3197 & No\\

\end{tabular}
    \caption{In the table, $D^{\text{ACEpot}}$ is the degree used in {\tt ACEpotentials.jl} \cite{2023-acepotentials}, $D$ is calculated by $D_N = D_N^{\text{ACEpot}}-N$ and $D^{\text{env}}$ is calculated by $D^{\text{env}}_N = D_N+4N$, where $4$ is the polynomial degree of the envelope function $f_{\rm env}$ in \eqref{equa:MLIP_1pbasis}. The equivalence of the two bases sets is verified by checking numerically whether every $\mathcal{B}_{\bn \bl i}$ basis is in ${\rm span}(\{\mathbf{B_{\bn \bl i}}\}_{\bn \bl i})$, and vice versa.}
    \label{table:ExtraBasis}
\end{table}

\section{Hyperparameters for our numerical experiments}
\label{Appendix:hyperparameters}
\subsection{Zuo data set}
\label{Appendix:hyperparameters_zuo}
{\tt acemodel} with total degree = (25, 21, 17, 13) (See \ref{Appendix:ExtraBasis} for the definition of specifying total degree as a tuple) and default weights for the least squares system in {\tt ACEpotentials.jl} \cite{2023-acepotentials} is applied. Only energy and force observations  are used to estimate the parameters. An algebraic smoothness prior with $p = 5$ is used. All other hyper parameters are kept as default.

\subsection{Fe data set}
\label{Appendix:hyperparameters_Fe}
{\tt acemodel} with total degree as specified in Table \ref{table:ExtraBasis} and $r_{\rm cut} = 6.5 \text{\AA}$ is applied. The energy of the one-body term is set to be $\text{Eref} = -3455.6995339$ following the GAP model in \cite{zhang2020Fe}. We specify the weights of the least squares system depending on the configuration type following \cite{zhang2020Fe}. All energy, force and viral observations are used to estimate the parameters. A algebraic smoothness prior of $p = 5$ is used. All other hyper parameters are kept as default.

\begin{table}[]
    \centering
    \begin{tabular}{C{1in}|C{1in}|C{1in}}
        \hline
     \textbf{Levels} & \textbf{Total degree} & \textbf{Basis size } \\\hline
        1 & 18, 14, 10, 6 & 183\\
        2 & 20, 16, 12, 8 & 314\\
        3 & 22, 18, 14, 10 & 539\\
        4 & 24, 20, 16, 12 & 929\\
        5 & 26, 22, 18, 14 & 1609 \\
    \end{tabular}
    \caption{Model levels in Section~\ref{section:Fedata} and corresponding total degree in {\tt ACEpotentials.jl}\cite{2023-acepotentials}. See \ref{Appendix:ExtraBasis} for the definition of specifying total degree as a tuple.}
    \label{table:modelsize}
\end{table}

\newpage

\bibliographystyle{plain}
\bibliography{references.bib}

\end{document}

%% file: notations.tex
\newcommand{\cchho}[1]{{\color{purple} \footnotesize   \tt [chho: #1]}}
\newcommand{\chho}[1]{{\color{purple} #1}}\newcommand{\tsgut}[1]{{\color{green} #1}}
\newcommand\chhosout{\bgroup\markoverwith{\textcolor{purple}{\rule[0.5ex]{2pt}{0.4pt}}}\ULon}

\newcommand{\br}{\mathbf{r}}
\newcommand{\bn}{\mathbf{n}}
\newcommand{\bl}{\mathbf{l}}
\newcommand{\bm}{\mathbf{m}}
\newcommand{\bp}{\mathbf{p}}
\newcommand{\bx}{\mathbf{x}}
\newcommand{\R}{\mathbb{R}}
\newcommand{\tdeg}{\text{deg}}
\newcommand{\Deg}{\text{Deg}}
\newcommand{\tord}{\text{ord}}
\newcommand{\even}{\text{even}}

\newcommand{\purifymat}{\mathcal{P}}

\newtheorem{theorem}{Theorem}[section]
\newtheorem{corollary}[theorem]{Corollary}
\newtheorem{lemma}[theorem]{Lemma}
\newtheorem{proposition}[theorem]{Proposition}
\newtheorem{remark}[theorem]{Remark}
\newtheorem{exmp}{Example}[section]
\algnewcommand{\LineComment}[1]{\State \(\triangleright\) #1}

%% file: main.bbl
\begin{thebibliography}{10}

\bibitem{ACE1x}
{\tt ACE1x.jl}.
\newblock Experimental features for {{\tt ACE}}{\tt potentials.jl}.
\newblock {\tt github.com/ACEsuit/ACE1x.jl}.

\bibitem{ACEpotentials_sw}
{\tt ACEpotentials.jl}.
\newblock Documentation and user interface for {Julia}-language development of
  {ACE} potentials.
\newblock {\tt github.com/ACEsuit/ACEpotentials.jl}.

\bibitem{bachmayr2023polynomial}
Markus Bachmayr, Genevieve Dusson, Christoph Ortner, and Jack Thomas.
\newblock Polynomial approximation of symmetric functions.
\newblock {\em arXiv:2109.14771}, 2023. to appear in Math. Comput.

\bibitem{batatia2022e3design}
Ilyes Batatia, Simon Batzner, D{\'a}vid~P{\'e}ter Kov{\'a}cs, Albert Musaelian,
  Gregor~NC Simm, Ralf Drautz, Christoph Ortner, Boris Kozinsky, and G{\'a}bor
  Cs{\'a}nyi.
\newblock The design space of {E}(3)-equivariant atom-centered interatomic
  potentials.
\newblock {\em arXiv:2205.06643}, 2022. to appear in Nature Machine
  Intelligence.

\bibitem{batatia2023general}
Ilyes Batatia, Mario Geiger, Jose Munoz, Tess Smidt, Lior Silberman, and
  Christoph Ortner.
\newblock A general framework for equivariant neural networks on reductive
  {Lie} groups.
\newblock {\em arXiv:2306.00091}, 2023. to appear in Advances in Neural
  Information Processing Systems.

\bibitem{Batatia2022mace}
Ilyes Batatia, David~P Kovacs, Gregor Simm, Christoph Ortner, and Gabor
  Cs{\'a}nyi.
\newblock Mace: Higher order equivariant message passing neural networks for
  fast and accurate force fields.
\newblock In S.~Koyejo, S.~Mohamed, A.~Agarwal, D.~Belgrave, K.~Cho, and A.~Oh,
  editors, {\em Adv. Neural Inf. Process. Syst.}, volume~35, pages
  11423--11436. Curran Associates, Inc., 2022.

\bibitem{batatia2023equivariantmat}
Ilyes Batatia, Lars~L Schaaf, Huajie Chen, G{\'a}bor Cs{\'a}nyi, Christoph
  Ortner, and Felix~A Faber.
\newblock Equivariant matrix function neural networks.
\newblock {\em arXiv:2310.10434}, 2023.

\bibitem{mlinvariant}
Ben Blum-Smith and Soledad Villar.
\newblock Machine learning and invariant theory.
\newblock {\em arXiv:2209.14991}, 2023.

\bibitem{braams2009permutationally}
Bastiaan~J Braams and Joel~M Bowman.
\newblock Permutationally invariant potential energy surfaces in high
  dimensionality.
\newblock {\em Int. Rev. Phys. Chem.}, 28(4):577--606, 2009.

\bibitem{byerly1893elemenatary}
William~Elwood Byerly.
\newblock {\em An elemenatary treatise on Fourier's series, and spherical,
  cylindrical, and ellipsoidal harmonics, with applications to problems in
  mathematical physics}.
\newblock Dover Publications, 1893.

\bibitem{chen2016qm}
Huajie Chen and Christoph Ortner.
\newblock Qm/mm methods for crystalline defects. part 1: Locality of the tight
  binding model.
\newblock {\em Multiscale Model. Sim.}, 14(1):232--264, 2016.

\bibitem{cohen2013stability}
Albert Cohen, Mark~A Davenport, and Dany Leviatan.
\newblock On the stability and accuracy of least squares approximations.
\newblock {\em Found. Comput. Math.}, 13:819--834, 2013.

\bibitem{deringer2019machine}
Volker~L Deringer, Miguel~A Caro, and G{\'a}bor Cs{\'a}nyi.
\newblock Machine learning interatomic potentials as emerging tools for
  materials science.
\newblock {\em Advanced Materials}, 31(46):1902765, 2019.

\bibitem{DrautzACE}
Ralf Drautz.
\newblock Atomic cluster expansion for accurate and transferable interatomic
  potentials.
\newblock {\em Phys. Rev. B}, 99:014104, Jan 2019.

\bibitem{drautzwavefunc}
Ralf Drautz and Christoph Ortner.
\newblock Atomic cluster expansion and wave function representations.
\newblock {\em arXiv:2206.11375}, 2022.

\bibitem{ACECompleteness}
Genevi{\`e}ve Dusson, Markus Bachmayr, G{\'a}bor Cs{\'a}nyi, Ralf Drautz, Simon
  Etter, Cas van~der Oord, and Christoph Ortner.
\newblock Atomic cluster expansion: Completeness, efficiency and stability.
\newblock {\em J. Comput. Phys.}, 454:110946, 2022.

\bibitem{LinearizeProductJacobi}
George Gasper.
\newblock Linearization of the product of {Jacobi polynomials}. {I}.
\newblock {\em Canadian J. Math.}, 22(1):171--175, 1970.

\bibitem{gerken2021geometric}
Jan~E Gerken, Jimmy Aronsson, Oscar Carlsson, Hampus Linander, Fredrik Ohlsson,
  Christoffer Petersson, and Daniel Persson.
\newblock Geometric deep learning and equivariant neural networks.
\newblock {\em Artif. Intell. Rev.}, pages 1--58, 2023.

\bibitem{giorgi2011polynomial_prodcheb}
Pascal Giorgi.
\newblock On polynomial multiplication in {Chebyshev} basis.
\newblock {\em IEEE Trans. Comput.}, 61(6):780--789, 2011.

\bibitem{lietransformer}
Michael~J Hutchinson, Charline Le~Lan, Sheheryar Zaidi, Emilien Dupont,
  Yee~Whye Teh, and Hyunjik Kim.
\newblock Lietransformer: Equivariant self-attention for {L}ie groups.
\newblock In {\em International Conference on Machine Learning}, pages
  4533--4543. PMLR, 2021.

\bibitem{MACEFF}
D{\'a}vid~P{\'e}ter Kov{\'a}cs, Ilyes Batatia, Eszter~Sara Arany, and Gabor
  Csanyi.
\newblock Evaluation of the {MACE} force field architecture: from medicinal
  chemistry to materials science.
\newblock {\em J. Chem. Phys.}, 159(4), jul 2023.

\bibitem{LinearACE4OrgMol}
D{\'a}vid~P{\'e}ter Kov{\'a}cs, Cas van~der Oord, Jiri Kucera, Alice~EA Allen,
  Daniel~J Cole, Christoph Ortner, and G{\'a}bor Cs{\'a}nyi.
\newblock Linear atomic cluster expansion force fields for organic molecules:
  beyond rmse.
\newblock {\em J. Chem. Theory Comput.}, 17(12):7696--7711, 2021.

\bibitem{lee2019set}
Juho Lee, Yoonho Lee, Jungtaek Kim, Adam Kosiorek, Seungjin Choi, and Yee~Whye
  Teh.
\newblock Set transformer: A framework for attention-based
  permutation-invariant neural networks.
\newblock In {\em International Conference on Machine Learning}, pages
  3744--3753. PMLR, 2019.

\bibitem{Li2022DLDFT}
He~Li, Zun Wang, Nianlong Zou, Meng Ye, Runzhang Xu, Xiaoxun Gong, Wenhui Duan,
  and Yong Xu.
\newblock Deep-learning density functional theory hamiltonian for efficient ab
  initio electronic-structure calculation.
\newblock {\em Nat. Comput. Sci.}, 2(6):367--377, 2022.

\bibitem{PACE}
Yury Lysogorskiy, Cas van~der Oord, Anton Bochkarev, Sarath Menon, Matteo
  Rinaldi, Thomas Hammerschmidt, Matous Mrovec, Aidan Thompson, G{\'a}bor
  Cs{\'a}nyi, Christoph Ortner, et~al.
\newblock Performant implementation of the atomic cluster expansion ({PACE})
  and application to copper and silicon.
\newblock {\em npj Comput. Mater.}, 7(1):97, 2021.

\bibitem{Munoz_2022}
Jose~M Munoz, Ilyes Batatia, and Christoph Ortner.
\newblock Boost invariant polynomials for efficient jet tagging.
\newblock {\em Mach. Learn.: Sci. Technol.}, 3(4):04LT05, 2022.

\bibitem{musil2021physics}
Felix Musil, Andrea Grisafi, Albert~P Bart{\'o}k, Christoph Ortner, G{\'a}bor
  Cs{\'a}nyi, and Michele Ceriotti.
\newblock Physics-inspired structural representations for molecules and
  materials.
\newblock {\em Chemical Reviews}, 121(16):9759--9815, 2021.

\bibitem{RecursiveEvalN-bodyequi}
Jigyasa Nigam, Sergey Pozdnyakov, and Michele Ceriotti.
\newblock Recursive evaluation and iterative contraction of n-body equivariant
  features.
\newblock {\em J. Chem. Phys.}, 153(12), 2020.

\bibitem{2019-tbloc0T}
Christoph Ortner, Jack Thomas, and Huajie Chen.
\newblock Locality of interatomic forces in tight binding models for
  insulators.
\newblock {\em ESAIM: Mathematical Modelling and Numerical Analysis},
  54(6):2295--2318, 2020.

\bibitem{drautzcarbon}
Minaam Qamar, Matous Mrovec, Yury Lysogorskiy, Anton Bochkarev, and Ralf
  Drautz.
\newblock Atomic cluster expansion for quantum-accurate large-scale simulations
  of carbon.
\newblock {\em J. Chem. Theory Comput.}, 19(15):5151--5167, 2023.

\bibitem{qi2017pointnet}
Charles~Ruizhongtai Qi, Li~Yi, Hao Su, and Leonidas~J Guibas.
\newblock Pointnet++: Deep hierarchical feature learning on point sets in a
  metric space.
\newblock In I.~Guyon, U.~Von Luxburg, S.~Bengio, H.~Wallach, R.~Fergus,
  S.~Vishwanathan, and R.~Garnett, editors, {\em Adv. Neural Inf. Process.
  Syst.}, volume~30. Curran Associates, Inc., 2017.

\bibitem{thomas2018tensor}
Nathaniel Thomas, Tess Smidt, Steven Kearnes, Lusann Yang, Li~Li, Kai Kohlhoff,
  and Patrick Riley.
\newblock Tensor field networks: Rotation-and translation-equivariant neural
  networks for 3d point clouds.
\newblock {\em arXiv:1802.08219}, 2018.

\bibitem{lloydHypercube}
Lloyd~N. Trefethen.
\newblock Multivariate polynomial approximation in the hypercube.
\newblock {\em Proc. Amer. Math. Soc.}, 145(11):4837--4844, 2017.

\bibitem{reg_cas}
Cas van Der~Oord, Genevi{\`e}ve Dusson, G{\'a}bor Cs{\'a}nyi, and Christoph
  Ortner.
\newblock Regularised atomic body-ordered permutation-invariant polynomials for
  the construction of interatomic potentials.
\newblock {\em Mach. Learn.: Sci. Technol.}, 1(1):015004, 2020.

\bibitem{vanderoord2022HAL}
Cas van~der Oord, Matthias Sachs, D{\'a}vid~P{\'e}ter Kov{\'a}cs, Christoph
  Ortner, and G{\'a}bor Cs{\'a}nyi.
\newblock Hyperactive learning for data-driven interatomic potentials.
\newblock {\em npj Computational Materials}, 9(1):168, 2023.

\bibitem{vandermause2020flyal}
Jonathan Vandermause, Steven~B Torrisi, Simon Batzner, Yu~Xie, Lixin Sun,
  Alexie~M Kolpak, and Boris Kozinsky.
\newblock On-the-fly active learning of interpretable {Bayesian} force fields
  for atomistic rare events.
\newblock {\em npj Comput. Mater.}, 6(1):20, 2020.

\bibitem{wang2023theoretical}
Yangshuai Wang, Shashwat Patel, and Christoph Ortner.
\newblock A theoretical case study of the generalisation of machine-learned
  potentials.
\newblock {\em arXiv:2311.01664}, 2023.

\bibitem{2023-acepotentials}
William~C Witt, Cas van~der Oord, Elena Gel{\v{z}}inyt{\.e}, Teemu
  J{\"a}rvinen, Andres Ross, James~P Darby, Cheuk~Hin Ho, William~J Baldwin,
  Matthias Sachs, James Kermode, et~al.
\newblock Acepotentials. jl: A julia implementation of the atomic cluster
  expansion.
\newblock {\em J. Chem. Phys.}, 159(16), 2023.

\bibitem{yutsis1962mathematical}
Adolfas~P Yutsis, Ioshua~Beniaminovich Levinson, and Vladislavas~Vladovich
  Vanagas.
\newblock Mathematical apparatus of the theory of angular momentum.
\newblock {\em Academy of Sciences of the Lithuanian S.S.R.}, 1962.

\bibitem{zhang2020Fe}
Lei Zhang, G{\'a}bor Cs{\'a}nyi, Erik van~der Giessen, and Francesco Maresca.
\newblock Atomistic fracture in bcc iron revealed by active learning of
  gaussian approximation potential.
\newblock {\em npj Computational Materials}, 9(1):217, Dec 2023.

\bibitem{ACEHam}
Liwei Zhang, Berk Onat, Genevi{\`e}ve Dusson, Adam McSloy, Gautam Anand,
  Reinhard~J Maurer, Christoph Ortner, and James~R Kermode.
\newblock Equivariant analytical mapping of first principles {Hamiltonians} to
  accurate and transferable materials models.
\newblock {\em npj Comput. Mater.}, 8(1):158, 2022.

\bibitem{ACESchrodinger}
Dexuan Zhou, Huajie Chen, Cheuk~Hin Ho, and Christoph Ortner.
\newblock A multilevel method for many-electron {S}chr\"{o}dinger equations
  based on the atomic cluster expansion.
\newblock {\em arXiv:2304.04260}, 2023. to appear in SIAM J. Sci. Comput.

\bibitem{Zuo2020}
Yunxing Zuo, Chi Chen, Xiangguo Li, Zhi Deng, Yiming Chen, Jörg Behler, Gábor
  Csányi, Alexander~V. Shapeev, Aidan~P. Thompson, Mitchell~A. Wood, and
  Shyue~Ping Ong.
\newblock Performance and cost assessment of machine learning interatomic
  potentials.
\newblock {\em J. Phys. Chem. A}, 124(4):731--745, 2020.
\newblock PMID: 31916773.

\end{thebibliography}
